\documentclass[11pt]{amsart}

\usepackage[colorlinks = true,
            linkcolor = red,
            urlcolor  = magenta,
            citecolor = black,
            anchorcolor = blue]{hyperref}
\usepackage{color}
\usepackage[shortalphabetic,initials,nobysame]{amsrefs}
\usepackage{upgreek}
\usepackage{tikz}
\usepackage[applemac]{inputenc} 
\usetikzlibrary{arrows}
\usepackage{xcolor}
\usepackage[all]{xy}
\usepackage{graphicx}
\usepackage{wrapfig}
\usepackage{yfonts}
\usepackage{graphicx}
\usepackage{amssymb}
\usepackage{epstopdf}
\usepackage{mathrsfs}
\usepackage[mathcal]{eucal}
\usepackage{bm}

\renewcommand{\eprint}[1]{\href{https://arxiv.org/abs/#1}{#1}}

\DeclareMathOperator{\GL}{\mathrm{GL}}

\DeclareMathOperator{\SL}{\mathrm{SL}}

\DeclareMathOperator{\Hom}{Hom}

\DeclareMathOperator{\Lie}{Lie}

\newcommand{\hcL}{\hat{\cL}}

\BibSpec{article}{%
    +{}  {\PrintAuthors}                {author}
    +{,} { \textit}                     {title}
     +{,} {}                            {note}
    +{.} { }                            {part}
    +{:} { \textit}                     {subtitle}
    +{,} { \PrintContributions}         {contribution}
    +{.} { \PrintPartials}              {partial}
    +{,} { }                            {journal}
    +{}  { \textbf}                     {volume}
    +{}  { \PrintDate}                {date}
    +{,} { \issuetext}                  {number}
    +{,} { \eprintpages}                {pages}
    +{,} { }                            {status}
    +{,} { \DOI}                   {doi}
    +{,} { \eprint}        {eprint}
      +{,} {\publisher}                {publisher}
    +{,} { \address}                   {address}
    +{}  { \parenthesize}               {language}
    +{}  { \PrintTranslation}           {translation}
    +{;} { \PrintReprint}               {reprint}
    +{.} {}                             {transition}
    +{}  {\SentenceSpace \PrintReviews} {review}
}

\newtheorem{Thm}{Theorem}[section]
\newtheorem{Lem}[Thm]{Lemma}
\newtheorem{Prop}[Thm]{Proposition}
\newtheorem{Cor}[Thm]{Corollary}

\theoremstyle{definition}
\newtheorem{Def}[Thm]{Definition}
\theoremstyle{remark}
\newtheorem{Rem}[Thm]{Remark}

\newtheoremstyle{named}{}{}{\itshape}{}{\bfseries}{.}{.5em}{#1 #3}
\theoremstyle{named}

\def\C{\mathbb{C}}
\def\Z{\mathbb{Z}}

\def\P{\mathbb{P}}

\def\fb{\mathfrak{b}}
\def\g{\mathfrak{g}}
\def\Frenkel:2013uda{\mathfrak{h}}

\def\cD{\mathcal{D}}

\def\cF{\mathcal{F}}

\def\cL{\mathcal{L}}
\def\cM{\mathcal{M}}
\def\cN{\mathcal{N}}
\def\cO{\mathcal{O}}

\def\cQ{\mathcal{Q}}

\def\cV{\mathcal{V}}
\def\cW{\mathcal{W}}

\def\bo{\textbf{o}}

\def\=>{\Longrightarrow}

\def\to{\longrightarrow}

\def\o+{\oplus}
\def\bo+{\bigoplus}

\def\<{\langle}
\def\>{\rangle}
\def\({\left(}
\def\){\right)}

\def\^{\wedge}
\def\+{\dagger}

\def\dd[#1,#2]{\frac{d#1}{d#2}}
\def\del[#1,#2]{\frac{\partial #1}{\partial #2}}
\def\over[#1]{\overline{#1}}
\def\vec[#1]{\overrightarrow{#1}}

\makeatletter
\def\mr@ignsp#1 {\ifx\:#1\@empty\else #1\expandafter\mr@ignsp\fi}%
\newcommand{\multiref}[1]{\begingroup
\xdef\mr@no@sparg{\expandafter\mr@ignsp#1 \: }%
\def\mr@comma{}%
\@for\mr@refs:=\mr@no@sparg\do{\mr@comma\def\mr@comma{,}\ref{\mr@refs}}%
\endgroup}
\makeatother

\newcommand{\hypref}[2]{\ifx\href\asklFrenkel:2013udaas #2\else\href{#1}{#2}\fi}

\tikzset{->-/.style={decoration={
  markings,
  mark=at position .5 with {\arrow{latex}}},postaction={decorate}}}
\tikzset{
    >=latex
    }

\newcommand{\wt}{\widetilde}
\newcommand{\mc}{\mathcal}
\newcommand{\nc}{\newcommand}
\nc{\on}{\operatorname}
\nc{\la}{\lambda}
\nc{\wh}{\widehat}
\nc{\ghat}{\wh\g}
\nc{\mb}{\mathbf}

\begin{document}
\title{Toroidal $q$-Opers}

\author[P. Koroteev]{Peter Koroteev}
\address{
Department of Mathematics,
University of California,
Berkeley, CA 94720, USA}

\author[A.M. Zeitlin]{Anton M. Zeitlin}
\address{
          Department of Mathematics, 
          Louisiana State University, 
          Baton Rouge, LA 70803, USA; 
          IPME RAS, St. Petersburg, Russia}

\date{\today}

\numberwithin{equation}{section}

\begin{abstract}
We define and study the space of $q$-opers associated with Bethe equations for integrable models of XXZ type with quantum toroidal algebra symmetry. Our construction is suggested by the study of the enumerative geometry of cyclic quiver varieties, in particular, the ADHM moduli spaces. We define $(\overline{GL}(\infty),q)$-opers with regular singularities and then, by imposing various analytic conditions on singularities, arrive at the desired Bethe equations for toroidal $q$-opers.
\end{abstract}

\maketitle

\setcounter{tocdepth}{1}
\tableofcontents

\section{Introduction: Geometric facets of Bethe equations}
\subsection{Integrable Models and Bethe Ansatz}
The study of one-dimensional quantum integrable models fueled modern mathematics with a variety of interesting ideas, in particular, in discovery of the quantum groups and related structures. 
A particularly useful tool in the study of integrable models is the so-called algebraic Bethe ansatz method (see, e.g., \cite{Korepin_1993,Reshetikhin:2010si}) having its roots in the original papers of Hans Bethe from the 1930s. 

Let us briefly describe here the modern mathematical perspective on how algebraic Bethe ansatz works for integrable models of specific type, namely the spin chains. Let $\mathfrak{g}$ be a simple Lie algebra and 
$\hat{\mathfrak{g}}_{k=0}=\mathfrak{g}[t^{\pm 1}]$ be the corresponding loop algebra (affine algebra with vanishing central charge $k=0$). The finite-dimensional modules $\{V_i\}$ of $\mathfrak{g}$ give rise to the so-called evaluation modules $\{V_i(a_i)\}$,  where $a_i\in\mathbb{C}^\times$ stand for the value of the loop parameter $t$. These modules generate a tensor category, namely every finite-dimensional representation of $\hat{\mathfrak{g}}$ can be written as a tensor product of evaluation modules.  
Passing from $\mathfrak{g}[t^{\pm 1}]$ to the corresponding quantum affine algebra 
$U_{\hbar}(\hat{\mathfrak{g}})$ or the Yangian $Y_{\hbar}({\mathfrak{g}})$ 
one obtains a deformation of such tensor category, known as braided tensor category \cite{Chari:1994pz}. This object features a new  intertwining operator (not invertible in general):
$$R_{V_i(a_i),V_j(a_j)}:V_i(a_i)\otimes V_j(a_j)\to V_j(a_j)\otimes V_i(a_i),$$ 
satisfying famous Yang-Baxter equation.  We note, that in the deformed case the analogue of evaluation map exists only in the type $A$ and the $\{V_i(a_i)\}$ stand here for appropriately ``twisted" by $a_i$ finite-dimensional representations of $U_{\hbar}(\hat{\mathfrak{g}})$.

To describe the integrable model, we choose a specific object in such braided tensor category $$\mathcal{H}=V_{i_1}(a_{i_1})\otimes \dots \otimes V_{i_n}(a_{i_n}),$$ which we refer to as {\it physical space}, the vectors in this space are called {\it states}. 
For a given module $W(u)$ called {\it auxiliary module} with parameter $u$  known as {\it spectral parameter}, we define the {\it transfer matrix} 
$$T_{W(u)}={\rm Tr}_{W(u)}\Big[ (Z\otimes 1)P{R}_{W(u),\mathcal{H}}\Big].$$
Here the {\it twist Z} is given by $Z=\prod^r_{i=1}z_i^{\check{\alpha}_i}\in e^{\mathfrak{h}}$, where $\mathfrak{h}$ is the Cartan subalgebra in $\mathfrak{g}$, $\{\check{\alpha}_i\}_{i=1,\dots, r}$ are the simple coroots of $\mathfrak{g}$, and  $P$ is a permutation operator. The {\it monodromy matrix} $M^Z_{W(u)}=(Z\otimes 1)P{R}_{W(u),\mathcal{H}}$ is an operator in  $W(u)\otimes \mathcal{H}$. Notice, that the transfer matrix $T_{W(u)}$ is an operator acting on the physical space $\mathcal {H}$.  
The Yang-Baxter equation implies that transfer matrices, corresponding to various choices of $W(u)$ form a commutative algebra, known as {\it Bethe algebra}. The commutativity of Bethe algebra implies {\it integrability} and the expansion coefficients of the transfer matrix yield (nonlocal) {\it Hamiltonians} of the XXX or XXZ spin chain depending on whether we deal with the Yangian or the quantum affine algebra. From now on, we will primarily focus on quantum affine algebra and the XXZ model, although most of the construction below applies to the Yangian and the XXX models as well. 

The classic example of the XXZ Heisenberg magnet corresponds to the quantum algebra $U_{\hbar}(\hat{\mathfrak{sl}}(2))$ in which the physical space $\mathcal{H}$ is constructed from $V_i(a_i)=\mathbb{C}^2(a_i)$ -- the standard two-dimensional evaluation modules of $U_{\hbar}(\hat{\mathfrak{sl}}(2))$.

The solution of the integrable model implies finding the eigenvalues and eigenvectors of simultaneously diagonalized Hamiltonians, i.e. elements of the Bethe algebra. 
One way to accomplish the task is to follow the old-fashioned procedure from the 1980s known as {\it algebraic Bethe ansatz}.  
It implies that the eigenvalues of the transfer-matrices (upon rescaling) are symmetric functions of the roots of
the system of algebraic equations, known as Bethe ansatz equations. Although this approach is straightforward and effective, we will explore other modern techniques, which provide insights into representation-theoretic aspects of the problem.

\subsection{Modern Approach to Bethe Ansatz}
\subsubsection{Quantum Knizhnik-Zamolodchikov equations}
The intertwining operators for the quantum affine algebra $U_{\hbar}(\hat{\mathfrak{g}})$ and thus the matrix elements of their products, known as {\it conformal blocks} satisfy certain difference equations known as quantum Kniznik-Zamolodchikov (qKZ) equations (also known as Frenkel-Reshetikhin equations) \cite{FR1998}. 

Explicitly, qKZ equations can be written as follows: 
difference equations 
\begin{eqnarray}\label{qkz}
\Psi(a_{i_1}, \dots, q a_{i_k}, \dots , a_{i_n}, \{z_i\})=H^{(q)}_{i_k}\Psi(a_{i_1}, \dots,, a_{i_n}, \{z_i\}),
\end{eqnarray}
where the solutions $\Psi$ take values in $\mathcal{H}$ and operators $H^{(q)}_i$ are expressed in terms of products of R-matrices. The analytic properties of the solutions of qKZ equations will be discussed later in subsection 1.3.2. There is also a commuting system of equations in 
 $\{z_i\}$-variables for $\Psi$, known as {\it dynamical equations} 
 see e.g. \cite{Tarasov_2002,Tarasov_2005}.
 
 The solution to the qKZ equation is given by an integral expression, so that the integrand has the following asymptotic behavior in the limit $q\rightarrow 1$ (or $\eta=\log(q)\rightarrow 0$): 
\begin{eqnarray} 
 e^{\frac{Y(\{a_i\}, \{z_i\}, \{x_i\})}{\eta}}\Big[\phi_0(\{a_i\}, \{z_i\}, \{x_i\})+O(\eta)\Big],
 \end{eqnarray}
where $\{x_i\}$ are the variables of integration.  In the limit $q\rightarrow 1$ the stationary phase approximation gives 
 $\Psi=e^{\frac{S}{\eta}}(\Psi_0+O(\eta))$, where $S=Y|_{\sigma_i},$  
where $\sigma_i$ are the solutions of the equations $\partial_{x_i}Y=0$ which need to be solved with respect to the variables $\{x_i\}$.  These equations coincide with the Bethe equations, and $\Psi_0$ is the eigenvector for operators $H^{(1)}_i$, known as the nonlocal Hamiltonians of the corresponding XXZ model: they emerge as coefficients from the expansion of the transfer matrices with respect to the spectral parameter, viz. $H^{(1)}_i\Psi_0=e^{p_i}\Psi_0,~{\rm where} ~p_i={a_i}\partial_{a_i}S.$

\subsubsection{$QQ$-systems and Baxter operators}

When we earlier discussed the transfer matrices $T_{W(u)}$ we considered $W(u)$ to be a finite-dimensional module of $U_{\hbar}(\hat{\mathfrak{g}})$. We notice that the {\it universal R-matrix}, which produces particular braiding operators $R_{V_i(a_i), V_j(a_j)}$ belongs to the completion of the tensor product $U_{\hbar}(\hat{\mathfrak{b}}_{+})\otimes U_{\hbar}(\hat{\mathfrak{b}}_{-})$, where $U_{\hbar}(\hat{\mathfrak{b}}_{\pm})$ are the Borel subalgebras of $U_{\hbar}(\hat{\mathfrak{g}})$. Therefore there is no obstruction in taking auxiliary representations $W(u)$ to be representations of $U_{\hbar}(\hat{\mathfrak{b}}_{+})$.

The purpose of that is as follows. 
There exist  prefundamental representations of $U_{\hbar}(\hat{\mathfrak{b}}_{+})$ which are infinite-dimensional. If one extends the braided tensor category of finite-dimensional modules by such representations, the Grothendieck ring of the resulting category is generated by those modules. 

The corresponding transfer matrices turn out to be well-defined and 
moreover, the eigenvalues of the transfer matrices are polynomials of the spectral parameter, generating elementary symmetric functions of the solutions of Bethe equations. Such transfer matrices were originally introduced by Baxter and thus are known as {\it Baxter operators} \textit{ad hoc} via their eigenvalues. Their representation-theoretic meaning was realized much later, in the papers of Frenkel and Hernandez \cite{Frenkel:2013uda}, \cite{Frenkel:2016}, following earlier ideas of Bazhanov, Lukyanov and Zamolodchikov \cite{Bazhanov:1998dq} and Hernandez and Jimbo \cite{HJ}.

There are two series of prefundamental representations $\{V^i_+(u)\}_{i=1,\dots, r}$, $\{V^i_-(u)\}_{i=1,\dots, r}$ and the associated Baxter operators $\{Q^i_{\pm}(u)\}_{i=1,\dots, r}$.  They obey the following key relation \cite{Frenkel:2016}:
\begin{eqnarray}
{{\widetilde\xi_i}}Q^i_{-}({u})Q^i_{+}({\hbar}  {u})-{\xi_i}Q^i_{-}({\hbar}  {u})Q^i_{+}({u}) &=&\Lambda_i({u})\prod_{j\neq i}\Bigg[\prod^{-a_{ij}}_{k=1} Q^j_{+}({\hbar} ^{b_{ij}^k}{u})\Bigg]\nonumber\\ 
i&=&1,\dots,r, \quad b_{ij}^k\in \mathbb{Z}\nonumber 
\end{eqnarray}
Here polynomials $\Lambda_i({u})$ are known as Drinfeld polynomials, characterizing the representation $\mathcal{H}$ of $U_{{\hbar}}(\hat{\mathfrak{g}})$ and 
${\xi_i}$, ${{\widetilde\xi_i}}$ are some monomials of $\{z_i\}$.

This system of equations, known as the {\it $QQ$-system}, considered as equations on  $\{Q^i_{\pm}(u)\}_{i=1,\dots, r}$ and subject to some nondegeneracy conditions, are equivalent to the Bethe ansatz equations.

We note, that similar construction and the analogue of the $QQ$-system should also exist for Yangians with some progress being made in \cite{BFLMS}.

We mention that the $QQ$-systems emerged recently in a seemingly different context, the so-called ODE/IM correspondence (see \cite{Dorey1999,Bazhanov2001}). The statement of the correspondence can be roughly formulated as follows. 
The vacuum eigenvalues of the Baxter operators quantum KdV model associated with affine Lie algebra $\widehat{\mathfrak{g}}$ appear as spectral determinants of certain singular differential operators associated with the so-called affine opers associated with 
${^L}\widehat{\mathfrak{g}}$. In a particular case of standard quantum KdV, these operators are just singular Sturm-Liouville operators. As it was shown in \cite{MRV1,MRV2} they turn out to be the solution of the $QQ$-system with different analyticity conditions on entire Q-functions (which are generally nonpolynomial in this case).

\subsection{Geometric Interpretations}
\subsubsection{Quantum K-theory of Nakajima Varieties}

The relation between enumerative algebraic geometry and integrability was known for some time. Starting from the pioneering works of Witten and Dubrovin, it flourished in the works of A. Givental and his school in the 90s. Recently, the progress in the understanding of supersymmetric gauge theory merged with the developments in the geometric representation theory. In particular, the study of the so-called {\it symplectic resolutions} from the representation-theoretic point of view gave a new life to this fruitful relationship in works of A. Okounkov and his collaborators \cite{Braverman:2010ei},\cite{Okounkov:2015aa}, \cite{2012arXiv1211.1287M}. It was observed that some integrable systems based on quantum groups, specifically XXX and XXZ models, naturally emerge from enumerative geometry for a large class of algebraic varieties, known as {\it Nakajima quiver varieties} \cite{Nakajima:1999hilb},\cite{Ginzburg:}.

Let us recall this connection in the simplest nontrivial examples of such varieties, namely the cotangent bundles over Grassmannians $T^*\textbf{Gr}_{k,n}$. The standard objects in the enumerative geometry are the appropriate deformations of the cup product and the tensor product in the equivariant cohomology and K-theory correspondingly, where the deformation is characterized by the series in {\it K\"ahler parameters}, with coefficients being produced by curve counting. 

The physics results of Nekrasov and Shatashvili \cite{Nekrasov:2009uh} lead to the following conjecture about the equivariant quantum K-theory $K_T(T^*\textbf{Gr}_{k,n})$: the  quantum multiplication by the generating function for the exterior algebra of tautological bundle coincides with the Baxter $Q$-operator for the Heisenberg XXZ-spin chain. 
Also, since tautological bundles generate the entire quantum K-theory, one can describe the equivariant quantum K-theory ring as the ring of symmetric functions of Bethe roots.

The proof of that conjecture was given in \cite{Pushkar:2016qvw}. It uses the theory of quasimaps to Nakajima varieties as the `curve counting' which is different from the older approach to quantum K-theory due to A. Givental. To relate the quantum equivariant $K$-theory with spin chains, it is not enough to consider the operators of quantum multiplication by classical $K$-theory classes: in fact, both the multiplication in the equivariant $K$-theory and the  tautological classes  should be deformed simultaneously: in our case of $T^*\textbf{Gr}_{k,n}$, by just one K\"ahler parameter $z$. One can define elements $\hat{\mathcal{V}}^{\tau} \in K_{T} (T^{*}\textbf{Gr}_{k,n})[[z]]$ which we call quantum tautological bundles.  In the classical limit $z\rightarrow 0$, these elements coincide with the corresponding classical bundles $V^{\tau}$, which is a certain tensorial polynomial of standard tautological bundles,  corresponding  to the symmetric polynomial $\tau$ in $k$ variables in the standard $K$-theory.   
The localized equivariant quantum K-theory $K^{loc}_T(T^*\textbf{Gr}_{k,n})$ can be identified with appropriate weight subspace in the space $\mathcal{H}$ of the XXZ Heisenberg magnet, so that considering the union of such spaces for all $k$, one obtains the entire space of states $\mathcal{H}$.
 
To prove this conjecture one needs to define and compute {\it vertex functions}, which are quasimap analogues of Givental's {\it I-functions}. These are certain Euler characteristics, which count quasimaps and determine the quantum K-theory classes.  Such vertex functions satisfy the quantum difference equations (QDE) which coincide with qKZ and dynamical equations \cite{Okounkov:2016sya}, which were discussed in subsection 1.2.1. 
To understand the action of the operators of quantum multiplication by the quantum tautological bundles, one has to study the $q\rightarrow 1$ asymptotics of such solutions of qKZ.  That allows us to identify quantum tautological classes with the elements of the Bethe algebra, thereby leading to the proof of the conjecture of Nekrasov and Shatashvili.

Later these results have been proven for larger classes of Nakajima varieties, e.g., partial flag varieties, see, e.g., \cite{Koroteev:2017aa,Koroteev:2018a}.

Notice that this approach gives geometric interpretation to qKZ and dynamical equations associated to $\mathfrak{g}$ of simply-laced type for the specific physical spaces $\mathcal{H}$ emerging from quiver varieties. Moreover, each of the Q-operators on its own has a geometric meaning. The  $Q^i_+$-operators correspond to the exterior powers of tautological bundles. Their $Q^i_-$ counterparts correspond to the exterior powers of tautological bundles of Nakajima varieties with a different choice of stability parameters. In the case of flag varieties, such a change in the stability parameters is provided by the action of Weyl reflection. 

However, the $QQ$-system relations themselves do not arise naturally, since in particular, $Q^i_{\pm}$ operators do not act in the same space. In the next section, we will discuss another geometric viewpoint on Bethe ansatz, specifically related to the geometric interpretation of the $QQ$-system.

\subsubsection{Quantum q-Langlands Correspondence}

To oversee the geometric interpretation of $QQ$-systems, we take several steps back in time and deformation-wise. 
Earlier in the introduction we described the construction of the XXZ spin chain. There is a certain scaling limit of the XXZ model which is called the Gaudin model. This limit can be understood quasiclassically as $\hbar\rightarrow 0$ of (\ref{qkz}), so that $q=\hbar^{k+h^{\vee}}$, where $h^{\vee}$ is a dual Coxeter element.  
Then the qKZ equation turns into the differential equation, known as Knizhnik-Zamolodchikov (KZ) equation:
\begin{eqnarray}\label{kz}
(k+h^{\vee})\partial_{a_i}\Psi=H_i\Psi,
\end{eqnarray}
where $H_i$ emerge as coefficients of $\log{\hbar}$ in $H^{(q)}$. The solution $\Psi$ belongs to the classical limit of $\mathcal{H}$, viz. the tensor product of some evaluation representations of Lie algebra $\hat{\mathfrak{g}}$ with evaluation parameters $a_i$: $\mathcal{H}^{cl}=V_1(a_1)\otimes \dots \otimes V_n(a_n)$. The mutually commuting Gaudin Hamiltonians $H_i$,  have easy to read expressions 
$$H_i=\sum^n _{j=1,i\neq j}\frac{t^{\alpha}_i\otimes t^{\alpha}_j}{a_i-a_j}+{\mathcal{Z}}_i\,,$$ 
where ${t^{\alpha}}$ form an orthonormal basis in $\mathfrak{g}$ with respect to the Killing form, $\mathcal{Z}$ belongs to Cartan subalgebra of $\mathfrak{g}$ and indices $i,j$ indicate on which of the 
representations $V_i$ these elements act.

One can see that in the limit $k\rightarrow -\hbar^{\vee}$, known as {\it critical level} limit, this equation turns into an eigenvalue problem for Gaudin Hamiltonians. 
It is possible to interpret the solutions of the KZ equations in a particular analyticity region in evaluation parameters $|a_1|>|a_2|>\dots >|a_n|$, as the equations for the intertwiners of $\hat{\mathfrak{g}}$ with central charge $k$. Moreover, the Gaudin Hamiltonians were shown in \cite{Feigin:1994in} to be part of a bigger
structure, namely the center $Z(U(\hat{\mathfrak{g}}))$ of $U(\hat{\mathfrak{g}})$  at the critical level ${k=-h^{\vee}}$. 

There is a natural Poisson structure on $Z(U(\hat{\mathfrak{g}}))$, arising from standard commutators away from the critical level. The famous theorem of Feigin and Frenkel \cite{FEIGIN_1992} provides isomorphism of this Poisson algebra and the classical limit of {\it $W$-algebra} $W(^L \mathfrak{g})$, associated to the Langlands dual Lie algebra $^L \mathfrak{g}$, also known as {\it Gelfand-Dickey algebra} of pseudodifferential operators in the case $ \mathfrak{g}=\mathfrak{sl}_n$. 

Later this statement was reformulated \cite{FrenkelWak} in terms of special connections for principal $^L\tilde{G}$-bundles, known as {\it oper connections}, on the punctured disk, where $^L\tilde{G}$ is an adjoint Lie group associated to $^L \mathfrak{g}$. The reformulated Feigin-Frenkel theorem implies that there is an isomorphism between $Z(U(\hat{\mathfrak{g}}))$ and the space of functions on $^L\tilde{G}$-oper connections on a punctured disk. The path from such connections to Gelfand-Dickey pseudo-differential operators is given by a well-known construction, known as {\it  Drinfeld-Sokolov reduction} \cite{Drinfeld:1985}. 

Let us return to the eigenvalue problem for Gaudin Hamiltonians, arising from critical level limit of (\ref{kz}). E. Frenkel's theorem \cite{Frenkel:2003qx} gives a geometric description of the spectrum in terms of opers. Explicitly, Frenkel's theorem states that there is a one-to-one correspondence between the space of {\it Miura oper} connections with regular singularities with trivial monodromies around them on $\mathbb{P}^1$ in case when $\mathcal{Z}=0$. 
The word `Miura' there means that there is an extra condition on such oper connections: they have to preserve the reduction of $^L\tilde{G}$-bundle to Borel subgroup. Later, this theorem was generalized for $\mathcal{Z}\neq 0$ by adding irregular singularity at $\infty\in\mathbb{P}^1$ \cite{Feigin:2006xs}.  
 
The constraints on such connections could be expressed in terms of Wronskian-type relations,  which are particularly manifest in the case of $SL(N)$. That suggests, that the $QQ$-system, which is a deformation of the  Wronskian relation, should arise from an appropriate $\hbar$-deformation of Miura opers. Below we shall provide further motivation and hints in this direction.
 
The pseudodifferential operators, corresponding to such Miura oper connections with regular singularities through Drinfeld-Sokolov reduction, describe the constraints on the conformal blocks/intertwiners of $W(^L \mathfrak{g})$-algebra in the limit when central charge $c\rightarrow \infty$. The most famous such constraint is known as the Belavin-Polyakov-Zamolodchikov (BPZ) equation (essentially the Sturm-Liouville problem with singular potential) for conformal blocks of Virasoro algebra, which is the case when $\mathfrak{g}=\mathfrak{sl}(2)$.

Naturally, that led to the quantum Langlands correspondence linking conformal blocks of $W(^L \mathfrak{g})$-algebras and $\hat{\mathfrak{g}}$-conformal blocks away from the critical level. Recently, a q-deformation of this correspondence was proposed in  \cite{Aganagic:2017smx}. The proof of the latter, given in \cite{Aganagic:2017smx}, provided in the case of simply-laced $\mathfrak{g}$ is based on the enumerative geometry approach which we touched briefly in 1.3.1. The key to that is to further deform this correspondence, namely, identify conformal blocks for the quantum affine algebra $U_{\hbar}(\hat{\mathfrak{g}})$ and the deformed W-algebra $W_{q,t}(^L \mathfrak{g})$, which is  the 2-parametric deformation of Gelfand-Dickey algebra \cite{FR1998}.  

The conformal blocks for $U_{\hbar}(\hat{\mathfrak{g}})$, as we discussed in subsection 1.2.1, satisfy the qKZ equation. We remind, that as in the classical case, they correspond to the solution of qKZ, analytic in the region $|a_1|>|a_2|>\dots >|a_n|$. However, the solutions of qKZ which are provided by enumerative geometry, i.e.  vertex functions, are analytic in $\{z_i\}$-variables. Also, it turns out that they are the ones, producing the conformal blocks of $W_{q,t}(^L \mathfrak{g})$-algebra. The transition between two families of solutions is crucial for establishing the exact correspondence between such conformal blocks. We refer to \cite{Aganagic:2017smx} for the details.

\subsubsection{Miura $\hbar$-opers}

A natural question is to understand the difference analogues of BPZ-type equations which serve as constraints for the conformal blocks of $W_{q,t}(^L \mathfrak{g})$. As we have discussed, the differential BPZ equations on the critical level correspond to the classical objects, 
namely $^L\tilde{G}$-oper connections with regular singularities on $\mathbb{P}^1$. 
Let $^L G$ be  the simply-conected group with Lie algebra $^L \mathfrak{g}$. There is a natural classical object, the $\hbar$-difference connection, locally a meromorphic $^L G$-valued function $A(z)$ on Zariski open set of $\mathbb{P}^1$, which transform upon trivialization change $A(z)\to g(\hbar z)A(z)g^{-1}(z)$.

In \cite{Frenkel:2020}, following the constructions in \cite{KSZ} done for $SL(N)$,  
we developed the $\hbar$-difference analogue of opers as such $\hbar$-difference connections for any simply connected semisimple Lie group $^L G$ with a fixed Borel subgroup $^L B_-$.  
Locally, these $\hbar$-connections have the form $A(z)=n'(z)\prod^r_{i=1}s_i\phi_i^{\check{\alpha}_i}(z)n(z)$. Here $n(z), n'(z)\in G(z)$, $\phi_i(z)\in \mathbb{C}(z)$, $s_i$ are the lifts of the fundamental Weyl reflections to $^L G$. In other words  $A(z)\in B_-(z)cB_-(z)$, where $c=\prod^r_{i=1}s_i$ is a Coxeter element.

Moreover, we defined such $(^L G, \hbar)$-opers and their Miura versions with regular singularities, which amounts to the connections of this type which preserve the opposite Borel subgroup of $B_+$ and taking $\phi_i(z)=\Lambda_i(z)\in \mathbb{C}[z]$.  We proved several structural theorems about them.

One of the major statements we make in \cite{Frenkel:2020} is devoted to the explicit relation of these objects to the $QQ$-systems and Bethe ansatz \footnote{ We note here, that the QQ-system we obtained in \cite{Frenkel:2020} differs from the ones we discussed in the subsection 1.2.2 in the non-simply laced case (the difference is in the coefficients $\{b_{ij}\}$). We mention the upcoming paper \cite{FHRnew}, where the integrable model corresponding to such QQ-system is discussed. In this paper we will be working mostly with the QQ-system associated with $\mathfrak{sl}(n)$ and we refer the reader to \cite{Frenkel:2020} for more details.}. To do that, we work with two versions of what we call {\it Z-twisted condition} for Miura opers. The simplest $Z$-twisted condition implies that the $(^L G,\hbar)$-oper connection can be $\hbar$-gauge equivalent to semisimple element $Z\in H\subset ^L G$,  where $H$ is the Cartan subgroup. That means  $A(z)=g(\hbar z)Zg^{-1}(z)$. This condition is a difference version of zero monodromy condition and double pole irregular singularity at $\infty$ point of $\mathbb{P}^1$.

The relaxed version of this $Z$-twisted condition is as follows. Given the principal $^L G$-bundle, one can construct an associated bundle for any fundamental representation  $V_{\omega_i}$ for the fundamental weight $\omega_i$.  
It turns out, one can associate a $(GL(2),\hbar)$-oper to any such pair  $(^L G,\hbar)$-oper and $V_{\omega_i}$: this is done by restricting the Miura $(^L G,\hbar)$-oper to the two-dimensional subspace, spanned by two top weights in $V_{\omega_i}$. This is possible, since Miura $(^L G,\hbar)$-oper preserves the reduction to positive Borel subgroup $^L B_+\subset ^L\!\! G$.  

We say that the resulting Miura oper is {\it $Z$-twisted Miura-Pl\"ucker}  $(^L G,\hbar)$-oper if for every such $(GL(2),\hbar)$ oper is $\hbar$-gauge equivalent to the restriction of $Z$ to the corresponding two-dimensional space. 

In \cite{Frenkel:2020} we showed that {\it $Z$-twisted Miura-Pl\"ucker}  $(^L G,\hbar)$-opers with mild non-de\-ge\-ne\-racy conditions are in one-to-one correspondence with certain solutions of $QQ$-systems and that does not depend on the order in the Coxeter element. In simply-laced case such $QQ$-systems are equivalent to standard Bethe ansatz equations. The non-simply laced case is more involved (see the discussion in \cite{Frenkel:2020} and the upcoming paper \cite{FHRnew}.

While it immediately follows that any $Z$-twisted Miura oper is indeed  $Z$-twisted Miura-Pl\"ucker one, the opposite statement, however, is highly nontrivial. In \cite{Frenkel:2020} we introduce a chain of $\hbar$-gauge  transformations, which we refer to as $\hbar$-B\"acklund transformations, which on the level of $QQ$-systems amounts to the $Q^i_+(z)\rightarrow Q^i_-(z)$, $Z\rightarrow s_i(Z)$, where $s_i$ is elementary Weyl reflection. However, at every step, in order to progress further, we have to impose the nondegeneracy condition on the $QQ$-system and the associated Miura oper. We have shown that if one can proceed with this transformations to $Z$-twisted Miura-Pl\"ucker oper, corresponding to the $w_0(Z)$, where $w_0$ is the longest Weyl group element, then such $Z$-twisted Miura-Pl\"ucker Miura oper is $Z$-twisted. We call such Miura-Pl\"ucker opers and the associated $QQ$-system $w_0$-generic. 

We also discuss the explicit version of $\hbar$-version of Drinfeld-Sokolov reduction, following the ideas of  \cite{1998CMaPh.192..631S}.  The scalar difference equations emerging this way from Z-twisted Miura $(^L G,\hbar)$-opers and the correspondence with the difference equations which the conformal blocks for $W_{q,t}(^L \mathfrak{g})$-algebras remains an interesting open problem.

\subsection{Our goals in this paper} 

\subsubsection{Two approaches to $(SL(r+1),\hbar)$-opers} 

In this paper, we are investigating several problems. The first one is devoted to the correspondence between the results of \cite{KSZ}, where we work with $SL(r+1)$ case only, and a more general approach of \cite{Frenkel:2020}. In \cite{KSZ} we used a definition of (Miura) $\hbar$-oper which is very specific to 
$SL(r+1)$. It can be deduced from the `universal' definition of $(SL(r+1),\hbar)$-oper as  an $\hbar$-connection for the principal $SL(r+1)$-bundle, which we discussed in the previous section, with the standard order of reflections in the corresponding Coxeter element (following the order in the Dynkin diagram), so that in the defining representation it is represented as the matrix with zeroes above its superdiagonal. 

Considering the associated bundle, corresponding to the defining representation, one can reinterpret the oper condition in the following way. Namely, it is the condition on the oper action in the complete flag of subbundles of this associated bundle, which reflects its matrix structure described above. In particular, that implies that on a Zariski dense subset in $\mathbb{P}^1$ the total space of the flag can be recreated by the consecutive action of the $(SL(r+1),\hbar)$-oper connection on the section of the line bundle. The Miura condition can be reformulated as the constraint, that the connection preserves a different complete flag of subbundles.

Such definition lead to another approach to the derivation of the $QQ$-systems from $Z$-twisted Miura $(SL(r+1),\hbar)$-opers with regular singularities. This is done using  $\hbar$-deformed Wronskian matrices. Their matrix elements are components of the nontrivial section of the line bundle in the trivialization when oper connection is represented by the regular semisimple twist element $Z$ and describes the relative position of two flags of subbundles. It turns out, that the points where these flags are in a non-generic position correspond to Bethe roots and $QQ$-systems, as we have demonstrated in \cite{KSZ}. Here we show that the extension of the $QQ$-system by $\hbar$-B\"acklund transformations is provided by various minors in this $\hbar$-Wronskian matrix
\footnote{We note here, that such extensions of the $QQ$-systems were introduced in various circumstances e.g. 
\cite{BFLMS,MRV1,MRV2,Frenkel:2016} and studied systematically in the case $Z=1$ by Mukhin and Varchenko (see \cite{Mukhin_2005}}. 

More importantly, we explicitly construct the element $g(z)$, such that the connection takes form $A(z)=g(\hbar z)Zg^{-1}(z)$. This element can be represented both in abstract Lie-theoretic form as well explicitly in the matrix notation, which uses polynomials of the extended $QQ$-system. As a consequence, we obtain that the $w_0$-generic condition, which was needed in general for {\it $Z$-twisted Miura-Pl\"ucker}  $(^L G,\hbar)$-oper to be just $Z$-twisted, is not needed for $^L G=SL(N)$.

\subsubsection{Completion to $(\overline{GL}(\infty),\hbar)$-opers}

Following the  calculations of $(SL(r+1),\hbar)$-opers it is not hard to extend this construction to $SL(\infty)$ -- the group of infinite-dimensional matrices with unit determinant with a finite amount of nonzero off-diagonal entries and finite amount of non-unit elements on the diagonal. However, for any Miura $(SL(\infty), \hbar)$-oper, the corresponding $QQ$-system will always be finite. Let us explain how to construct a Miura oper, which corresponds to the `complete' $QQ$-system associated with the Dynkin diagram of $A_{\infty}$.     
We note, that $SL(\infty)$  has a well-defined set of fundamental representations based on semi-infinite wedge spaces, which has an interpretation in terms of $Dirac$ $sea$, and the generators of Lie algebra $\mathfrak{sl}(\infty)$ are represented via quadratic expressions of the fermionic operators of exterior and interior multiplication, thereby generating Clifford algebra.  

One can complete the corresponding Lie algebra $\mathfrak{sl}(\infty)$ by allowing infinite sums of generators. The resulting Lie algebra, endowed by central extension equal to 1, has fundamental representations realized in the same spaces as $\mathfrak{sl}(\infty)$. This is an important construction, that plays a central role in the celebrated {\it boson-fermion} correspondence \cite{FRENKEL1981259,doi:10.1142/8882} 
 
To address related Miura opers, we take a certain completion of $SL(\infty)$, which will be sufficient to put an infinite number of terms in the $QQ$-system. Namely, we construct the group corresponding to the completion of the upper Borel subgroup in the Bruhat decomposition of $SL(\infty)$. 
The resulting object, denoted by $\overline{GL}(\infty)$, is the group of the infinite matrices with an infinite number of elements above the diagonal and an infinite number of nonunital elements on the diagonal, while the number of elements below the diagonal remains finite. It has the same set of fundamental representations realized in the same set Dirac sea spaces as described above.

The resulting Miura $(\overline{GL}(\infty),\hbar)$-opers satisfy similar properties as the $SL(r+1)$ ones. One can define $Z$-twisted and $Z$-twisted Miura-P\"ucker opers and explicitly construct the operator from completed upper Borel subalgebra, 
diagonalizing the corresponding connection matrix.   
As before, it is constructed from the elements of the extended $QQ$-system. 

As an application of this construction, we can built the main novel objects of the current paper, namely {\it toroidal opers}.

\subsubsection{Toroidal opers and the q-Langlands correspondence for toroidal algebras}

There is a natural family of automorphisms of $sl(\infty)$ algebra, corresponding to the Dynkin diagram translations through $n$ vertices. On the group-theoretic level, such transformations are realized via the $n$-th power of the `completed' Coxeter element $c$ (infinite matrix with the only nonzero elements being units on the superdiagonal). Imposing the condition $c^nA(z)c^{-n}=A(pz)$ for $Z$-twisted $(\overline{GL}(\infty),\hbar)$, where $p$ is a new parameter, we obtain that the resulting constrained infinite $QQ$-system generate Bethe equations for the toroidal algebras $\widehat{\widehat{\mathfrak{gl}}}(n)$.

While the corresponding $QQ$-system for toroidal algebras \cite{Frenkel:2016} has yet to emerge from the perspective of prefundamental representations and Grothendieck ring, the Bethe equations for toroidal algebra $\widehat{\widehat{\mathfrak{gl}}}(n)$ in representation-theoretic setting emerged 
through the shortcut, namely  Baxter $TQ$-relation  \cite{Feigin_2017,Hernandez:2016}, the relation between the Q-operator and the transfer-matrix.

However, a more natural approach to generate Bethe equations for toroidal algebras emanates from enumerative geometry. In section 1.3.1, we discussed elementary examples of quiver varieties, namely $T^*\textbf{Gr}_{k,n}$ and, in general, cotangent bundles to partial flag varieties. The corresponding quantum K-theory ring reproduces the Bethe algebra for the XXZ model related to $\widehat{\mathfrak{sl}}(n)$. Another set of varieties, which have been extensively studied, are the framed cyclic quiver varieties, which are related to $\widehat{\widehat{\mathfrak{gl}}}(n)$ toroidal algebras, where $n$ is the number of vertices. In the simplest situation of one vertex, such variety is identified with the space of ADHM instantons \cite{Schiffmann_2013}. One can find more details on algebraic properties of quantum toroidal algebras and their geometric realization in the recent reviews \cite{Negut:thesis}, \cite{negu2020rmatrix}.

 According to general construction, the $q\rightarrow 1$ asymptotic of $z$-analytic solutions of the resulting qKZ equations reproduce the Bethe equations, which serve as constraints for the quantum K-theory ring. These are exactly the equations we reproduce from toroidal opers.
 Given that $t_1,t_2$ are the standard deformation parameters of $U_{t_1,t_2}\big(\widehat{\widehat{\mathfrak{gl}}}(n)\big)$ we obtain the following exchange of parameters:
\begin{equation}\label{eq:parametersmatch}
(\hbar, p)\leftrightarrow ((t_1t_2)^{-1}, t_1),
\end{equation}
which serve as the first example of the analogue q-Langlands correspondence for toroidal algebras.

\subsubsection{String Theory Motivation}
In string theory literature it is common to study limits when the number of objects, like branes becomes infinite. The most relevant example to this paper is topological holography program which was initiated by Gopakumar and Vafa \cite{Gopakumar_1999}. According to \textit{loc. cit.} a topological phase transition can be regarded as an interpolation between two desingularizations of the conifold geometry -- deformed conifold $T^*S^3$ and resolved conifold $\mathscr{O}(-1)_{\mathbb{P}^1}^{\oplus 2}$. 

The M-theory description of the former phase, in the presence of certain defects and flux through one of the complimentary complex directions, after dimensional reduction, leads to a three-dimensional quiver gauge theory on $S^1\times\mathbb{C}_q$. Massive spectrum of such 3d theories is described by the equivariant quantum K-theory of the corresponding quiver varieties which we discussed earlier. Parameter $\hbar$ from above plays the role of a $\cN=4$ R-symmetry equivariant parameter. 

The latter, resolved phase yields five dimensional gauge theory. The moduli space of instantons in this 5d theory is given by the ADHM quiver which will later in this paper be discussed in connection with toroidal $q$-opers. 

The topological phase transition from the deformed phase to the resolved phase occurs when the number of branes which wrap $S^3$ cycle in the deformed geometry, and determines the number of gauge groups in the 3d theory becomes infinite. In addition, a certain quantization condition between the Omega-background parameters of the 3d gauge theory and other mass parameters of the problem must be satisfied. Namely, if $s_i$ and $s_{i+1}$ are complexified gauge field vacuum expectation values of vector superfields of the $i$th and $(i+1)$st gauge groups respectively, then the condition reads $\frac{s_{i+1}}{s_i}=p^n$ where $n$ is an integer. On the resolved side of the transition parameter $p$ becomes equivariant parameter of the K-theory of the ADHM moduli space. Same parameters already appeared in \eqref{eq:parametersmatch}.

Representation-theoretic aspects of the Gopakumar-Vafa transition in connection with quantum geometry of quiver varieties of A-type were discussed in \cite{Koroteev:2018}. This paper provides an alternative description of the same physics in terms of \textit{bona fide} classical objects -- $q$-opers. By combining our results with those of the first author in \cite{Koroteev:2018} we can establish the \textit{quantum/classical} duality between quantum XXZ spin chain of $\widehat{A}_0$-type, whose Bethe equations coincide with relations in quantum equivariant K-theory of the ADHM quiver variety, and the so-called 1-toroidal $q$-opers.
This correspondence can be regarded as the large rank limit of the quantum/classical duality which was discussed in both physics \cite{Gaiotto:2013bwa,Koroteev:2015dja} and mathematics \cite{KSZ} literature.

\vskip.14in

We also note that  oper-related structures in type-A as well as their super analogues appeared in recent physics literature on integrability in the AdS/CFT correspondence \cite{Kazakov_2016,Chernyak:2020} as well as some earlier work \cite{Krichever_1997}.

\subsubsection{Structure of the paper} 
In Section \ref{Sec:SLrOpers} we give two equivalent definitions of $(SL(r+1),q)$-opers and their Miura versions as $q$-connections, which were introduced in \cite{Frenkel:2020} and \cite{KSZ} correspondingly. The first definition uses the Lie-theoretic approach and the second one is using complete flags of subbundles. 

In Sections \ref{Sec:ZtwistedMiura},\ref{Sec:MiuraPlucker}, and \ref{Sec:MiuraSL} we elaborate on the Lie-theoretic definition and remind basic constructions of \cite{Frenkel:2020}. Section \ref{Sec:ZtwistedMiura} is devoted to $Z$-twisted $(SL(r+1),q)$-opers, which are q-gauge equivalent to a diagonal matrix. In Section \ref{Sec:MiuraPlucker} a more mild version of $Z$-twisted condition is introduced, which is related to associated bundles leading to the notion of $Z$-twisted Miura-Pl\"ucker $(SL(r+1),q)$-opers. We also discuss nondegeneracy conditions for these objects. Section \ref{Sec:MiuraSL} addresses the one-to-one correspondence between $Z$-twisted Miura-Pl\"ucker opers and the nondegenerate solutions of the QQ-systems (and thus Bethe ansatz equations)  as well as their extension.  We also prove that  $Z$-twisted Miura-Pl\"ucker $(SL(r+1),q)$-opers are $Z$-twisted and relate extended QQ-system to quantum B\"acklund transformations, introduced in \cite{Frenkel:2020}. 

In Section \ref{Sec:QWronskians} we use the second definition of Miura $(SL(r+1),q)$-opers and show how $Z$-twisted condition as well as quantum B\"acklund transformations can be reformulated in terms of q-Wronskian matrices, extending the results of \cite{KSZ}. 

In Section \ref{Sec:FockSpace} we describe the fermionic realization of $Z$-twisted Miura $(SL(r+1),q)$-opers using the realization of the fundamental representations in the fermionic Fock space. We then use it as a motivation to write an infinite rank formula. To do that, in Section \ref{Sec:FockSpace} we introduce the group 
$\overline{GL}(\infty)$ and its representations in the fermionic Dirac sea, i.e. semi-infinite wedge space and then in Sections \ref{Sec:GLinfOpers} and \ref{Sec:MiuraPluckerInf} we extend the 
finite-dimensional notions of $(SL(r+1),q)$-oper theory from earlier
Sections to the case of 
$\overline{GL}(\infty)$. 
In particular, we show the relation between the corresponding infinite generalization of the $QQ$-system and $Z$-twisted Miura $(\overline{GL}(\infty),q)$-opers. 

Finally, Section \ref{Sec:QuantumToroidal} is devoted to the main target of the paper, the toroidal opers. These are nondegenerate $Z$-twisted Miura opers with certain periodicity conditions. 
The main goal of Section \ref{Sec:QuantumToroidal} is to show that they are in one-to-one correspondence with the nondegenerate solutions of the $QQ$-system for toroidal algebras.  We also discuss the relation to the enumerative geometry of ADHM spaces and generalizations to framed cyclic quiver varieties.

\vspace*{2mm}

\noindent{\bf Acknowledgments.} We thank E. Frenkel for his advices. P.K. is partially supported by AMS Simons grant. A.M.Z. is partially
supported by Simons Collaboration Grant, Award ID: 578501. 

\section{$(SL(r+1),q)$-opers}\label{Sec:SLrOpers}
\subsection{Group-theoretic data and notations.}
Consider $SL(r+1)$ be the simple algebraic group of invertible $(r+1)\times (r+1)$
matrices over $\mathbb{C}$.  We fix a Borel subgroup $B_-$ with unipotent
radical $N_-=[B_-,B_-]$ of lower triangular matrices and strictly lower triangular matrices correspondingly. The maximal torus is the corresponding set of diagonal matrices $H\subset B_-$.  Let $B_+$ be the opposite Borel subgroup containing $H$.  
Let $\{
\alpha_1,\dots,\alpha_{r} \}$ be the set of positive simple roots for
the pair $H\subset B_+$.  Let $\{ \check\alpha_1,\dots,\check\alpha_{r}
\}$ be the corresponding coroots. Then the elements of the Cartan
matrix of the Lie algebra $\mathfrak{sl}(r+1)$ of $G$ are given by $a_{ij}=\langle
\alpha_j,\check{\alpha}_i\rangle$. The Lie algebra $\mathfrak{sl}(r+1)$ has Chevalley
generators $\{e_i, f_i, \check{\alpha}_i\}_{i=1, \dots, r}$, so
that $\fb_-=\Lie(B_-)$ is generated by the $f_i$'s and the
$\check{\alpha}_i$'s and $\fb_+=\Lie(B_+)$ is generated by the $e_i$'s 
and the $\check{\alpha}_i$'s. 
In the defining representation $\check{\alpha}_i\equiv E_{ii}-E_{i+1,i+1}$, $e_i\equiv E_{i,i+1}$, $f_i\equiv E_{i-1,i}$, where $E_{ij}$ stand for the matrix  with the only nonzero element 1 at ij-th place. 
The fundamental weights $\omega_1,\dots\omega_r$ are defined by the condition $\langle \omega_i,
\check{\alpha}_j\rangle=\delta_{ij}$. 

Let $W_{SL(r+1)}=N(H)/H\equiv S_{r+1}$ be the Weyl group of $SL(r+1)$. Let $w_i\in W_{SL(r+1)}$, $(i=1,
\dots, r)$ denote the simple reflection corresponding to
$\alpha_i$. We also denote by $w_0$ be the longest element of $W$, so
that $B_+=w_0(B_-)$.  Recall that a Coxeter element of $W$ is a
product of all simple reflections in a particular order. It is known
that the set of all Coxeter elements forms a single conjugacy class in
$W_G$. We will fix once and for all (unless specified otherwise) a
particular ordering of the simple
roots, according to the natural ordering provide by Dynkin diagram. Let $c=w_{r}w_{r-1}\dots w_{1}$ be the Coxeter element associated 
to this ordering. In what follows (unless specified otherwise) all
products over $i \in \{ 1, \dots, r \}$ will be taken in this order;
thus, for example, we write $c=\prod_i w_i$.  We also fix
representatives $s_i\in N(H)$ of $w_i$. In particular, $s=\prod_i s_i$
will be a representative of $c$ in $N(H)$.

In the following we will denote the deformation parameter $q$ instead of $\hbar$ for convenience purposes.\footnote{The notation mix-up is due to the fact that in the context of the $q$-deformed geometric Langlands corresondence variable $\hbar$ is used for 
the deformation parameter of the quantum algebra, while $q$ is the parameter in the qKZ equation (see Introduction). Since we do not use qKZ equation anywhere in the paper, we renamed $\hbar$ as $q$.}

\subsection{$(SL(r+1),q)$-opers: two definitions}

Let's consider the automorphism $M_q: \P^1 \to \P^1$ sending $z \mapsto qz$, where
$q\in\C^\times$ is {\em not} a root of unity.

Given a principal $SL(r+1)$-bundle $\cF_{SL(r+1)}$ over $\P^1$ (in Zariski
topology), let $\cF_{SL(r+1)}^q$ denote its pullback under the map $M_q: \P^1
\to \P^1$ sending $z\mapsto qz$. A meromorphic $(SL(r+1),q)$-{\em
  connection} on a principal $SL(r+1)$-bundle $\cF_{SL(r+1)}$ on $\P^1$ is a section
$A$ of $\Hom_{\cO_{U}}(\cF_{SL(r+1)},\cF_{SL(r+1)}^q)$, where $U$ is a Zariski open
dense subset of $\P^1$. We can always choose $U$ so that the
restriction $\cF_{SL(r+1)}|_U$ of $\cF_{SL(r+1)}$ to $U$ is isomorphic to the trivial
$SL(r+1)$-bundle. Choosing such an isomorphism, i.e. a trivialization of
$\cF_{SL(r+1)}|_U$, we also obtain a trivialization of
$\cF_{SL(r+1)}|_{M_q^{-1}(U)}$. Using these trivializations, the restriction
of $A$ to the Zariski open dense subset $U \cap M_q^{-1}(U)$ can be
written as section of the trivial $SL(r+1)$-bundle on $U \cap M_q^{-1}(U)$,
and hence as an element $A(z)$ of $SL(r+1)(z)$, where  we set $K(z)=K(\C(z))$.
Changing the trivialization of $\cF_{SL(r+1)}|_U$ via $g(z) \in SL(r+1)(z)$ changes
$A(z)$ by the following $q$-{\em gauge transformation}:
\begin{equation}    \label{gauge tr}
A(z)\mapsto g(qz)A(z)g(z)^{-1}.
\end{equation}
This shows that the set of equivalence classes of pairs $(\cF_{SL(r+1)},A)$ as
above is in bijection with the quotient of $SL(r+1)(z)$ by the $q$-gauge
transformations \eqref{gauge tr}.
Equivalently, one could consider the associated to $\cF_{SL(r+1)}$ the vector bundle $E$ of rank $r+1$ over $\P^1$ and define $(SL(r+1),q)$-connection as a section of $\Hom_{\cO_{U}}(E,E^q)$, which is invertible and has determinant 1.

Following \cite{Frenkel:2020} we define a $(SL(r+1),q)$-oper as follows.

\begin{Def}    \label{qop}
  A meromorphic $(SL(r+1),q)$-{\em oper} (or simply a $q$-{\em oper}) on
  $\mathbb{P}^1$ is a triple $(\cF_{SL(r+1)},A,\cF_{B_-})$, where $A$ is a
  meromorphic $(SL(r+1),q)$-connection on a $SL(r+1)$-bundle $\cF_{SL(r+1)}$ on
  $\mathbb{P}^1$ and $\mathcal{F}_{B_-}$ is the reduction of $\cF_{SL(r+1)}$
  to $B_-$ satisfying the following condition: there exists a Zariski
  open dense subset $U \subset \P^1$ together with a trivialization
  $\imath_{B_-}$ of $\mathcal{F}_{B_-}$, such that the restriction of
  the connection $A: \cF_{SL(r+1)} \to \cF_{SL(r+1)}^q$ to $U \cap M_q^{-1}(U)$,
  written as an element of $SL(r+1)(z)$ using the trivializations of
  ${\mathcal F}_{SL(r+1)}$ and $\cF_{SL(r+1)}^q$ on $U \cap M_q^{-1}(U)$ induced by
  $\imath_{B_-}$ takes values in the Bruhat cell 
  $$
  B_-(\C[U \cap
  M_q^{-1}(U)]) \,c\, B_-(\C[U \cap M_q^{-1}(U)])\,.
  $$
\end{Def}

Thus locally, any $q$-oper connection
$A$ can be written (using a particular trivialization $\imath_{B_-}$)
in the form
\begin{eqnarray}    \label{qop1}
A(z)=n'(z)\prod_i (\phi_i(z)^{\check{\alpha}_i} \, s_i )n(z)
\end{eqnarray}
where $\phi_i(z) \in\C(z)$ and  $n(z), n'(z)\in N_-(z)$ are such that
their zeros and poles are outside the subset $U \cap M_q^{-1}(U)$ of
$\P^1$.

However, we used another definition in \cite{KSZ}, namely:

\begin{Def}    \label{qopflag}
  A meromorphic $(GL(r+1),q)$-{\em oper}  on
  $\mathbb{P}^1$ is a triple $(\mathcal{A},E, \mathcal{L}_{\bullet})$, where $E$ is  a vector bundle of rank $r+1$ and $\mathcal{L}_{\bullet}$ is the corresponding complete flag of the vector bundles, 
  $$\mathcal{L}_{r+1}\subset ...\subset \mathcal{L}_{i+1}\subset\mathcal{L}_i\subset\mathcal{L}_{i-1}\subset...\subset \mathcal{L}_1=E,$$ 
  where $\mathcal{L}_{r+1}$ is a line bundle, so that $q$-connection 
  $\mathcal{A}\in \Hom_{\cO_{U}}(E,E^q)$ 
  satisfies the following conditions:\\ 
i) $\mathcal{A}\cdot \mathcal{L}_i\subset \mathcal{L}_{i-1} $,\\
ii)  There exists a Zariski open dense subset $U \subset \P^1$, such that the restriction of $\mathcal{A}\in Hom(\mathcal{L}_{\bullet}, \mathcal{L}^q_{\bullet})$ to $U \cap M_q^{-1}(U)$ is invertible and satisfies the condition that the induced maps
  $\bar{\mathcal{A}}_i:\mathcal{L}_{i}/\mathcal{L}_{i+1}\to \mathcal{L}_{i-1}/\mathcal{L}_{i}$ are isomorphisms on $U \cap M_q^{-1}(U)$. \\
  An $(SL(r+1),q)$-$oper$ is a $(GL(r+1),q)$-oper with the condition that $det(\mathcal{A})=1$ on $U \cap M_q^{-1}(U)$.
\end{Def}

The equivalence of Definitions \ref{qop} and \ref{qopflag} can be proven along the same lines as the equivalence of the analogous definitions in the case of classical opers. 
One can derive the second definition from the first by considering the associated bundle 
$E=(\cF_{SL(r+1)}\times V_{\omega_1})/SL(r+1)$, where $V_{\omega_1}$ in the defining representation of $G$. That immediately provides a flag of subbundles in $E$, preserved by $B_-$. From the chosen order in the Coxeter element we obtain that the induced   $q$-connection on $E$ locally has the form of the matrix with coefficients in $\mathbb{C}(z)$ so that it has zeroes above the superdiagonal. That immediately leads to conditions i), ii) of Definition \ref{qopflag}. 
This way q-connection $A$ induces q-connection $\mathcal{A}$ in the Definition \ref{qopflag}. 
Notice that  the second definition implies local formula (\ref{qop1}) in the defining representation and thus by faithfulness the first definition follows from the second. In the following we will use the same notation for $A$ and $\mathcal{A}$: it will be clear from the context which q-connection is used.

\subsection{Miura $(SL(r+1),q)$-opers}\label{Miurafinsec}
Miura  condition for the  the $q$-opers
corresponds to the introduction of an additional datum: reduction of the underlying
$SL(r+1)$-bundle to the Borel subgroup $B_+$ (opposite to $B_-$) that is
preserved by the oper $q$-connection.

\begin{Def}    \label{Miura}
  A {\em Miura $(SL(r+1),q)$-oper} on $\mathbb{P}^1$ is a quadruple
  $(\cF_{SL(r+1)},A,\cF_{B_-},\cF_{B_+})$, where $(\cF_{SL(r+1)},A,\cF_{B_-})$ is a
  meromorphic $(SL(r+1),q)$-oper on $\P^1$ and $\cF_{B_+}$ is a reduction of
  the $SL(r+1)$-bundle $\cF_{SL(r+1)}$ to $B_+$ that is preserved by the
  $q$-connection $A$.
\end{Def}

An equivalent definition using flag of subbundles can be obtained by using the explicit  identification of $G/B_+$ with the flag variety.

\begin{Def}    \label{Miuraflag}
  A {\em Miura $(SL(r+1),q)$-oper} on $\mathbb{P}^1$ is a quadruple
  $(E, A, \mathcal{L}_{\bullet}, \hat{\mathcal{L}}_{\bullet})$, where $(E, A, \mathcal{L}_{\bullet})$ is a
  meromorphic $(SL(r+1),q)$-oper on $\P^1$ and $\hat{\mathcal{L}}_{\bullet}=\{\mathcal{L}_i\}$  is another full flag 
  of subbundles in $E$ that is preserved by the
  $q$-connection $A$.
\end{Def}

Forgetting $\cF_{B_+}$, we associate a $(SL(r+1),q)$-oper to a given Miura
$(SL(r+1),q)$-oper. We will refer to it as the $(SL(r+1),q)$-oper underlying this
Miura $(SL(r+1),q)$-oper.

From a point of view of local consideration, let $U$ be a Zariski open dense subset on $\P^1$ as in Definition
  \ref{qop}. Choosing a trivialization $\imath_{B_-}$ of $\cF_{SL(r+1)}$ on $U
  \cap M_q^{-1}(U)$, we can write the $q$-connection $A$ in the form
  \eqref{qop1}. On the other hand, using the $B_+$-reduction
  $\cF_{B_+}$, we can choose another trivialization of $\cF_{SL(r+1)}$ on $U
  \cap M_q^{-1}(U)$ such that the $q$-connection $A$ acquires the form
  $\wt{A}(z) \in B_+(z)$. Hence there exists $g(z) \in SL(r+1)(z)$ such that
\begin{equation}    \label{connecting}
g(zq) n'(z)\prod_i (\phi_i(z)^{\check{\alpha}_i} \, s_i )n(z)
g(z)^{-1} = \wt{A}(z) \in B_+(z).
\end{equation}

Suppose we are given a principal $SL(r+1)$-bundle $\cF_{SL(r+1)}$ on any smooth
complex manifold $X$ equipped with reductions $\cF_{B_-}$ and
$\cF_{B_+}$ to $B_-$ and $B_+$, respectively. Then we assign to any
point $x \in X$ an element of the Weyl group $S_{r+1}$. Namely, the fiber
$\cF_{SL(r+1),x}$ of $\cF_{SL(r+1)}$ at $x$ is a $G$-torsor with reductions
$\cF_{B_-,x}$ and $\cF_{B_+,x}$ to $B_-$ and $B_+$,
respectively. Choose any trivialization of $\cF_{SL(r+1),x}$, i.e. an
isomorphism of $SL(r+1)$-torsors $\cF_{SL(r+1),x} \simeq SL(r+1)$. Under this
isomorphism, $\cF_{B_-,x}$ gets identified with $aB_- \subset SL(r+1)$ and
$\cF_{B_+,x}$ with $bB_+$. Then $a^{-1}b$ is a well-defined element of
the double quotient $B_-\backslash SL(r+1)/B_+$, which is in bijection with
$W_{SL(r+1)}$. Hence we obtain a well-defined element of $W_{SL(r+1)}=S_{r+1}$.

We will say that $\cF_{B_-}$ and $\cF_{B_+}$ have a {\em generic
  relative position} at $x \in X$ if the element of $W_G$ assigned to
them at $x$ is equal to $1$ (this means that the corresponding element
$a^{-1}b$ belongs to the open dense Bruhat cell $B_- \cdot B_+ \subset
SL(r+1)$).

Using Bruhat decomposition:   $SL(r+1)(z) = \bigsqcup_{w \in W_{SL(r+1)}} B_+(z) w N_-(z)$, we claim that $g(z)$ from (\ref{connecting}) lies in the $w=1$ cell, namely: $
g(z) \in B_+(z) N_-(z).$

Using the notion of relative position, we can reformulate this local statement as the following theorem, 
which was proven in \cite{Frenkel:2020}: 

\begin{Thm}    \label{gen rel pos}
  For any Miura $(SL(r+1),q)$-oper on $\mathbb{P}^1$, there exists an open
  dense subset $V \subset \P^1$ such that the reductions $\cF_{B_-}$
  and $\cF_{B_+}$ are in generic relative position for all $x \in V$.
\end{Thm}

Returning back to the local expression (\ref{connecting}) we now wish to characterize the explicit representatives for $\tilde{A}(z)$.

\begin{Thm}    \label{gen elt}
 Every element of the set
$N_-(z)\prod_i\phi_i(z)^{\check{\alpha}_i}s_iN_- (z)\; \cap \; B_+(z)$ can be written
in the form
\begin{equation}    \label{gicheck}
\prod_i g_i^{\check{\alpha}_i}e^{\frac{ t_i(z)\phi_i(z)}{g_i}e_i}, \qquad
g_i \in \mathbb{C}^{\times}(z),
\end{equation}
where each $t_i(z) \in \mathbb{C}(z)$ is determined by the lifting $s_i$.
\end{Thm}

This fact was proven in higher generality in \cite{Frenkel:2020}. 
Note that in the case of $SL(r+1)$ for a given order of $s_i$ this follows directly from the matrix realization. 

From now on we consider the liftings $s_i$ of simple reflections $w_i\in W$  in such a way that $t_i=1$ for $(i=1,\dots,r)$.

\section{Z-twisted Miura $(SL(r+1),q)$-opers }\label{Sec:ZtwistedMiura}
\subsection{Z-twisted (Miura) opers }
In this paper we consider a class of (Miura) $q$-opers that are gauge
equivalent to a constant element of $SL(r+1)$ (as $(SL(r+1),q)$-connections). Moreover,  we assume that such element 
$Z$ be the regular element of the maximal torus $H$. One can express it as follows 
\begin{equation}    \label{Z}
Z = \prod_{i=1}^r \zeta_i^{\check\alpha_i}, \qquad \zeta_i \in
\C^\times.
\end{equation}

\begin{Def}    \label{Ztwoper}
  A {\em $Z$-twisted $(SL(r+1),q)$-oper} on $\mathbb{P}^1$ is a $(SL(r+1),q)$-oper
  that is equivalent to the constant element $Z \in H \subset H(z)$
  under the $q$-gauge action of $SL(r+1)(z)$, i.e. if $A(z)$ is the
  meromorphic oper $q$-connection (with respect to a particular
  trivialization of the underlying bundle), there exists $g(z) \in
  G(z)$ such that
\begin{eqnarray}    \label{Ag}
A(z)=g(qz)Z g(z)^{-1}.
\end{eqnarray}
A {\em $Z$-twisted Miura $(SL(r+1),q)$-oper} is a Miura $(SL(r+1),q)$-oper on
$\mathbb{P}^1$ that is equivalent to the constant element $Z \in H
\subset H(z)$ under the $q$-gauge action of $B_+(z)$, i.e.
\begin{eqnarray}    \label{gaugeA}
A(z)=v(qz)Z v(z)^{-1}, \qquad v(z) \in B_+(z).
\end{eqnarray}
\end{Def}

It follows from Definition \ref{Ztwoper} that any $Z$-twisted
$(SL(r+1),q)$-oper is also $Z'$-twisted for any $Z'$ in the $S_{r+1}$-orbit of
$Z$. But if we endow it with the structure of a $Z$-twisted Miura
$(SL(r+1),q)$-oper (by adding a $B_+$-reduction $\cF_{B_+}$ preserved by the
oper $q$-connection), then we fix a specific element in this
$S_{r+1}$-orbit.

Thus we have the following Proposition, which allows to characterize $Z$-twisted Miura q-opers 
associated to $Z$-twisted q-opers.

\begin{Prop}    \label{Z prime}
  Let $Z \in H$ be regular. For any $Z$-twisted $(SL(r+1),q)$-oper $(\cF_{SL(r+1)},A,\cF_{B_-})$
  and any choice of $B_+$-reduction $\cF_{B_+}$ of $\cF_{SL(r+1)}$ preserved
  by the oper $q$-connection $A$, the resulting Miura $(SL(r+1),q)$-oper is
  $Z'$-twisted for a particular $Z' \in S_{r+1} \cdot Z$. The set of 
  $A$-invariant $B_+$-reductions $\cF_{B_+}$ on the 
  $(SL(r+1),q)$-oper is in one-to-one correspondence with the elements of $W$.
\end{Prop}

\subsection{(Miura) q-opers with regular singularities}\label{Miurareg}

Let $\{ \Lambda_i(z) \}_{i=1,\ldots,N-1}$ be a collection of
non-constant polynomials.

\begin{Def}    \label{d:regsing}
  A $(SL(r+1),q)$-{\em oper with regular singularities determined by $\{
    \Lambda_i(z) \}_{i=1,\ldots,r}$} is a $q$-oper on $\P^1$ whose
  $q$-connection \eqref{qop1} may be written in the form
\begin{equation}    \label{Lambda}
A(z)= n'(z)\prod_i(\Lambda_i(z)^{\check{\alpha}_i} \, s_i)n(z), \qquad
n(z), n'(z)\in N_-(z).
\end{equation}

  {\em A Miura $(SL(r+1),q)$-oper with regular singularities determined by
polynomials $\{ \Lambda_i(z) \}_{i=1,\ldots,r}$} is a Miura
  $(SL(r+1),q)$-oper such that the underlying $q$-oper has
regular singularities determined by $\{ \Lambda_i(z)
\}_{i=1,\ldots,r}$.
\end{Def}

The following theorem follows from Theorem \ref{gen elt} and gives an explicit parameterization of generic elements from the space of Miura opers.

\begin{Thm}    \label{Miura form}
For every Miura $(SL(r+1),q)$-oper with regular singularities determined by
the polynomials $\{ \Lambda_i(z) \}_{i=1,\ldots,r}$, the underlying
$q$-connection can be written in the form 
\begin{equation}    \label{form of A}
A(z)=\prod_i
g_i(z)^{\check{\alpha}_i} \; e^{\frac{\Lambda_i(z)}{g_i(z)}e_i}, \qquad
g_i(z) \in \C(z)^\times.
\end{equation}
\end{Thm}

\subsection{Cartan connections} Consider a Miura $(SL(r+1),q)$-oper. By Corollary
\ref{Miura form}, the underlying $(SL(r+1),q)$-connection can be written in the
form \eqref{form of A}. Since it preserves the $B_+$-bundle
$\cF_{B_+}$ underlying this Miura $(SL(r+1),q)$-oper (see Definition
\ref{Miura}), it may be viewed as a meromorphic $(B_+,q)$-connection
on $\P^1$. Taking the quotient of $\cF_{B_+}$ by $N_+ = [B_+,B_+]$ and
using the fact that $B/N_+ \simeq H$, we obtain an $H$-bundle
$\cF_{B_+}/N_+$ and the corresponding $(H,q)$-connection, which we
denote by $A^H(z)$. According to formula \eqref{form of A}, it is
given by the formula
\begin{equation}    \label{AH}
A^H(z)=\prod_ig_i(z)^{\check{\alpha}_i}.
\end{equation}
We call $A^H(z)$ the \emph{associated Cartan $q$--connection} of the
Miura $q$-oper $A(z)$.

Now, if our Miura $q$-oper is $Z$-twisted (see Definition
\ref{Ztwoper}), then we also have $A(z)=v(qz)Z v(z)^{-1}$, where
$v(z)\in B_+(z)$.  Since $v(z)$ can be written as
\begin{equation}    \label{vz}
v(z)=
\prod_i y_i(z)^{\check{\alpha}_i} n(z), \qquad n(z)\in N_+(z), \quad
y_i(z) \in \C(z)^\times,
\end{equation}
the Cartan $q$-connection $A^H(z)$ has the form
\begin{equation}    \label{AH1}
A^H(z)=\prod_i
y_i(qz)^{\check{\alpha}_i} \; Z \; \prod_i y_i(z)^{-\check{\alpha}_i}
\end{equation}
and hence we will refer to $A^H(z)$ as $Z$-{\em twisted Cartan
  $q$-connection}. This formula shows that $A^H(z)$ is completely
determined by $Z$ and the rational functions $y_i(z)$. Indeed,
comparing this equation with \eqref{AH} gives
\begin{equation}    \label{giyi}
g_i(z)=\zeta_i\frac{y_i(qz)}{y_i(z)}\,.
\end{equation}

We note that $A^H(z)$ determines the $y_i(z)$'s uniquely
up to scalar.  Indeed, if there is another choice $\tilde{y}_i(z)$, we obtain that $h_i(z)=y_i(z)\tilde{y}^{-1}_i(z)$ satisfies the equation $h_i(qz)=h_i(z)$. Given the condition that $q$ is not a root of unity, $h_i(z)$ has to be constant.

\section{Miura-Pl\"ucker $(SL(r+1),q)$-opers }  \label{Sec:MiuraPlucker}
In this section, we will talk about the notion of nondegeneracy and we will relax the Z-twisted condition slightly (we refer to \cite{Frenkel:2020} for details). We will associate to the given $(SL(r+1),q)$-Miura oper a collection of $(GL(2),q)$-opers and require that all of them are Z-twisted with some nondegeneracy conditions.  This will lead to the notion of $Z$-twisted Miura-Pl\"ucker q-opers. It turns out, as we will find out in the Section 5, that these objects, supplied by nondegeneracy condition are in one-to-one correspondence with solutions of certain equations called QQ-system, which are in turn related to Bethe Ansatz equations.  Also, in the next section we will show that for $SL(r+1)$ this relaxed $Z$-twisted Miura-Pl\"ucker condition is equivalent to the original $Z$-twisted condition.

\subsection{The associated Miura $(GL(2),q)$-opers and Miura-Pl\"ucker condition} \label{gl2op}
Let $V_i$ be the irreducible representation of $SL(r+1)$ with the highest weight $\omega_i$. 
Notice, that the one-dimensional and two-dimensional subspaces $L_i$ and $W_i$ of
$V_i$ spanned by the weight vectors $\nu_{\omega_i}$ (the highest weight vector),  and $\nu_{\omega_i}$, $f_i \cdot\nu_{\omega_i}$ are a $B_+$-invariant subspaces of $V_i$.

Now let $(\cF_{SL(r+1)},A,\cF_{B_-},\cF_{B_+})$ be a Miura $(SL(r+1),q)$-oper
with regular singularities determined by 
polynomials $\{ \Lambda_i(z)
\}_{i=1,\ldots,r}$ (see Definition \ref{d:regsing}). Recall that
$\cF_{B_+}$ is a $B_+$-reduction of a $G$-bundle $\cF_{SL(r+1)}$ on $\P^1$
preserved by the $(SL(r+1),q)$-connection $A$. Therefore for each
$i=1,\ldots,r$, the vector bundle
$$
\cV_i = \cF_{B_+} \underset{B_+}\times V_i = \cF_{SL(r+1)} \underset{SL(r+1)}\times
V_i
$$
associated to $V_i$ contains a rank two
subbundle
$$
\cW_i = \cF_{B_+} \underset{B_+}\times W_i
$$
associated to $W_i \subset V_i$, and $\cW_i$ in
turn contains a line subbundle
$$
{\mathscr L}_i = \cF_{B_+} \underset{B_+}\times L_i
$$
associated to $L_i \subset W_i$.

Denote by $\phi_i(A)$ the $q$-connection on the vector bundle $\cV_i$
corresponding to the
above Miura $q$-oper connection $A$. Since, by definition $A$ preserves $\cF_{B_+}$, we obtain that $\phi_i(A)$ preserves the
subbundles $\mathscr{L}_i$ and $\cW_i$ of $\cV_i$ and thus produces $(GL(2),q)$-oper on $\cW_i$. 
Let us denote such q-oper by $A_i$.

If we trivialize $\cF_{B_+}$ on a Zariski open subset of $\P^1$ so
that $A(z)$ has the form \eqref{form of A} with respect to this
trivialization (see Corollary \ref{Miura form}). This trivializes the
bundles $\cV_i$, $\cW_i$, and $\mathscr L_i$, so that the $q$-connection
$A_i(z)$ becomes a $2 \times 2$ matrix whose entries are in $\C(z)$. Moreover, $\cW_i$ decomposes into direct sum of two subbundles, $\hat{\mathscr{L}}_i$, preserved by $B_+$ and $\mathscr{L}_i$ with respect to which it satisfies the (GL(2),q)-oper condition. We can unify all that in the following Proposition.

\begin{Prop} \cite{Frenkel:2020}   \label{2flagthm}
The quadruple $(A_i, \cW_i, \mathscr{L}_i, \hat{\mathscr{L}}_i)$ forms a $(GL(2),q)$ Miura oper, so that explicitly: 
\begin{equation}    \label{2flagformula}
A_i(z)=\begin{pmatrix}
  g_i(z) &  &\Lambda_i(z) \prod_{j>i} g_j(z)^{-a_{ji}}\\
&&\\  
  0 & &g^{-1}_i(z) \prod_{j\neq i} g_j(z)^{-a_{ji}}
 \end{pmatrix},
\end{equation}
where we use the ordering of the simple roots determined by the
Coxeter element $c$.
\end{Prop}

Now we impose the condition (\ref{AH1}) on the corresponding $A^H$ connection, namely
$$g_i=\zeta_i\frac{y_i(qz)}{y_i(z)}.$$

Let $G_i\cong \SL(2)$ be the subgroup of $SL(r+1)$ corresponding to the
 $\mathfrak{sl}(2)$-triple spanned by $\{e_i, f_i,
 \check{\alpha}_i\}$, which preserves $W_i$. 
Performing the gauge transformation via diagonal matrix for \ref{2flagformula}, we can represent the resulting connection as follows:
\begin{align}    \notag
\wt{A}_i(z) = u(qz)A_i(z)u^{-1}(z)&= \begin{pmatrix} 1 & 0 \\ 0 & \prod_{j\ne i}
  \zeta_j^{-a_{ji}}
\end{pmatrix} {\mc A}_i(z) \\    \label{Ai}
&= \begin{pmatrix} 1 & 0 \\ 0 & \prod_{j\ne i}
  \zeta_j^{-a_{ji}}
\end{pmatrix} g_i^{\check \alpha_i}(z) \;
e^{\frac{\rho_i(z)}{g_i(z)}e_i}.
\end{align}
where, where
 \begin{equation}\label{ri}
 \rho_i(z)=\Lambda_i(z)\prod_{j> i} (\zeta_j
 y_j(qz))^{-a_{ji}}\prod_{j<i}y_j(z)^{-a_{ji}}.
 \end{equation}
 
Thus, under the assumption \eqref{AH1}, our Miura $(SL(r+1),q)$-oper $A(z)$
gives rise to a collection of meromorphic Miura $(\SL(2), q)$-opers
${\mc A}_i(z)$ for $i=1,\ldots,N-1$. It should be noted that it has
regular singularities in the sense of Definition~\ref{d:regsing} if
and only if $\rho_i(z)$ is a polynomial. For example, this holds for
all $i$ if all $y_j(z), j=1,\dots, N-1$, are polynomials, an observation
we will use below.

Now we are ready to relax $Z$-twisted condition as follows.

\begin{Def}\cite{Frenkel:2020}    \label{ZtwMP}
  A $Z$-{\em twisted Miura-Pl\"ucker $(SL(r+1),q)$-oper} is a meromorphic
  Miura $(SL(r+1),q)$-oper on $\P^1$ with the underlying $q$-connection
  $A(z)$, such that there exists $v(z) \in B_+(z)$ such that for all
  $i=1,\ldots,r$, the Miura $(\GL(2),q)$-opers $A_i(z)$ associated to
  $A(z)$ by formula \eqref{2flagformula} can be written in the form:
\begin{equation}    \label{gaugeA3}
A_i(z) = v(zq) Z v(z)^{-1}|_{W_i} = v_i(zq)Z_iv_i(z)^{-1}
\end{equation}
where $v_i(z) = v(z)|_{W_i}$ and $Z_i = Z|_{W_i}$.
\end{Def}

Note, that it follows from the above definition that the
$(H,q)$-connection $A^H(z)$ associated to a $Z$-twisted
Miura-Pl\"ucker $(SL(r+1),q)$-oper can be written in the same form
\eqref{AH1} as the $(H,q)$-connection associated to a $Z$-twisted
Miura $(SL(r+1),q)$-oper.

However, while it is true that every $Z$-twisted Miura $(SL(r+1),q)$-oper is automatically a
$Z$-twisted Miura-Pl\"ucker $(SL(r+1),q)$-oper, but the converse is not
necessarily true if $r \neq 1$.

\subsection{Nondegeneracy conditions}\label{nondegsl}
In what follows, we will say that $v, w \in \mathbb{C}^\times$ are
\emph{$q$-distinct} if $q^\Z v\cap q^\Z w=\varnothing$.

In this subsection  we introduce two nondegeneracy conditions for $Z$-twisted
Miura-Pl\"ucker $q$-opers. The first of them, called the
$H$-nondegeneracy condition, is applicable to arbitrary Miura
$q$-opers with regular singularities. Recall from Corollary \ref{Miura
  form} that the underlying $q$-connection can be represented in the
form \eqref{form of A}.

\begin{Def}\cite{Frenkel:2020}    \label{nondeg Cartan}
  A Miura $(SL(r+1),q)$-oper $A(z)$ of the form \eqref{form of A} is called
  $H$-\emph{nondegene\-rate} if the corresponding $(H,q)$-connection
  $A^H(z)$ can be written in the form \eqref{AH1}, where for all
  $i,j,k$ with $i\ne j$ and $a_{ik} \neq 0, a_{jk} \neq
    0$, the zeros and poles of $y_i(z)$ and $y_j(z)$ are
  $q$-distinct from each other and from the zeros of
  $\Lambda_k(z)$.
\end{Def}

Next, we define the second nondegeneracy condition. This condition
applies to $Z$-twisted Miura-Pl\"ucker $(SL(r+1),q)$-opers.  Firstly, we start from $(SL(2),q)$-opers.

Consider a Miura $(\SL(2),q)$-oper given by formula \eqref{form of A},
which reads in this case:
$$
A(z)=g(z)^{\check{\alpha}} \; \on{exp}\left( \frac{\Lambda(z)}{g(z)}e
\right) 
$$
so that the corresponding Cartan
$q$-connection $A^H(z)$ is 
$
A^H(z)=g(z)^{\check{\alpha}} 
$, 
where $y(z)$ is a rational function. Let us assume that $A(z)$ is
$H$-nondegenerate (see Definition \ref{nondeg Cartan}). This means
that the zeros of $\Lambda(z)$ are $q$-distinct from the zeros and
poles of $y(z)$.

If we apply to $A(z)$ a $q$-gauge transformation by an element of
$h(z)^{\check\alpha} \in H[z]$, we obtain a new $q$-oper connection
\begin{equation}    \label{wtA}
\wt{A}(z) = \wt{g}(z)^{\check{\alpha}} \; \on{exp}\left(
  \frac{\wt\Lambda(z)}{\wt{g}(z)}e \right),
\end{equation}
where
$\wt{g}(z) = g(z) h(zq) h(z)^{-1}$,  $\wt\Lambda(z) =
\Lambda(z) h(zq) h(z)$. 
It also has regular singularities, but for a
different polynomial $\wt{\Lambda}(z)$, and $\wt{A}(z)$ may no longer
be $H$-nondegenerate.  However, it turns out there is an essentially
unique gauge transformation from $H[z]$ for which the resulting
$\wt{A}(z)$ is $H$-nondegenerate $\wt{A}^H(z)$ and $\wt{y}(z)$
is a polynomial.  This choice allows us to fix the polynomial
$\Lambda(z)$ determining the regular singularities of our
$(SL(2),q)$-oper.

\begin{Lem}  \cite{Frenkel:2020}  \label{nondegsl2}
\begin{enumerate}
\item There is an $(\SL(2),q)$-oper $\wt{A}(z)$ in the $H[z]$-gauge
  class of $A(z)$ for which $\wt{A}^H(z)=\wt{g}(z)^{\check{\alpha}}$ is
  nondegenerate and the rational function $\wt{y}(z)$ is a
  polynomial. This oper is unique up to a scalar $a \in \C^\times$ that
  leaves $\wt{g}(z)$ unchanged but multiplies $\wt{y}(z)$ and
  $\wt{\Lambda}(z)$ by $a$ and $a^2$, respectively.
\item This $(\SL(2),q)$-oper $\wt{A}(z)$ may also be characterized by
  the property that $\wt{\Lambda}(z)$ has maximal degree subject to
  the constraint that it is $H$-nondegenerate.
\end{enumerate}
\end{Lem}

This motivates the following definition.

\begin{Def} \cite{Frenkel:2020}   \label{ngsl2}
  A $Z$-twisted Miura $(\SL(2),q)$-oper is called \emph{nondegenerate}
  if it is $H$-nondegenerate and the rational function $y(z)$
  appearing in formula \eqref{AH1} is a polynomial.
\end{Def}

We now turn to the general case. Recall Definition \ref{ZtwMP} of
$Z$-twisted Miura-Pl\"ucker $(SL(r+1),q)$-opers.

\begin{Def} \cite{Frenkel:2020}    \label{nondeg Miura}
  Suppose that $r>1$. A $Z$-twisted
  Miura-Pl\"ucker $(SL(r+1),q)$-oper $A(z)$ is called \emph{nondegenerate} if
  its associated Cartan $q$-connection $A^H(z)$ is nondegenerate and
  each associated $Z_i$-twisted Miura $(\SL(2),q)$-oper $\mathcal{A}_i(z)$ is
  nondegenerate.
\end{Def}

It turns out that this simply means that in addition to $A^H(z)$ being
nondegenerate, each $y_i(z)$ from formula \eqref{AH1} is a polynomial. Here we provide the complete proof, since we will need it for infinite-dimensional case.

\begin{Prop} \cite{Frenkel:2020}    \label{nondeg1}
  Suppose that  $r>1$, and let $A(z)$ be a
  $Z$-twisted Miura-Pl\"ucker $(SL(r+1),q)$-oper.
  The following statements are equivalent:
\begin{enumerate}
\item\label{nondegen1} $A(z)$ is nondegenerate.
  \item\label{nondegen2} The Cartan $q$-connection $A^H(z)$ is
    nondegenerate, and each
    $A_i(z)$ has regular singularities, i.e. $\rho_i(z)$ given
    by formula \eqref{ri} is in $\C[z]$.
    \item\label{nondegen3} Each $y_i(z)$ from formula \eqref{AH1} is a
      polynomial, and for all $i,j,k$ with $i\ne
      j$ and $a_{ik} \neq 0, a_{jk} \neq
    0$, the zeros of $y_i(z)$ and $y_j(z)$ are
  $q$-distinct from each other and from the zeros of
  $\Lambda_k(z)$.
\end{enumerate}
\end{Prop}

\begin{proof} To prove that \eqref{nondegen2} implies
  \eqref{nondegen3}, we need only show that if each $\rho_i(z)$ given
  by formula \eqref{ri} is in $\C[z]$, then the $y_i(z)$'s are
  polynomials.  Suppose $y_i(z)$ is not a polynomial, and choose $j\ne
  i$ such that $a_{ij} \neq 0$.  Then $-a_{ij}>0$ and so the
  denominator of $y_i(z)$ or $y_i(qz)$ appears in the denominator of
  $\rho_j(z)$.  Moreover, since the poles of $y_i(z)$ are $q$-distinct
  from the zeros of $\Lambda_j(z)$ and the other $y_k(z)$'s, the poles
  of $y_i(z)$ or $y_i(qz)$ would give rise to poles of
  $\rho_j(z)$. But then $A_j(z)$ would not have regular singularities.

  Next, assume \eqref{nondegen3}. Then $A^H(z)$ is nondegenerate by
  Definition \ref{nondeg Cartan}.  Since all the $y_i(z)$'s are
  polynomials, the same if true for the $\rho_i(z)$'s.  (Here we
  are using the fact that the off-diagonal elements of the Cartan
  matrix, $a_{ij}$ with $i\neq j$, are less than or equal to 0.)
  Since $\rho_i(z)$ is a product of polynomials whose roots are
  $q$-distinct from the roots of $y_i(z)$, we see that the Cartan
  $q$-connection associated to $A_i(z)$ is nondegenerate.

Finally, \eqref{nondegen2} is a trivial consequence of
\eqref{nondegen1}.
\end{proof}

If we apply a $q$-gauge transformation by an element $h(z)\in H[z]$ to
$A(z)$, we get a new $Z$-twisted Miura-Pl\"ucker $(SL(r+1),q)$-oper.
However, the following proposition shows that it is only nondegenerate
if $h(z)\in H$ is constant with respect to $z$.  As a consequence, the $\Lambda_k$'s of a
nondegenerate $q$-oper are determined up to scalar multiples.

\begin{Prop} \cite{Frenkel:2020} If $A(z)$ is a nondegenerate $Z$-twisted Miura-Pl\"ucker \\
$(SL(r+1),q)$-oper and $h(z)\in H[z]$, then $h(qz)A(z)h(z)^{-1}$ is
  nondegenerate if and only if $h(z)$ is a constant element of $H$.
\end{Prop}

\section{Z-twisted Miura $(SL(r+1),q)$-opers and $QQ$-systems}\label{Sec:MiuraSL}

\subsection{$QQ$-systems and Miura-Pl\"ucker  $(SL(r+1),q)$-opers}

One of the main  results of previous section was the explicit structure  of  the  non-degenerate  Miura-Pl\"ucker  $(SL(r+1),q)$-oper   
with regular singularities defined by $\{\Lambda_i(z)\}_{i=(1,\dots, r)}$ and associated with regular element $Z=\prod_i\zeta_i^{\check{\alpha}_i}$. Following Proposition \ref{nondeg1} the local expression, namely $A(z)$ can be expressed as follows:
\begin{equation}\label{form of A1}
A(z)=\prod_i
g_i(z)^{\check{\alpha}_i} \; e^{\frac{\Lambda_i(z)}{g_i(z)}e_i}, \qquad
g_i(z)=\zeta_i\frac{Q_i^+(qz)}{Q_i^+(z)}\,.
\end{equation}
where $Q_i^{+}(z)$ are monic polynomials (here we changed the notation $y_i(z)\equiv Q_i^{+}(z)$). 
From now on, we will assume that $Z$ satisfies the following property:
\begin{equation}    \label{assume}
\prod_{i=1}^r \zeta_i^{a_{ij}}=\frac{\zeta_j^2}{\zeta_{j+1}\zeta_{j-1}} \notin q^\Z, \qquad
\forall j=1,\ldots,r\,,
\end{equation}
where $a_{ij}$ are matrix elements of the Cartan matrix for $\mathfrak{sl}_{r+1}$.
Since $\prod_{i=1}^r \zeta_i^{a_{ij}}\ne 1$ is a special case of
\eqref{assume}, this implies that $Z$ is {\em regular semisimple}.

\subsection{The $SL(r+1)$ $QQ$-system}
In \cite{Frenkel:2020} the following statement was proven (we specialize that result to the case of $SL(r+1)$):
\begin{Thm}    \label{inj}
  There is a one-to-one correspondence between the set of
  nondegenerate $Z$-twisted Miura-Pl\"ucker $(SL(r+1),q)$-opers and the set
  of nondegenerate polynomial solutions of the $QQ$-system 
\begin{equation}\label{eq:QQAtype}
\xi_{i} Q^+_i(q z) Q^-_i(z) - \xi_{i+1} Q^+_i(z) Q^-_i(qz) = \Lambda_i (z) Q^+_{i-1}(z)Q^+_{i+1}(qz)\,, \qquad i = 1,\dots, r
\end{equation}
subject to the boundary conditions $Q^\pm_{0}(z)=Q^\pm_{r+1}(z)=1$ and $\xi_0=\xi_{r+2}=1$ so that  
$$
\xi_1=\zeta_1,\quad \xi_2= \frac{\zeta_2}{\zeta_1},\quad \dots \quad \xi_r=\frac{\zeta_{r}}{\zeta_{r-1}},\quad \xi_{r+1}= \frac{1}{\zeta_r}\,.
$$
\end{Thm}
Note, that $\xi_i$ is the $i$th element on the diagonal of $Z$ from \eqref{Z}.

We will say that a polynomial solution $\{ Q_i^+(z),Q_i^-(z)
\}_{i=1,\ldots,r}$ of \eqref{eq:QQAtype} is {\em nondegenerate} if the following conditions are satisfied: relation \eqref{assume} holds;
for $i\neq j$ the zeros of $Q^+_i(z)$ and $Q^-_{j}(z)$ are
$q$-distinct from each other and from the zeros of $\Lambda_{k}(z)$ for $|i-k|=1,\,|j-k|=1$.

For the convenience we will rewrite \eqref{eq:QQAtype} as follows:
\begin{equation}\label{eq:QQgamma}
\xi_{i} \phi_i(z)- \xi_{i+1} \phi_i(qz) = \rho_i(z)\,,
\end{equation}
where 
\begin{equation}
\phi_i(z)=  \frac{Q^-_i(z)}{Q^+_i(z)} \,,\qquad \rho_i(z)= \Lambda_i (z)\frac{Q^+_{i-1}(qz)Q^+_{i+1}(z)}{Q^+_i(z)Q^+_i(qz)} \,.
\end{equation}

\subsection{Extended $QQ$-system and $Z$-twisted $(SL(r+1),q)$-opers}
As it was demonstrated in \cite{Frenkel:2020} for a simply-connected simple Lie group $G$ the set of nondegenerate $Z$-twisted Miura-Pl\"ucker q-opers includes as a subset the set of $Z$-twisted Miura $(G,q)$-opers.
The opposite inclusion was possible provided that $Z$-twisted Miura-Pl\"ucker q-opers are in addition $w_0$-generic (see Theorem 7.10). We will discuss this notion in detail later, in subsection 5.6. 

In this section we shall demonstrate that when $G$ is a special linear group then we do not need this extra condition and that the corresponding  $Z$-twisted Miura-Pl\"ucker $(SL(r+1),q)$-oper will be $Z$-twisted Miura q-oper, namely there exists $v(z) \in B_+(z)$, such that the q-connection $A(z)$ reduces to an element of the form \eqref{Z}, or, equivalently
\begin{equation}\label{eq:qGaugeTrSpecial}
v(qz)^{-1}A(z)= Zv(z)^{-1}\,.
\end{equation}
Moreover, we will construct explicit expression for $v(z)$.

The following statement is a generalization of the result of \cite{Mukhin_2005} to $Z$-twisted q-opers.

\begin{Thm}\label{th:SLNqMiura}
Let $A(z)$ be as in \eqref{form of A1} and $Z$ as in \eqref{Z}. 
Suppose $Q^-_{i,i+1,\dots, j}(z)$ ( $i,j\in\mathbb{Z}$, $i<j$) are polynomials, satisfying equations:
\begin{align}\label{eq:QQAll}
\xi_{i} \,\phi_i(z)-\xi_{i+1}\,  \phi_i(qz)&= \rho_i(z)\,,\qquad \qquad &i=1,\dots,r\notag\\
\xi_{i} \,\phi_{i,i+1}(z)-\xi_{i+2}\,  \phi_{i, i+1}(qz)&=\rho_{i+1}(z)\phi_i(qz)\,,\qquad \qquad &i=1,\dots,r-1\notag\\
\dots&\dots\\
\xi_{i} \,\phi_{i,\dots,r-2+i}(z)-\xi_{r+i-1}\,  \phi_{i,\dots,r-2+i}(qz)&=\rho_{r-1}(z)\phi_{i,\dots,r-3+i}(qz)\,,\qquad \qquad &i=1,2\notag\\
\xi_{1} \phi_{1,\dots,r}(z)-\xi_{r+1}  \phi_{1,\dots,r}(qz)&= \rho_r(z)\phi_{1,\dots,r-1}(qz)\,,\qquad \qquad &\notag
\end{align}
where for all $j>i$
\begin{equation}\label{eq:phimdef}
\phi_{i,\dots,j}(z)=\frac{Q^-_{i,\dots,j}(z)}{Q^+_j(z)}.
\end{equation}
Then there exist $v(z)\in B_+(z)$ such that \eqref{eq:qGaugeTrSpecial} holds and is given by
\begin{equation}\label{eq:qGaugeGen}
v(z)= \prod\limits_{i=1}^r Q^+_i(z)^{\check{\alpha}_i} \cdot \prod\limits_{i=1}^r V_i(z)\,,
\end{equation}
where 
\begin{equation}
V_i(z)= \prod\limits_{j=i}^r \exp\left(-\phi_{i,\dots,j}(z)\, e_{i,\dots,j}\right)\,, \quad e_{i,\dots,j}=[\dots[[e_i,e_{i+1}],e_{i+2}]\dots e_j]\,.
\end{equation}
\end{Thm}

We shall prove a more general statement in Section \ref{Sec:MiuraPluckerInf} about $(\overline{GL}(\infty),q)$ opers which will contain Theorem \ref{th:SLNqMiura} as a corollary. Here, to illustrate how the theorem works, we will regard some low rank examples.

Notice that although the expression for $v(z)$ in \eqref{eq:qGaugeGen} is rather complicated, the inverse $v(z)^{-1}$ can be succinctly presented as
\begin{equation}\label{eq:vinverse}
v(z)^{-1}=\displaystyle\begin{pmatrix}
\frac{1}{Q_1^+(z)} & \frac{Q^-_1(z)}{Q_2^+(z)} & \frac{Q^-_{12}(z)}{Q_3^+(z)} & \dots  & \frac{Q^-_{1,\dots,r-1}(z)}{Q_r^+(z)}  & Q^-_{1,\dots,r}(z) \\
0 & \frac{Q^+_1(z)}{Q_2^+(z)} & \frac{Q^-_2(z)}{Q_3^+(z)} & \dots & \frac{Q^-_{2,\dots,r-1}(z)}{Q_r^+(z)}  & Q^-_{2,\dots,r}(z)\\
0 & 0 & \frac{Q^+_2(z)}{Q_3^+(z)}  & \dots & \frac{Q^-_{3,\dots,r-1}(z)}{Q_r^+(z)} & Q^-_{3,\dots,r}(z)\\
\vdots & \vdots & \vdots & \ddots & \vdots  & \vdots\\
0 & \dots & \dots& \dots &  \frac{Q^+_{r-1}(z)}{Q_r^+(z)} & Q^-_r(z) \\
0 & \dots & \dots& \dots &  0 & Q^+_r(z)
\end{pmatrix}\,.
\end{equation}

Before we continue the following statement will be needed.
\begin{Lem}\label{Th:BCH}
The following relations hold for any $u,v\in\mathbb{C}$ and $i,j=1,\dots, r$
\begin{align}
u^{\check{\alpha}_i} e^{v e_j} &=
\exp\left({u^{a_{ji}}v\, e_i} \right)u^{\check{\alpha}_i}\,.
\end{align}
In general, if $[X,Y]=sY$ we have
\begin{equation}
u^X e^{v Y} = \exp(u^s v Y)u^X\,.
\end{equation}
\label{Th:LemmaMV05}
\end{Lem}

Using this Lemma we can rewrite the q-connection \eqref{form of A1} such that the roots of $SL(r+1)$ are placed in the decreasing order. 

\begin{Lem}\label{Th:BCHAConnection}
Let
\begin{equation}
\rho_i(z)=\Lambda_i(z)\frac{Q_{i-1}(qz)Q_{i+1}(z)}{Q_{i}(qz)Q_{i}(z)}\,.
\end{equation}
Then the $(SL(r+1),q)$-oper reads 
\begin{equation}\label{new form of A1}
A(z)=\prod_{i=r}^1 Q^+_i(qz)^{\check{\alpha}_i} \cdot \prod\limits_{i=r}^1 e^{\frac{\zeta_{i}}{\zeta_{i+1}}\rho_i(z)e_i} \cdot \prod_{i=r}^1 \zeta_i^{\check{\alpha}}Q^+_i(z)^{-\check{\alpha}_i}\,,
\end{equation}
or as a matrix
\begin{equation}\label{eq:MiuraqConnection}
A(z)=
\begin{pmatrix}
g_1(z) & \Lambda_1(z) & 0 & 0 & \dots &0& 0\\
0& \frac{g_2(z)}{g_1(z)} & \Lambda_2(z)  & 0 & \dots & 0&0\\
0& 0&  \frac{g_3(z)}{g_2(z)} & \Lambda_3(z)  & \dots & 0& 0\\
\vdots& \vdots & \cdots & \ddots & \ddots & \vdots & \vdots \\
\vdots& \vdots & \cdots & \dots & \ddots & \Lambda_{r-1}(z) & 0 \\
0& 0& 0& \dots &\dots &\frac{g_r(z)}{g_{r-1}(z)} &\Lambda_r(z)\\
0& 0& 0& \dots &\dots&0 &\frac{1}{g_r(z)}
\end{pmatrix}
\end{equation}
\end{Lem}

At this point the above choice of the order of simple roots may seem unsubstantiated, however, it will be justified in later sections, where we will consider $(\overline{GL(\infty)},q)$-opers.

\subsection{Examples}
\subsubsection{Miura $(SL(2),q)$-oper} 
The twist element $Z=\zeta^{\check{\alpha}}=\text{diag}(\zeta,\zeta^{-1})=\text{diag}(\xi_1,\xi_2)$
The $q$-connection \eqref{new form of A1} reads
\begin{equation}
A(z)= Q^+(qz)^{\check{\alpha}} \cdot e^{\zeta\rho(z)e} \cdot \zeta^{\check{\alpha}}Q^+(z)^{-\check{\alpha}}
=
\begin{pmatrix}
g(z)& \Lambda(z)\\
0 & g(z)^{-1}
\end{pmatrix}
\end{equation}
We look for the gauge transformation in the form
\begin{equation}
v(z)=Q^+(z)^{\check{\alpha}}e^{-\phi(z) e}\,,
\end{equation}
where $\phi(z)=\frac{Q^-(z)}{Q^+(z)}$.
The left hand side of \eqref{eq:qGaugeTrSpecial} reads
\begin{align}
v(qz)^{-1}A(z)&=e^{\phi(qz)e} e^{\zeta\rho(z)e} \cdot \zeta^{\check{\alpha}}Q^+(z)^{-\check{\alpha}} 
\end{align}
where $\rho(z)=\frac{\Lambda(z)}{Q^+(z) Q^+(qz)}$. Meanwhile, the right hand side of \eqref{eq:qGaugeTrSpecial} equals
\begin{align}
Z v(z)^{-1}&=\zeta^{\check{\alpha}}e^{-\phi(z) e}Q^+(z)^{-\check{\alpha}}=e^{-\zeta^2\phi(z) e}\zeta^{\check{\alpha}}Q^+(z)^{-\check{\alpha}}\,,
\end{align}
where we used Lemma \ref{Th:LemmaMV05} in the last step.
Comparing the above two expressions yields the desired $QQ$-system equation
\begin{equation}
\zeta \phi(z)- \zeta^{-1} \phi(qz)=\rho(z)\,,
\end{equation}
or equivalently as
\begin{equation}
\xi_1 \phi(z)- \xi_2 \phi(qz)=\rho(z)\,.
\end{equation}

\subsubsection{Miura $(SL(3),q)$-oper}
Consider $Z=\zeta_1^{\check{\alpha}_1}\zeta_2^{\check{\alpha}_2}=\text{diag}(\zeta_1,\frac{\zeta_2}{\zeta_1},\frac{1}{\zeta_2})=\text{diag}(\xi_1,\xi_2,\xi_3)$. 
The q-connection is given by
\begin{align}
A(z)&= Q^+_1(qz)^{\check{\alpha}_1} Q^+_2(qz)^{\check{\alpha}_2} \cdot  e^{\zeta_2\rho_2(z)e_2} e^{\frac{\zeta_1\rho_1(z)}{\zeta_2}e_1}    \cdot \zeta_1^{\check{\alpha}_1}\zeta_2^{\check{\alpha}_2} Q^+_1(z)^{-\check{\alpha}_1} Q^+_2(z)^{-\check{\alpha}_2}\notag\\
&=
\begin{pmatrix}
g_1(z) & \Lambda_1(z) & 0\\
0 & \frac{g_2(z)}{g_1(z)} & \Lambda_2(z) \\
0 & 0& \frac{1}{g_2(z)}
\end{pmatrix}
\end{align}
while the gauge transformation reads
\begin{equation}
v(z)= Q^+_1(z)^{\check{\alpha}_1} Q^+_2(z)^{\check{\alpha}_2}e^{-\phi_1(z)e_1} e^{-\phi_{12}(z)[e_1,e_2]} e^{-\phi_2(z)e_2}
\end{equation}
Thus the left hand side of \eqref{eq:qGaugeTrSpecial} becomes
\begin{align}\label{eq:sl3lhs}
v&(qz)^{-1}A(z)= e^{\phi_2(qz)e_2} e^{\phi_{12}(qz)[e_1,e_2]} e^{\phi_1(qz)e_1}\cdot e^{\zeta_2\rho_2(z)e_2} e^{\frac{\zeta_1\rho_1(z)}{\zeta_2}e_1} \cdot \zeta_1^{\check{\alpha}_1}\zeta_2^{\check{\alpha}_2}Q^+_1(z)^{-\check{\alpha}_1} Q^+_2(z)^{-\check{\alpha}_2}\notag\\
&=e^{(\phi_2(qz)+\zeta_2\rho_2)e_2} e^{(\phi_{12}(qz)+\zeta_2\rho_2(z)\phi_1(qz))[e_1,e_2]} e^{(\phi_1(qz)+\frac{\zeta_1}{\zeta_2}\rho_1(z))e_1}\cdot \zeta_1^{\check{\alpha}_1}\zeta_2^{\check{\alpha}_2}Q^+_1(z)^{-\check{\alpha}_1} Q^+_2(z)^{-\check{\alpha}_2}
\end{align}
Meanwhile the right hand side of \eqref{eq:qGaugeTrSpecial} equals
\begin{equation}
Zv(z)^{-1}= \zeta_1^{\check{\alpha}_1}\zeta_2^{\check{\alpha}_2} \cdot
e^{\phi_2(qz)e_2} e^{\phi_{12}(qz)[e_1,e_2]} e^{\phi_1(qz)e_1} 
Q^+_1(z)^{-\check{\alpha}_1} Q^+_2(z)^{-\check{\alpha}_2}
 \end{equation}
Now we need to move all Cartan elements from the front of the above expression to its rear using Lemma \ref{Th:LemmaMV05}
\begin{align}\label{eq:sl3rhs}
Zv(z)^{-1}=  e^{\frac{\zeta_2^2\phi_2(z)}{\zeta_1}e_2}e^{\zeta_1 \zeta_2\phi_{12}(z)[e_1,e_2]}e^{\frac{\zeta_1^2\phi_1(z)}{\zeta_2}e_1} \cdot \zeta_1^{\check{\alpha}_1}\zeta_2^{\check{\alpha}_2} Q^+_1(z)^{-\check{\alpha}_1} Q^+_2(z)^{-\check{\alpha}_2}
\end{align}
By comparing \eqref{eq:sl3lhs} with \eqref{eq:sl3rhs} we get
\begin{align}
\zeta_1 \phi_1(z)-\frac{\zeta_2}{\zeta_1}  \phi_1(qz)&= \rho_1(z)\,,\notag\\
\frac{\zeta_2}{\zeta_1} \phi_2(z)-\frac{1}{\zeta_2}  \phi_2(qz)&=\rho_2(z)\,,\notag\\
\zeta_1 \phi_{12}(z)-\frac{1}{\zeta_2}  \phi_{12}(qz)&=\rho_2(z)\phi_1(q z)\,,
\end{align}
or equivalently, 
\begin{align}\label{eq:SL3QQ}
\xi_1 \phi_1(z)-\xi_2  \phi_1(qz)&= \rho_1(z)\,,\notag\\
\xi_2 \phi_2(z)-\xi_3  \phi_2(qz)&=\rho_2(z)\,,\notag\\
\xi_1 \phi_{12}(z)-\xi_3  \phi_{12}(qz)&=\rho_2(z)\phi_1(qz)\,.
\end{align}

\subsection{The Extended $QQ$-System and Bethe ansatz}
The first line of \eqref{eq:QQAll} is the $SL(r+1)$ $QQ$-system \eqref{eq:QQgamma}. In the rest of the equations we introduced new functions \eqref{eq:phimdef}. Notice that 
$$
\rho_{i+1}(z)\phi_i(qz) =\Lambda_{i+1} (z)\frac{Q^-_{i}(qz)Q^+_{i+2}(z)}{Q^+_{i+1}(z)Q^+_{i+1}(qz)} =: \rho_{i,i+1}(z)\,,
$$
where $\rho_{i,i+1}(z)$ is $\rho_{i+1}(z)$ with $Q^+_i(z)$ replaced by $Q^-_i(z)$.
In terms of this new notation we can rewrite \eqref{eq:QQAll} as follows
\begin{align}\label{eq:QQAllMutations}
\xi_{i} \,\phi_i(z)-\xi_{i+1}\,  \phi_i(qz)&= \rho_i(z)\,,\qquad \qquad &i=1,\dots,r\notag\\
\xi_{i} \,\phi_{i,i+1}(z)-\xi_{i+2}\,  \phi_{i, i+1}(qz)&=\rho_{i,i+1}(z)\,,\qquad \qquad &i=1,\dots,r-1\notag\\
\dots&\dots\\
\xi_{i} \,\phi_{i,\dots,r-2+i}(z)-\xi_{r+i-1}\,  \phi_{i,\dots,r-2+i}(qz)&=\rho_{i,\dots,r-1+i}(z)\,,\qquad \qquad &i=1,2\notag\\
\xi_{1} \phi_{1,\dots,r}(z)-\xi_{r+1}  \phi_{1,\dots,r}(qz)&= \rho_{1,\dots,r}(z)\,,\qquad \qquad &\notag
\end{align}
For the future reference let us rewrite the above equations in terms of the Q-polynomials:
\begin{align}\label{eq:QQALLFull}
\xi_{i} Q^+_i(q z) Q^-_i(z) - \xi_{i+1} Q^+_i(z) Q^-_i(qz) &=\Lambda_i (z) Q^+_{i-1}(qz)Q^+_{i+1}(z)\,,\notag\\
\xi_{i} Q^+_{i+1}(q z) Q^-_{i,i+1}(z) - \xi_{i+2} Q^+_{i+1}(z) Q^-_{i,i+1}(qz) &= \Lambda_{i+1} (z) Q^-_{i}(qz)Q^+_{i+2}(z)\,,\notag\\
\dots&\dots\\
\xi_i \,Q^+_{r-2+i}(qz)Q^-_{i,\dots,r-2+i}(z)-\xi_{r-1+i}\,  Q^+_{r-2+i}(z)Q^-_{i,\dots,r-2+i}(qz)&=\Lambda_{r-1+i}(z)Q^-_{i,\dots,r-1+i}(qz)Q^+_{r+i}(z)\,,\notag\\
\xi_{1} Q^+_r(qz)Q^-_{1,\dots,r}(z)-\xi_{r+1}  Q^+_r(z)Q^-_{1,\dots,r}(qz)&= \Lambda_r(z) Q^-_{1,\dots,r-1}(qz)\,.\notag
\end{align}
We shall refer to all equations of \eqref{eq:QQALLFull} as the \textit{extended $QQ$-system} for $SL(r+1)$. We call its solution {\it nondegenerate}, if the resulting solution of the original $QQ$-system is nondegenerate.

Let us now show that starting from the solution of the nondegenerate $QQ$-system, we obtain solutions to the extended $QQ$-system as well. To do that we need the result (which is true for other simply laced groups) of \cite{Frenkel:2020}: 

\begin{Thm}
The solutions of the nondegenerate $SL(r+1)$ $QQ$-system are in one-to-one correspondence to the solutions of the Bethe Ansatz equations for $\mathfrak{sl}(r+1)$ XXZ spin chain:
\begin{equation}    \label{eq:bethe}
\frac{Q^+_{i}(qw_k^i)}{Q^+_{i}(q^{-1}w^k_i)} \frac{\xi_i}{\xi_{i+1}}=
- \frac{\Lambda_i(w_k^i) Q^{+}_{i+1}(qw_k^i)Q^{+}_{i-1}(w_k^i)}{\Lambda_i(q^{-1}w_k^i)Q^{+}_{i+1}(w_k^i)Q^{+}_{i-1}(q^{-1}w_k^i)},
\end{equation}
where $i=1,\ldots,r; k=1,\ldots,m_i$.
\end{Thm} 

We will extend the statement of this Theorem as follows.

\begin{Thm}\label{Th:BetheQQEquiv}
There is a one-to-one correspondence between the set of nondegenerate solutions of the extended $QQ$-system \eqref{eq:QQALLFull}, the set of nondegenerate solutions of the $QQ$-system \eqref{eq:QQAtype}, and the set of  solutions of Bethe Ansatz equations \eqref{eq:bethe}.
\label{Th:LemmaNondeg}
\end{Thm}

\begin{proof}
Consider the first line of \eqref{eq:QQALLFull} which is also presented in \eqref{eq:QQAtype}. If the $QQ$-system is nondegenerate, then, according to the previous theorem, there is a bijection between its polynomial nondegenerate solutions and Bethe equations.
We will show now recursively that given the nondegenerate solution of the $QQ$-system 
one can construct elements $Q_{i,i+1, \dots, j}$ satisfying the equations of the extended $QQ$-system.   
Let us immediately consider the degenerate case, when $Q^-_i(z)$ and $Q^+_{i+1}(z)$ have common roots: without loss of generality, let us now assume that $Q^-_i(z)$ and $Q^+_{i+1}(z)$ have just one common root $u$.  Now we show that we can construct the solution to the second equation in (\ref{eq:QQALLFull}), namely $Q^+_{i,i+1}$. 
Introducing the notation  
\begin{equation}\label{eq:Qplredef}
Q^+_{i+1}(z)=(z-u){\check Q}^+_{i+1}(z),
\end{equation}
we see from nondegeneracy condition $QQ$-system, namely the fact that $Q_i^+(z)$ and $Q^+_{i+1}(z)$ have q-distinct roots, we have: 
\begin{equation}\label{eq:Qmredef}
Q^-_i(z)=(z-u)\left(q^{-1}z-u\right)\widetilde Q^-_i(z)\,.
\end{equation}

Now consider the following equation from the second line from \eqref{eq:QQAllMutations}
$$
\xi_{i} \,\phi_{i,i+1}(z)-\xi_{i+1}\,  \phi_{i, i+1}(qz)= \Lambda_{i+1} (z)\frac{Q^-_{i}(qz)Q^+_{i+2}(z)}{Q^+_{i+1}(z)Q^+_{i+1}(qz)}\,.
$$
Substituting \eqref{eq:Qplredef} and \eqref{eq:Qmredef} we get
\begin{equation}\label{eq:QQ2}
\xi_{i} \,\frac{Q^-_{i,i+1}(z)}{(z-u){\check Q}^+_{i+1}(z)}-\xi_{i+2}\,  \frac{Q^-_{i,i+1}(qz)}{(qz-u){\check Q}^+_{i+1}(qz)}= \Lambda_{i+1} (z)\frac{{\widetilde Q}^-_{i}(qz)Q^+_{i+2}(z)}{{\check Q}^+_{i+1}(z){\check Q}^+_{i+1}(qz)}\,.
\end{equation}
From the residue decomposition of both sides of the above equation we conclude that
$u$ must be the root of polynomial $Q^-_{i,i+1}(z)$
$$
Q^-_{i,i+1}(z) = (z-u){\check Q}^-_{i,i+1}(z)\,.
$$

Thus, one can represent the resulting system as follows:
\begin{align}    \label{hi}
&\rho_{i,i+1}(z)=h_i(z)+
\sum_{k=1, k\neq s}^{m_{i+1}}\frac{b_k}{z-w^{i+1}_{k}}
+ \sum_{k=1, k\neq s}^{m_{i+1}} \frac{c_k}{qz-w^{i+1}_{k}},\\    \label{tildephi}
&\phi_{i,i+1}(z)=\wt{\phi}_i(z)+\sum_{k=1, k\neq s}^{m_{i+1}} \frac{d_k}{z-w^{i+1}_{k}},
\end{align}
where $w^{i+1}_s=u$, $h_i(z)$ and $\wt{\phi}_i(z)$ are polynomials. By matching the polar and polynomial parts of (\ref{eq:QQ2}) we can readily find coefficients $d_k$ and polynomials $\wt{\phi}_i(z)$ and hence $Q^+_{i,i+1}(z)$.

The only constraint we need to satisfy is the one on $b_k, c_k$, 
namely $\frac{b_k}{\xi_{i+2}}+\frac{c_k}{\xi_{i}}=0$, where $k\neq s$. These equations are 
explicitly given by:
\begin{equation}   \label{mutant}
\frac{Q^+_{i+1}(qw_k^{i+1})}{Q^+_{i+1}(q^{-1}w_k^{i+1})} \frac{\xi_{i}}{\xi_{i+2}}=
- \frac{\Lambda_{i+1}(w_k^{i+1}) Q^{+}_{i+2}(qw_k^{i+1})Q^{-}_{i}(w_k^{i+1})}{\Lambda_{i+1}(q^{-1}w_k^{i+1})Q^{+}_{i+2}(w_k^{i+1})Q^{-}_{i}(q^{-1}w_k^{i+1})},~{\rm  where } ~ k\neq s.
\end{equation}
At the same time the $i$th equation can be rewritten as:
$$
\xi_{i} \frac{Q^+_{i}(qz)}{Q^+_{i}(z)}  - \xi_{i+1}  \frac{Q^-_{i}(qz)}{Q^-_{i}(z)}   =\frac{\Lambda_{i} (z) Q^+_{i-1}(qz)Q^+_{i+1}(z)}{Q^+_{i}(z)Q^-_{i}(z)}
$$
which leads to 
$$
\xi_{i} \frac{Q^+_{i}(qw_k^{i+1})}{Q^+_{i}(w_k^{i+1})}  =\xi_{i+1}  \frac{Q^-_{i}(qw_k^{i+1})}{Q^-_{i}(w_{k}^{i+1})}, 
$$
where $w_k^{i+1}$ are the roots of $Q_{i+1}^+(z)$ for $k\neq s$.

Thus the equations (\ref{mutant}) are equivalent to the Bethe equations emerging from the $QQ$-system:
\begin{equation}   
\frac{Q^+_{i+1}(qw_k^{i+1})}{Q^+_{i+1}(q^{-1}w_k^{i+1})} \frac{\xi_{i+1}}{\xi_{i+2}}=
- \frac{\Lambda_{i+1}(w_k^{i+1}) Q^{+}_{i+2}(qw_k^{i+1})Q^{+}_{i}(w_k^{i+1})}{\Lambda_{i+1}(q^{-1}w_k^{i+1})Q^{+}_{i+2}(w_k^{i+1})Q^{+}_{i}(q^{-1}w_k^{i+1})}
\end{equation}
Therefore, we found that the equation $\eqref{eq:QQ2}$ follows from the XXZ Bethe equations. 

The above step can be iterated if $Q^-_{i,\dots, j}$ polynomials have coincident roots with $Q^{+}_{j}$. Therefore we have shown that any $QQ$-system of the form \eqref{eq:QQAllMutations} with such degeneracies is equivalent to a nondegenerate $QQ$-system. \end{proof}

In the next section we shall present a different proof of the above theorem, exploring the definition of Miura $(SL(r+1,q)$-opers involving flags of subbundles.

\subsection{The extended $QQ$-System, Bethe equations and B\"acklund Transformations}
We would like to understand the representation theoretic meaning of the extended $QQ$-system a bit better. In fact, motivated by Lemma 7.3 of \cite{Frenkel:2020} we can demonstrate that, starting from the original $QQ$-system (the first line of \eqref{eq:QQALLFull} or \eqref{eq:QQAtype}), upon certain assumptions, one can recover all the remaining equations of the entire extended $QQ$-system by \textit{B\"acklund transformations.} 

In \cite{Frenkel:2020} B\"acklund transformations were introduced for Miura q-opers (Proposition 7.1) and
were associated to the $i$-th simple reflection from the Weyl group:
\begin{Prop}    \label{fiter}
  Consider the $q$-gauge transformation of the $q$-connection given by \eqref{form of A1}
  \begin{eqnarray}
A \mapsto A^{(i)}=e^{\mu_i(qz)f_i}A(z)e^{-\mu_i(z)f_i},
\quad \operatorname{where} \quad \mu_i(z)=\frac{Q^+_{i-1}(z)Q^+_{i+1}(z)}{Q^i_{+}(z)Q^i_{-}(z)}\,.
\label{eq:PropDef}
\end{eqnarray}
Then $A^{(i)}(z)$ can be obtained from $A(z)$ by
substituting in formula \eqref{form of A1}
\begin{align}
Q^j_+(z) &\mapsto Q^j_+(z), \qquad j \neq i, \\
Q^i_+(z) &\mapsto Q^i_-(z), \qquad Z\to s_i(Z) \quad \Big(\zeta_i\mapsto \displaystyle\frac{\zeta_{i-1}\zeta_{i+1}}{\zeta_i}\Big)\,\quad \,.
\label{eq:Aconnswapped}
\end{align}
\end{Prop}

It is possible that after the transformation the resulting operator gives rise to the nondegenerate $QQ$-system. 
Denoting the the  $QQ$-system after the B\"acklund transformation as $\{\widetilde{Q}^{\pm}_i\}_{i=1,\dots, r}$, we obtain:
\begin{align} \label{qqm2}
  \{ \wt{Q}^+_j \}_{j=1,\ldots,r} &= \{ Q_1^{+}, \dots, 
 Q_{i-1}^+,Q_i^-,Q_{i+1}^+ \dots , Q^r_{+} \}; \\ \notag
  \{ \wt{Q}^-_j \}_{j=1,\ldots,r} &=  \left\{ Q^-_1, \dots,
   {Q^*}_{i-1}^-,-Q_i^+,Q_{i,i+1}^- \dots , Q^-_{r} \right\}\,\\
  \{ \wt{\zeta}_j \}_{j=1,\ldots,r} &= \left\{
  \zeta_1,\dots,\zeta_{i-1},\frac{\zeta_{i-1}\zeta_{i+1}}{\zeta_i}
,\dots,\zeta_r\right\} \notag
\end{align}

The last line can be also rewritten in terms of $\xi$ variables as follows: 
$$
\{ \wt{\xi}_j \}_{j=1,\ldots,r} = \left\{ \xi_1,\dots,\xi_{i-1},\xi_{i+1}, \xi_i, \xi_{i+2}
,\dots,\xi_{r+1}\right\} 
$$

Here we note, that the notation $Q^-_{i,i+1}$ was used for $Q^-_{i+1}$, since the equation this new polynomial satisfies, is the second one from the extended $QQ$-system. At the same time, the new polynomial ${Q^*}^-_{i-1}(z)$ does not belong to what we called the extended $QQ$-system.

As an example, if we apply the 1st B\"acklund transformation 
$$
Q^+_ 1 \mapsto Q^-_1\,,\quad
Q^-_ 1 \mapsto -Q^+_1\,,\quad
\xi_1 \mapsto \xi_{2}\,,\quad
\xi_{2} \mapsto \xi_1\,,\quad
Q^-_{2} \mapsto Q^-_{1,2}\,,
$$
to the $QQ$ system for $SL(3)$: 
\begin{align}
\xi_1 Q^+_1(qz)Q^-_1(z)-\xi_2 Q^+_1(z)Q^-_1(qz)&= \Lambda_1(z)Q^+_2(z)\,,\notag\\
\xi_2 Q^+_2(qz)Q^-_2(z)-\xi_3 Q^+_2(z)Q^-_2(qz)&= \Lambda_2(z)Q^+_1(qz)
\end{align}
the first equation above will not change, however, the second will become
\begin{equation}
\xi_1 Q^+_2(qz)Q^-_{1,2}(z)-\xi_3 Q^+_2(z)Q^-_{1,2}(qz)= \Lambda_2(z)Q^-_1(qz)\,,
\end{equation}
which completes its extended $QQ$-system \eqref{eq:SL3QQ}.

In general, one can talk about the successive B\"acklund transformations 
associated with the Weyl group element $w$. If such transformations are possible, namely if after each of the elementary B\"acklund transformation one arrives to the nondegenerate q-oper (i.e. nondegenerate solution of the $QQ$-system), such oper is called $w$-generic in \cite{Frenkel:2020}. As one can see, the equations of the extended $QQ$-system emerge as a part of the all $QQ$-system equations obtained after every B\"acklund transformation if the Weyl group element is constructed by successive reflections along the order in the Dynkin diagram: $w=s_is_{i+1}\dots s_{j-1}s_j$.

We will refer to the collection of $QQ$-system equations, obtained via B\"acklund transformations for all Weyl group elements $w$ as the {\it full $QQ$-system}. 
    
One of the applications of the B\"acklund transformations which was proven in \cite{Frenkel:2020} is that $Z$-twisted Miura-Pl\"ucker $(G,q)$-oper is $Z$-twisted Miura $(G,q)$-oper if it is $w_0$-generic, where $w_0$ is the longest root.

Here we show that a stronger result holds for $(SL(r+1),q)$-opers. 
Combining  Theorems \ref{th:SLNqMiura} and \ref{Th:BetheQQEquiv}, we obtain the following theorem, which is  the central result of this section.

\begin{Thm}\label{existext}
The nondegenerate Z-twisted Miura-Pl\"ucker $(SL(r+1),q)$-opers are Z-twisted Miura $(SL(r+1),q)$-opers. They are in one-to-one correspondence with the nondegenerate solutions of the QQ system and thus $\mathfrak{sl}(r+1)$ XXZ Bethe equations. 
\end{Thm}

\section{$q$-Opers via Quantum Wronskians} \label{Sec:QWronskians}
\subsection{Sections of line bundles and q-Wronskians}

In this section we will make use of an alternative definition of Miura $(SL(r+1),q)$-opers (see Definition \ref{Miuraflag}) to describe Z-twisted Miura q-opers with regular singularities, following \cite{KSZ}.  
Namely, we have a complete flag of subbundles $\cL_\bullet$ such that q-connection $A$ maps $\cL_{i}$ 
  into $\cL_{i-1}^q$ and the induced maps
  $\bar{A}_i:\cL_{i}/\cL_{i+1}\to \cL^q_{i-1}/\cL^q_{i}$ are
  isomorphisms for $i=1,\dots,r$ on  $U \cap M_q^{-1}(U)$, where $U$ is the Zariski open dense subset.
Explicitly, considering the determinants  
\begin{equation}\label{altqW} 
 \left(\Big(\prod_{j=0}^{i-2}(A(q^{i-2-j}z)\Big)s(z)\wedge\dots\wedge A(q^{i-2}z) s(q^{i-2}z)\wedge s(q^{i-1}z)
  \right)\bigg|_{\Lambda^i\cL_{r-i+2}^{q^{i-1}}}
\end{equation}
for $i=1,\dots, r+1$,  
where $s$ is a local section of $\cL_{r+1}$, we claim that 
$(E,A,\cL_\bullet)$ is an $(SL(r+1),q)$-oper if and only if at every point of $U \cap M_q^{-1}(U)$, there 
exists local section for which each  such determinant  is nonzero (see \cite{KSZ}).   
When we encounter the case of regular singularities (see Section \ref{Miurareg}), each  
$\bar{A}_i$ is an isomorphism except at zeroes of $\Lambda_i$ and thus we require the determinants to vanish at zeroes of the following polynomial $W_k(s)$: 
\begin{eqnarray}\label{eq:WPDefs}
W_k(s)=P_1(z) \cdot P_2(q^2z)\cdots P_{k}(q^{k-1}z),  \qquad 
P_i(z)=\Lambda_{r}\Lambda_{r-1}\cdots\Lambda_{r-i+1}(z)\,.
\end{eqnarray}  

Now we discuss the $Z$-twisted Miura condition. Recall from Section \ref{Miurafinsec} that Miura condition implies that there exist a flag $\hcL_\bullet$ which is preserved by the q-connection $A$. The $Z$-twisted condition implies that in the gauge when $A$ is given by fixed semisimple diagonal element $Z\in H$ such flag is formed by the standard basis $e_1, \dots, e_{r+1}$. 

The relative position between two flags is generic on $U \cap M_q^{-1}(U)$. The regular singularity condition implies that {\it quantum Wronskians}, namely  determinants
\begin{equation}\label{qD}
\mathcal{D}_k(s)=e_1\wedge\dots\wedge{e_{r+1-k}}\wedge
Z^{k-1}s(z)\wedge Z^{k-2} s(qz)\wedge\dots\wedge Z s(q^{k-2})\wedge s(q^{k-1}z)\,
\end{equation}
have a subset of zeroes, which coincide with those of
$\cW_k(s)$.  To be more explicit, for
$k=1,\dots,r+1$, we have nonzero constants $\alpha_k$ and polynomials
\begin{equation} \mathcal{V}_k(z) = \prod_{a=1}^{r_k}(z-v_{k,a})\,,
\label{eq:BaxterRho}
\end{equation}
for which 
\begin{equation}
\det\begin{pmatrix} \,     1 & \dots & 0 & \xi_1^{k-1}s_{1}(z) & \cdots & \xi_{1} s_{1}(q^{k-2}z)  &  s_{1}(q^{k-1}z) \\ 
 \vdots & \ddots & \vdots& \vdots & \vdots & \ddots & \vdots \\  
0 & \dots & 1&\xi_{k}^{k-1}s_{r+1-k}(z) &\dots & \xi_{k} s_{r+1-k}(q^{k-2}z) &   s_{k}(q^{k-1}z)  \\  
0 & \dots & 0&\xi^{k-1}_{k+1}s_{r+1-k+1}(z) & \dots & \xi_{r+1-k+1} s_{k+1}(q^{k-2}z)  &  s_{k+1}(q^{k-1}z)  \\
\vdots & \ddots & \vdots&\vdots & \vdots & \ddots & \vdots \\
0 & \dots & 0&\xi_{r+1}^{k-1}s_{r+1}(z) & \dots &\xi_{r+1} s_{r+1}(q^{k-2}z)  & s_{r+1}(q^{k-1}z)  \, \end{pmatrix} =\alpha_{k} W_{k}
\cV_{k} \,; 
\label{eq:MiuraQOperCond}
\end{equation}
Since $\cD_{r+1}(s)=\cW_{r+1}(s)$, we have
$\cV_{r+1}=1$.  We also set $\cV_0=1$; this is consistent with the fact
that \eqref{qD} also makes sense for $k=0$, giving
$\cD_0=e_1\wedge\dots\wedge e_{r+1}$.

We can also rewrite \eqref{eq:MiuraQOperCond} as
\begin{equation}
  \underset{i,j}{\det} \left[\xi_{r+1-k+i}^{k-j} s_{r+1-k+i}(q^{j-1}z)\right] = \alpha_{k} W_{k} \mathcal{V}_{k}\,,
\label{eq:MiuraDetForm}
\end{equation}
where $i,j = 1,\dots,k$. 

Note that the above determinants have slightly different form those of \cite{KSZ} -- twist parameters $\xi_i$ entered in different powers. This is due to a different order of the simple roots in the definition of the q-oper.

\begin{Thm}[\cite{KSZ}]\label{qWthKSZ}
Polynomials $\{\mathcal{V}_k(z)\}_{k=1,\dots, r}$ give the solution to the $QQ$-system \ref{eq:QQAtype} 
so that $Q^+_j(z)=\mathcal{V}_j(z)$ under the nondegeneracy condition that for all $i,j,k$ with $i \neq j$ and $a_{ik} \neq  0, a_{jk} \neq 0$, the zeros of $\mathcal{V}_i(z)$ and $\mathcal{V}_j(z)$ are
$q$-distinct from each other and from the zeros of $\Lambda_k(z)$.
\end{Thm}

\begin{Rem} Technically, we used stronger conditions in \cite{KSZ}, namely that zeroes of $\{\Lambda_i\}_{i=1\dots r}$ and 
$\{\mathcal{V}_j(z)\}_{j=1\dots r}$ have to be q-disjoint to satisfy $QQ$-system equations, but one can relax it easily and even more than it is done in the statement above.
\end{Rem}

In the next subsection we will show that the extended $QQ$-system  can be obtained from various minors in q-Wronskian matrices. This  theorem allows to relate the section $s(z)$, generating the line bundle $\mathcal{L}_{r+1}$ with the elements of the extended $QQ$-system  using the transformation \eqref{eq:qGaugeTrSpecial}. 

\begin{Prop}\label{Th:SectionsQQ}
Let $v(z)$ be the  gauge transformation from \eqref{eq:qGaugeTrSpecial} and $s(z)$ be the section generating $\mathcal{L}_{r+1}$ in the definition of the $(SL(r+1),q)$-oper. Then the components of $s(z)$ in the 
gauge when q-oper connection is equal to $Z$ is given by:
\begin{equation}\label{eq:comptssection}
s_{r+1}(z)=Q^+_{r}(z)\,,\qquad s_r(z)=Q^-_r(z)\,,\qquad s_k(z)=Q^-_{k,\dots,r}(z)\,,
\end{equation}
for $k=1,\dots, r-1$.
\end{Prop}

\begin{proof}
The Proposition follows from the direct application of \eqref{eq:vinverse}
Starting from \eqref{eq:vinverse} the Proposition follows after acting with $v(z)^{-1}$ on the basis vector $e_{r+1}=(0,0,\dots, 0, 1)$. 
\end{proof}

In the next subsection we will show that the extended $QQ$-system  can be obtained from various minors in q-Wronskian matrices. 

\subsection{Wronskians and extended $QQ$-systems}

First, we will rewrite the extended $QQ$-system in a more convenient way to relate it to the minors in the q-Wronskian matrix. Namely, we multiply $Q$-terms by certain polynomials to get rid of the $\Lambda$-polynomials in the right hand side. This is done in the following Lemma.
\begin{Lem}\label{Th:PropQQtilde}
The system of equations \eqref{eq:QQAll} is equivalent to the following set of equations
\begin{align}\label{eq:QQAll1}
\xi_{i} \,\mathscr{D}^+_{i}(q z)\mathscr{D}^-_{i}(z)-\xi_{i+1}\, \mathscr{D}^+_{i} (z)\mathscr{D}^-_{i}( qz)&= (\xi_{i}-\xi_{i+1})\,\mathscr{D}^+_{i-1} (qz)\mathscr{D}^+_{i+1}(z)\,,\notag\\
\xi_{i} \,\mathscr{D}^+_{i+1}(q z)\mathscr{D}^-_{i,i+1}(z)-\xi_{i+2}\, \mathscr{D}^+_{i+1} (z)\mathscr{D}^-_{i,i+1}( qz)&= (\xi_{i}-\xi_{i+2})\,\mathscr{D}^-_{i} (qz)\mathscr{D}^+_{i+2}(z)\,,\notag\\
\dots&\dots\\
\xi_i \,\mathscr{D}^+_{r+i-2}(q z)\mathscr{D}^-_{i,\dots,r-1+i}(z)-\xi_{r+i-1},  \mathscr{D}^+_{r+i-2}(z)\mathscr{D}^-_{i,\dots,r-1+i}(q z)&=(\xi_i-\xi_{r+i-1})\mathscr{D}^-_{i,\dots,r-2+i}(qz)\mathscr{D}^+_{r+i-1}(z)\,,\notag\\
\xi_{1} \mathscr{D}^+_{r}(qz)\mathscr{D}^-_{1,\dots,r}(z)-\xi_{r+1} \mathscr{D}^+_{r}(z)\mathscr{D}^-_{1,\dots,r}(qz)&= (\xi_{1}-\xi_{r+1})\mathscr{D}^-_{1,\dots,r-1}(qz)\,.\notag
\end{align}
where index $i$ ranges between the same values as in the corresponding equations in \eqref{eq:QQAll}, for the polynomials 
\begin{equation}\label{eq:Dkdef}
\mathscr{D}^+_k={Q^+_k}{F_k}\,,\quad \mathscr{D}^-_k= Q^-_k {F_k}\eta_k,\,\qquad \mathscr{D}^-_{l,\dots, k}= Q^-_{l,\dots, k} {F_k}\eta_{l,\dots, k}\,.
\end{equation}
where 
$$
F_i(z)=  W_{r-i}(q^{r-i}z)\,,\qquad \eta_{l,\dots, i} = \prod_{a=0}^{i-l}(\xi_l-\xi_{l+a+1})\,.
$$
\end{Lem}

For the future we shall refer to \eqref{eq:QQAll1} as the \textit{extended} $\mathscr{D}\mathscr{D}$-system for $SL(r+1)$ and to its first line specifically as merely the $\mathscr{D}\mathscr{D}$-system.

\begin{proof}
The proof is the direct extension of the proof of Lemma 4.2 in \cite{KSZ} to other equations in \eqref{eq:QQALLFull}. Since all equations are treated analogously, let us consider the second set of \eqref{eq:QQAll1}
which we can write as
\begin{equation}\label{eq:2ndeqshifter}
\xi_{i-1} \,\mathscr{D}^+_{i}(q z)\mathscr{D}^-_{i-1,i}(z)-\xi_{i+1}\, \mathscr{D}^+_{i} (z)\mathscr{D}^-_{i-1,i}( qz)= (\xi_{i-1}-\xi_{i+1})\,\mathscr{D}^-_{i-1} (qz)\mathscr{D}^+_{i+1}(z)\,.
\end{equation}
After replacing 
$$
\mathscr{D}^+_i={Q^+_i}{F_i}\,,\quad \mathscr{D}^-_i= Q^-_i {F_i}\,\eta_i,\,\qquad \mathscr{D}^-_{i-1, i}= Q^-_{i-1,i} {F_{i}}\,\eta_{i-1,i}
$$
and assigning 
$$
\eta_{i-1} = \xi_{i-1}-\xi_{i},\qquad \eta_{i-1,i} = (\xi_{i-1}-\xi_{i})(\xi_{i-1}-\xi_{i+1})\,,
$$
we can see that \eqref{eq:2ndeqshifter} is equivalent to the second equation of \eqref{eq:QQALLFull} provided that the following difference equation is satisfied:
$$
\frac{F_{i-1}(qz) F_{i+1}(z)}{F_{i}(qz) F_{i}(z)}\cdot\frac{\eta_{i-1}}{\eta_{i-1,i}}(\xi_{i-1}-\xi_{i+1})=\Lambda_i(z)\,.
$$
The validity of this relation follows from the above formulae and form the definitions \eqref{eq:WPDefs}
$$
\frac{F_{i-1}(qz) F_{i+1}(z)}{ F_{i}(qz) F_{i}(z)} = \frac{W_{r-i+1}(q^{r-i+1}z)}{W_{r-i}(q^{r-i}z)}\frac{W_{r-i-1}(q^{r-i}z)}{W_{r-i}(q^{r-i+1}z)} = \frac{P_{r-i+1}(q^{r-i}z)}{P_{r-i}(q^{r-i}z)}=\Lambda_i(z)\,.
$$

\end{proof}

As we shall see below, one can express the solutions of the $QQ$- and $\mathscr{D}\mathscr{D}$-systems in terms of  section $s(z)$ of subbundle  $\mathcal{L}_{r+1}$.
Following the discussion of \cite{KSZ} (Section 4) we consider the following matrices: 
\begin{equation}
M_{i_1,\dots, i_j}=
\begin{pmatrix} \,  
\xi_{i_1}^{j-1}s_{i_1}(z)  & \cdots  & \xi_{i_1}s_{i_1}(q^{j-2}z)&  s_{i_1}(q^{j-1}z) \\ \vdots & \ddots & \vdots & \vdots \\   \xi_{i_j}^{j-1}s_{i_j}(z)  & \cdots& \xi_{i_j} s_{i_j}(q^{j-2}z)  & s_{i_j}(q^{j-1}z) \, 
\end{pmatrix}\\, \quad V_{i_1,\dots, i_j}=
\begin{pmatrix} 
\,   \xi_{i_1}^{j-1}&\cdots & \xi_{i_1} &  1  \\ \vdots & \ddots & \vdots & \vdots \\  \xi_{i_j}^{j-1} & \cdots & \xi_{i_j}   & 1  \, 
\end{pmatrix}\\,
\label{eq:MiuraQOperCond1q}
\end{equation}
where $s_i$ are polynomials and $V_{i_1,\dots,i_j}$ is the Vandermonde-like matrix whose determinant is
\begin{equation}
\text{det} V_{i_1,\dots, i_j} = \prod_{i<j}(\xi_i-\xi_j)\,.
\end{equation}

In \cite{KSZ} the following Proposition was proven (Lemma 4.4 in slightly different notations), which allows to express solutions of the $QQ$-system in terms of q-Wronskians of \eqref{eq:MiuraQOperCond1q}. Here we provide the sketch of proof for completeness.

\begin{Prop}\label{eq:qWronskiansShift}
Given polynomials $\mathscr{D}^+_i, \mathscr{D}^-_i$ for $i=1,\dots,r$ satisfying the first line of \eqref{eq:QQAll1}, there exist unique polynomials
$s_1,\dots,s_{r+1}$ such that 
\begin{equation}
\mathscr{D}^+_i (z)= \frac{\det\, M_{r+2-i,\dots,r+1}(z)}{\det\,
  V_{r+2-i,\dots,r+1}}\,\qquad\text{ and } \qquad
\mathscr{D}^-_i(z) = \frac{\det\, M_{r+1-i,r+3-i,\dots,r+1}(z)}{\det\, V_{r+1-i,r+3-i,\dots,r+1}}\,,
\label{eq:MM0Inda}
\end{equation}
where matrix $M$ is given in \eqref{eq:MiuraQOperCond1q}.
\end{Prop}

\begin{proof}
The proof is based on the determinant Desnanot-Jacobi identity which holds for any $l\times l$ matrix $M$. In this proof we shall use this identity in the following form
\begin{equation}\label{eq:JacobiNew}
 M^{1}_1 M^{2}_{l}- M^{1}_{l}M^{2}_{1}= M^{1,2}_{1,l} M\,,
\end{equation}
where $M^a_b$ is the determinant of matrix $M$ with row $a$ and column $b$ removed (respectively $M^{a,c}_{b,d}$ is the determinant of matrix $M$ with rows $a$ and $c$ and column $b$ and $d$ removed). Note that the Desnanot-Jacobi identity holds for any pairs of indices $\{a,c\}$ and $\{b,d\}$ as long as $a\neq c$ and $b\neq d$.

In \cite{KSZ} it was shown, using periodic properties of matrix $M$ that the first line of \eqref{eq:QQAll1} can be identified with \eqref{eq:JacobiNew} if $M_i=M_{r+1-i,\dots,r+1}(z)$
 is the determinant of the bottom-right $i\times i$ submatrix of the $(r+1)\times(r+1)$ matrix $M_{1,\dots,r+1}(z)$.
It is easy to see, for instance, that 
$$
\left(M_i\right)^1_{1}(z)=M_{r+2-i,\dots,r+1}(qz)\,,\quad \left(M_i\right)^{2}_{1}(z)=M_{r+1-i,r+3-i,\dots,r+1}(qz) \,,
$$
\begin{align}
\left(M_i\right)^1_{i}(z)&=\xi_{r+2-i}\left(\prod_{a=r+3-i}^{r+1}\xi_i\right)\cdot M_{r+2-i,\dots,r+1}(z)\,,\notag\\\left(M_i\right)^{2}_{i}(z)&=\xi_{r+1-i}\left(\prod_{a=r+3-i}^{r+1}\xi_i\right)\cdot M_{r+1-i,r+3-i,\dots,r+1}(z) \,,\notag
\end{align}
$$
\left(M_i\right)^{1,2}_{1,i}(z)=\left(\prod_{a=r+3-i}^{r+1}\xi_i\right)\cdot M_{r+3-i,\dots,r+1}(qz)\,.
$$

We can substitute the above five relations into \eqref{eq:JacobiNew} and then divide its both sides by $V_{r+2-i,\dots,r+1}V_{r+1-i,r+3-i,\dots,r+1}$. The first $\mathscr{D}\mathscr{D}$-relation will follow after observing that
$$
V_{r+3-i,\dots,r+1}V_{r+1-i,\dots,r+1}=(\zeta_{r+1-i}-\zeta_{r+2-i})V_{r+2-i,\dots,r+1}V_{r+1-i,r+3-i,\dots,r+1}\,.
$$
\end{proof}

In the above proof we have derived an alternative presentation of $\mathscr{D}^\pm_i$ polynomials  and their $q$-shifted counterparts in terms of minors $M_i$
\begin{equation}
\mathscr{D}^+_i (z)= \frac{(M_i)^1_{i}(z)}{(V_i)^1_{i}}\,,\quad
\mathscr{D}^-_i (z)= \frac{(M_i)^{2}_{i}(z)}{(V_i)^{2}_{i}}\,, \quad
\mathscr{D}^+_i (qz)= \frac{(M_i)^1_1(z)}{(V_i)^1_{1}}\,,\quad
\mathscr{D}^-_i (qz)= \frac{(M_i)^{2}_{1}(z)}{(V_i)^{2}_1}\,, 
\label{eq:MM0IndaMinors}
\end{equation}
where $V_i=V_{r+1-i,\dots,r+1}(z)$ is the determinant of the bottom-right $i\times i$ submatrix of the $(r+1)\times(r+1)$ matrix $V_{1,\dots,r+1}(z)$. This way all polynomials which appear in the $\mathscr{D}\mathscr{D}$-system can be universally presented as ratios of (unshifted) minors of two sets of matrices $\{M_i\}$ and $\{V_i\}$ for $i=1,\dots,r$.

Thus $i$-th equation of the $\mathscr{DD}$-system represents a Densanot-Jacobi determinant identity for matrix $M_i$ of the form \eqref{eq:MiuraQOperCond1q}.
In the following subsection we shall demonstrate that all equations of the extended $\mathscr{DD}$-system can also be thought of as determinant identities for matrices which are obtained from $M_i$s by permutation of rows and columns. The latter is provided by B\"acklund transformations.

\subsection{B\"acklund transformations and the extended $\mathscr{D}\mathscr{D}$-System}
We have shown that the extended $QQ$-system is equivalent to the extended $\mathscr{D}\mathscr{D}$-system in Lemma \ref{Th:PropQQtilde}. Let us focus on $\mathscr{D}\mathscr{D}$-system, namely the 
equations, corresponding to the first line in \eqref{eq:QQAll1}. We already mentioned that all the equations from \eqref{eq:QQALLFull} can be obtained from the $QQ$-system by applying B\"acklund transformations. The same works for the $\mathscr{D}\mathscr{D}$-system. The $i$-th B\"acklund transformation replaces the data
$$
  \{ \mathscr{D}^+_j ,\mathscr{D}^-_j\}_{j=1,\ldots,r}\,,\quad   \{ \xi_j \}_{j=1,\ldots,r+1}
$$
with the following
\begin{align} \label{qqm2DD}
\{\wt{\mathscr{D}}^+_j \}_{j=1,\ldots,r} &= \{ \mathscr{D}_1^{+}, \dots,
  \mathscr{D}_{i-1}^+,\mathscr{D}_i^-,\mathscr{D}_{i+1}^+ \dots , \mathscr{D}_r^{+} \}; \\ \notag
\{\wt{\mathscr{D}}^-_j \}_{j=1,\ldots,r} &=  \left\{ \mathscr{D}^-_1, \dots,\mathscr{D}_{i-1}^{*\,-},\mathscr{D}_i^+,\mathscr{D}_{i-1,i}^-, \mathscr{D}_{i+2}^-\dots , \mathscr{D}^-_{r} \right\}\,\\
\{ \wt{\xi}_j \}_{j=1,\ldots,r} &= \left\{\xi_1,\dots,\xi_{i-1},\xi_{i+1},\xi_{i},\dots,\xi_{r+1}\right\} \notag
\end{align}
Notice that in the $QQ$-system this rule works as $Q^-_i \mapsto -Q^+_i$. In the $\mathscr{D}\mathscr{D}$-system the sign disappears due to the presence of the multiplicative factor $\eta_i$ between $Q^-_i$ and $\mathscr{D}$ functions. 

We also note, that polynomials $\mathscr{D}_{i-1}^{*\,-}$ do not belong to the extended $\mathscr{D}\mathscr{D}$-system, rather they will be a part of the \textit{full} $\mathscr{D}\mathscr{D}$-system. By applying B\"acklund transformations further we can readily find all polynomials from the full $\mathscr{D}\mathscr{D}$-system, which is in one-to-one correspondence with the full $QQ$-system we discussed in section 5.

We can now find a similar presentation for other polynomials $\mathscr{D}^-_{i,\dots,j}$ in terms of ratios of determinants by combining the ideas above and Proposition \ref{Th:SectionsQQ}.
In particular, we need to understand how B\"acklund transformations act on the matrices \eqref{eq:MiuraQOperCond1q}. 

Let us start with the $(r+1)\times(r+1)$-matrices from \eqref{eq:MiuraQOperCond1q}  $M_1 = M_{1,\dots,r+1}$ and $V_1 = V_{1,\dots,r+1}$. During the $i$th B\"acklund transformation \eqref{qqm2} the following functions in the extended $\mathscr{D}\mathscr{D}$-system get interchanged: $\mathscr{D}^+_i(z) \leftrightarrow \mathscr{D}^-_i(z)$, $\mathscr{D}^-_{i+1}(z)\leftrightarrow \mathscr{D}^-_{i,i+1}(z)$ and $\xi_i \leftrightarrow \xi_{i+1}$ which, using the identification \eqref{eq:comptssection}, amounts to acting by permutations $r_i$ on the set of indices as
$$
\{1,\dots, i,i+1,\dots,r+1\} \mapsto \{1,\dots, i+1,i,\dots,r+1\}\,.
$$
Therefore we can define a new tuple of matrices
$r_i(M_1)(z)=M_{1,\dots, i+1,i,\dots,r+1}$ and $r_i(V_{1})(z)=V_{1,\dots, i+1,i,\dots,r+1}$ as well as their  submatrices $r_i(M_j)(z)$ and $r_i(V_j)(z)$, which are obtained by excising the corresponding $(r-j+1)\times(r-j+1)$ bottom-right blocks. Then by Proposition \ref{eq:qWronskiansShift} we must have 
\begin{equation}
\mathscr{D}^+_i (z)= \frac{\det\, r_i(M)_{r+2-i,\dots,r+1}(z)}{\det\,
  r_i(V)_{r+2-i,\dots,r+1}}\,,\qquad
\mathscr{D}^-_{i,i+1}(z) = \frac{\det\, r_i(M)_{r+2-i,r+4-i,\dots,r+1}(z)}{\det\, r_i(V)_{r+2-i,r+4-i,\dots,r+1}}\,,
\label{eq:MM0Inda22}
\end{equation}
or, equivalently,
\begin{equation}
\mathscr{D}^-_i (z)= \frac{r_i(M_i)^1_{i}(z)}{r_i(M_i)^1_{i}}\,,\qquad
\mathscr{D}^-_{i,i+1} (z)= \frac{r_i(M_{i+1})^{1}_{i-1}(z)}{r_i(V_{i+1})^{1}_{i-1}}\,.
\label{eq:MM0IndaMinorsRefl}
\end{equation}

Notice that the second equality of \eqref{eq:MM0Inda22} can be written in terms of the original matrix $M$
\begin{equation}
\mathscr{D}^-_{i,i+1} = \frac{\det\, M_{i,i+3\dots,r+1}}{\det\, V_{i,i+3\dots,r+1}}\,.
\label{eq:MM0Inda2}
\end{equation}

In order to determine similar expressions for other $\mathscr{D}^-_{i,\dots,i+k}$ one needs to act by other elements of the Weyl group $ W=S_{r+1}$. Essentially, the B\"acklund transformation, being associated with the elementary Weyl reflection interchanges two rows 
in the q-Wronskian matrix. This brings us to the following statement, which an be verified by a direct calculation along the lines of Proposition \ref{eq:qWronskiansShift} for each set of equations of the extended $\mathscr{D}\mathscr{D}$-system. 

\begin{Prop}\label{eq:qWronskiansShift2}
The polynomials $\mathscr{D}^-_{i,\dots i+k}$ from \eqref{eq:QQAll1} read
\begin{equation}
\mathscr{D}^-_{i,\dots,i+k} = \frac{\det\, M_{i,i+k+2\dots,r+1}}{\det\, V_{i,i+k+2\dots,r+1}}\,,\qquad i=1,\dots, r\,,\quad k=0,\dots, r-i\,.
\end{equation}
or equivalently,
\begin{equation}
\mathscr{D}^-_{i,\dots,i+k} = \frac{r_{i+k-1}(\dots r_{i+1}(r_i(M_{i+k}))\dots)^{i+1}_{r+1}}{r_{i+k-1}(\dots r_{i+1}(r_i(V_{i+k}))\dots)^{i+1}_{r+1}}\,,
\label{eq:MM0Inda2Equiv}
\end{equation}
\end{Prop}

Although we do not discuss polynomials $\mathscr{D}_{i-1}^{*\,-}$ which belong to the full $\mathscr{D}\mathscr{D}$-system, rather than to the extended $\mathscr{D}\mathscr{D}$-system, we can nevertheless provide a formula for these polynomials.

\begin{Prop}\label{eq:qWronskiansDStar}
The polynomials $\mathscr{D}_{i-1}^{*\,-}$ from \eqref{eq:QQAll1} read
\begin{equation}\label{eq:MM0Inda2Dstar}
\mathscr{D}_{i-1}^{*\,-} = \frac{r_{i-1}(M_{i-1})^1_{i+1}(z)}{r_{i-1}(V_{i-1})^1_{i+1}}\,.
\end{equation}
\end{Prop}

\begin{proof}
The $(i-1)$st equation of the full $\mathscr{D}\mathscr{D}$-system after applying the $i$th B\"acklund transformation reads
$$
\xi_{i-1}\mathscr{D}_{i-1}^+(qz)\mathscr{D}_{i-1}^{*\,-}(z)-\xi_{i+1}\mathscr{D}_{i-1}^+(z)\mathscr{D}_{i-1}^{*\,-}(qz)=(\xi_{i-1}-\xi_{i+1})\mathscr{D}_{i-2}^+(qz)\mathscr{D}_{i}^{-}(z)\,.
$$
Given the description of $\mathscr{D}$ polynomials in terms of minors \eqref{eq:MM0IndaMinors} it can be shown that the above equation is equivalent to the Jacobi determinant identity of the form \eqref{eq:JacobiNew} for matrix $M_{i-1}$.
\end{proof}

This statement implies that the solutions of the full $\mathscr{D}\mathscr{D}$- and thus the full $QQ$-system are well-defined, if the original $QQ$-system is nondegenerate. Also, notice that all the other equations in the $QQ$-system correspond to all possible Miura $(SL(r+1),q)$-opers for a given $(SL(r+1),q)$-oper.
Thus, the following theorem is true, which generalizes Theorem \ref{existext}.

\begin{Thm}
i)The solution of the nondegenerate $(SL(r+1))$ $QQ$-system can be extended to the solution of the the full $QQ$-system. 

ii) This full $QQ$-system is comprised of $(r+1)!$ $QQ$-systems, with B\"acklund transformations acting transitively between them. 

iii) Each such $QQ$-system determine one of the $(r+1)!$ $Z$-twisted Miura $(SL(r+1),q)$-opers, corresponding to a unique Z-twisted $(SL(r+1),q)$-oper. 
\end{Thm}

We can combine Lemma \ref{Th:PropQQtilde} with Propositions \ref{eq:qWronskiansShift} and \ref{eq:qWronskiansShift2} to get the following theorem which will be used in later sections to study infinite-dimensional $q$-opers.
\begin{Prop}
The polynomials which appear in the extended $QQ$-system \eqref{eq:QQAllMutations} are given by
\begin{align}
Q^+_{i} (z)&= \frac{1}{F_i(z)}\cdot\frac{(M_i)^1_{i}(z)}{(V_i)^1_{i}}\,,\qquad
Q^-_i (z)= \frac{1}{F_i(z) \eta_i}\cdot\frac{(M_i)^{2}_{i}(z)}{(V_i)^{2}_{i}}\,, \notag\\
Q^-_{i,\dots, i+k}(z)  &= \frac{1}{F_i(z)\eta_{i,\dots, k}}\cdot\frac{s_{i+k-1}(\dots s_{i+1}(s_i(M_{i+k}))\dots)^{2}_{i}(z)}{s_{i+k-1}(\dots s_{i+1}(s_i(V_{i+k}))\dots)^{2}_{i}}\,.
\end{align}
or, equivalently,
\begin{align}
Q^+_{i} (z)&= \frac{1}{F_i(z)}\cdot\frac{\det\, M_{r+2-i,\dots,r+1}(z)}{\det\,V_{r+2-i,\dots,r+1}}\,,\notag\\
Q^-_i (z)&= \frac{1}{F_i(z) \eta_i}\cdot\frac{\det\, M_{r+1-i,r+3-i,\dots,r+1}(z)}{\det\, V_{r+1-i,r+3-i,\dots,r+1}}\,, \\
Q^-_{i,\dots, i+k}(z)  &= \frac{1}{F_i(z)\eta_{i,\dots, k}}\cdot \frac{\det\, M_{i,i+k+2\dots,r+1}(z)}{\det\, V_{i,i+k+2\dots,r+1}}\,.\notag
\end{align}
\end{Prop}

Note also the following expressions for shifted $Q$-functions which will be used later.
\begin{equation}\label{eq:ShiftedQs}
Q^+_{i} (qz)= \frac{1}{F_i(z)}\cdot\frac{(M_i)^1_{1}(z)}{(V_i)^1_{1}}\,,\qquad
Q^-_i (qz)= \frac{1}{F_i(z) \eta_i}\cdot\frac{(M_i)^{2}_{1}(z)}{(V_i)^{2}_{1}}\,.
\end{equation}

\subsection{Example: Miura $(SL(3),q)$-oper}
Define matrices
\begin{align}
M_1(z)&=
\begin{pmatrix}
 \xi _1^2s_1\left(z\right) & \xi _1 s_1(q z) &   s_1(q^2z) \\
\xi _2^2  s_2\left(z\right) & \xi _2 s_2(q z) &  s_2(q^2z) \\
\xi _3^2 s_3\left(z\right) & \xi _3 s_3(q z) &  s_3(q^2z) 
\end{pmatrix}
\,,\qquad
V_1=
\begin{pmatrix}
 \xi _1^2 & \xi _1 &   1 \\
\xi _2^2  & \xi _2  & 1 \\
\xi _3^2  & \xi _3  &  1 
\end{pmatrix}
\,,\notag\\
M_2(z)&=(M_1)_{23}(z)=
\begin{pmatrix}
\xi _2 s_2(z) &  s_2(qz) \\
\xi _3 s_3(z) &  s_3(qz) 
\end{pmatrix}
\,,\quad
V_2=(V_1)_{23}=
\begin{pmatrix}
\xi _2  &  1 \\
\xi _3  &  1 
\end{pmatrix}
\,,\notag
\end{align}
and the matrices which are obtained from the above by the Weyl action
\begin{align}
r_1(M_1)(z)&=
\begin{pmatrix}
\xi _2^2  s_2\left(z\right) & \xi _2 s_2(q z) &  s_2(q^2z) \\
 \xi _1^2s_1\left(z\right) & \xi _1 s_1(q z) &   s_1(q^2z) \\
\xi _3^2 s_3\left(z\right) & \xi _3 s_3(q z) &  s_3(q^2z) 
\end{pmatrix}
\,,\qquad
r_1(V_1)=
\begin{pmatrix}
\xi _2^2  & \xi _2  & 1 \\
 \xi _1^2 & \xi _1 &   1 \\
\xi _3^2  & \xi _3  &  1 
\end{pmatrix}
\,,\notag\\
M_{1,2}(z)&=r_1(M_1)_{23}(z)=
\begin{pmatrix}
\xi_1 s_1(z) &  s_1(qz) \\
\xi_3 s_3(z) &  s_3(qz) 
\end{pmatrix}\,,\qquad
V_{1,2}=r_1(V_1)_{23}=
\begin{pmatrix}
\xi _1  &  1 \\
\xi _3  &  1 
\end{pmatrix}\,.\notag
\end{align}

Then the $\mathscr{D}\mathscr{D}$-system reads
\begin{align}\label{eq:QQAll12}
\xi_{1} \,\mathscr{D}^+_{1}(q z)\mathscr{D}^-_{1}(z)-\xi_2\, \mathscr{D}^+_{1} (z)\mathscr{D}^-_{1}( qz)&=(\xi_1-\xi_2)\mathscr{D}^+_{2}( qz)W (z)\,,\notag\\
\xi_{2} \,\mathscr{D}^+_{2}(q z)\mathscr{D}^-_{2}(z)-\xi_3\, \mathscr{D}^+_{2} (z)\mathscr{D}^-_{2}( qz)&= (\xi_2-\xi_3)\mathscr{D}^+_{1} (z)\,,\notag\\
\xi_{1} \,\mathscr{D}^+_{2}(q z)\mathscr{D}^-_{1,2}(z)-\xi_3\, \mathscr{D}^+_{2} (z)\mathscr{D}^-_{1,2}( qz)&=(\xi_1-\xi_3) \mathscr{D}^-_{1} (z)\,,
\end{align}
with the following
$$
\mathscr{D}^+_1(z)=\frac{(M_1)^1_2(z)}{V^1_2}\,, \qquad \mathscr{D}^-_{1}(z) =\frac{(M_1)^2_3(z)}{V^2_3}\,,\qquad W(z)=\frac{\text{det} M_1(z)}{\text{det} V_1}\,,
$$
as well as
$$
\mathscr{D}^+_2(z)=\frac{(M_2)^1_2(z)}{(V_2)^1_2}=s_3(z)\,,\qquad \mathscr{D}^-_2(z)=\frac{(M_2)^2_2(z)}{(V_2)^2_2}=s_2(z)\,,\qquad \mathscr{D}^-_{1,2}(z)=\frac{(M_{1,2})^2_2(z)}{(V_{1,2})^2_2}=s_1(z)\,,
$$
where $(M_i)^a_b(z)$ is the determinant of matrix $M_i(z)$ with row $a$ and column $b$ removed.
The shifted $\mathscr{D}$-functions read
$$
\mathscr{D}^+_{1}(qz) =\frac{(M_1)^1_1(z)}{(V_1)^1_1}\,,\qquad \mathscr{D}^-_1(qz)=\frac{(M_1)^2_1(z)}{(V_1)^2_1}\,,
$$
$$
\mathscr{D}^+_{2}(qz) =\frac{(M_2)^1_1(z)}{(V_2)^1_1}\,,\qquad \mathscr{D}^-_2(qz)=\frac{(M_2)^2_1(z)}{(V_2)^2_1}\,, \quad \mathscr{D}^-_{1,2}(qz)=\frac{(M_{1,2})^2_1(z)}{(V_{1,2})^2_1}\,,
$$
Then the solutions of the $SL(3)$ $QQ$-system read
$$
Q^+_1(z)=\frac{1}{F_1(z)(\xi_1-\xi_2)}\cdot\frac{(M_1)^1_2(z)}{V^1_2}\,, \qquad Q^-_{1}(z) =\frac{1}{F_1(z)(\xi_1-\xi_2)}\cdot\frac{(M_1)^2_3(z)}{V^2_3}\,,
$$
as well as
\begin{align}
Q^+_2(z)&=\frac{1}{F_2(z)(\xi_2-\xi_3)}\cdot\frac{(M_2)^1_2(z)}{(V_2)^1_2}=\frac{s_3(z)}{F_1(z)(\xi_1-\xi_2)}\,,\notag\\
Q^-_2(z)&=\frac{1}{F_2(z)(\xi_2-\xi_3)}\cdot\frac{(M_2)^2_2(z)}{(V_2)^2_2}=\frac{s_2(z)}{F_2(z)(\xi_2-\xi_3)}\,\\
Q^-_{1,2}(z)&=\frac{1}{F_2(z)(\xi_2-\xi_3)}\cdot\frac{(M_{1,2})^2_2(z)}{(V_{1,2})^2_2}=\frac{s_1(z)}{F_2(z)(\xi_2-\xi_3)}\,.\notag
\end{align}

\subsection{Explicit Formula for $(SL(r+1),q)$-Oper via minors}
We can now collect all the results of this section in order to present the Miura $(SL(r+1),q)$-oper \eqref{eq:MiuraqConnection} in terms of trivialization of subbundle $\mathcal{L}_{r+1}$. Consider functions $g_i(z)$ which appear on the diagonal
\begin{align}
g_i(z)&=\zeta_i\frac{Q^+_i(qz)}{Q^+_i(z)}=\zeta_i\frac{F_i(z)}{F_i(qz)}\cdot\frac{(M_i)^1_1}{(M_i)^1_i}\frac{(V_i)^1_i}{(V_i)^1_1}\notag\\
&=\zeta_i\cdot\left( \prod\limits_{a=1}^i\xi_{r+2-a}\right)\cdot\prod_{b=1}^i\frac{\Lambda_{r-i+b}(q^{1-b}z)}{\Lambda_{r-i+b}(z)}  \cdot\frac{(M_i)^1_1}{(M_i)^1_i}\,.
\end{align}
Then the diagonal entry of \eqref{eq:MiuraqConnection} becomes the following meromorphic function
\begin{align}\label{eq:gratioNew}
\frac{g_{i+1}}{g_i}(z)= \xi_{i+1}\, H^{(r)}_i(z,q)\cdot G^{(r)}_i(z,q)\,,
\end{align}
where
\begin{equation}\label{eq:GriMatrix}
H^{(r)}_i(z,q)=\prod_{b=1}^i\frac{\Lambda_{r-i+b}(q^{1-b}z)}{\Lambda_{r-i+b}(q^{-b}z)} \,,\qquad
G^{(r)}_i(z,q)=\frac{(M_{i+1})^1_1}{(M_i)^1_1}\cdot\frac{(M_i)^1_i}{(M_{i+1})^1_{i+1}}\,.
\end{equation}

\subsection{Relation to Berenstein-Fomin-Zelevinsky work on generalized minors}

We devoted this section to the description of Miura $(SL(r+1),q)$-opers via various minors of the q-Wronskian matrix. That matrix is produced by the components of the section of the line bundle and the components 
of the constant regular element $Z$, representing the q-connection in the given trivialization. One may wonder if such construction exists in the general case, for simply connected simple group $G$, namely whether there exists an analogue of the q-Wronskian. Of course, in that case, we do not have a line bundle, since the definition of  $(G,q)$ in terms of the flag of bundles is $SL$-specific. Nevertheless, there is a notion of generalized minors \cite{Berenstein_1997,BERENSTEIN199649,FZ}. These are the functions on $G$, defined on the dense set corresponding to the dense Bruhat cell $N_-HN_+$. For any $g=n_+hn_-$, the so-called {\it principal minors} $[g]^{\omega_i}$ are defined as the value of the multiplicative characters  $[~\cdot~]^{\omega_i}:H\to \mathbb{C}^*$ on $h$, namely $[h]^{\omega_i}$, corresponding to the fundamental weight $\omega_i$ for $i=1,\dots, r$.  
Other generalized minors are obtained by the action of the Weyl group elements on the left and the right of $g$ and then applying  the appropriate lifts of Weyl group elements $u,v$ on the right and the left and then applying $[~\cdot~]^{\omega_i}$, thus producing generalized minors $\Delta_{u\omega_i, v\omega_i}$.    
In the case of $SL(r+1)$, the nondegeneracy conditions imply that the full q-Wronskian matrix belongs to the dense Bruhat cell (i.e. it has Gauss decomposition) and the action of the Weyl group elements correspond to 
the permutations of rows and columns. 

One of the fundamental relations between generalized minors is as follows \cite{FZ}. Let, $u,v\in W$, such that for 
$i\in \{1, \dots, r\}$,  $\ell(us_i)=\ell(u)+1$,  $\ell(vs_i)=\ell(v)+1$. Then 
$$
\Delta_{u\omega_i, v\omega_i}\Delta_{us_i\omega_i, vs_i\omega_i}-
\Delta_{us_i\omega_i, v\omega_i}\Delta_{u\omega_i, vs_i\omega_i}=\prod_{j\neq i}\Delta_{u\omega_j, v\omega_j}^{-a_{ji}}, 
$$

When applied to the q-Wronskian matrix in $SL(r+1)$ case, these equations reproduce the $\mathscr{D}\mathscr{D}$-system. In the case of general $G$ the left and right-hand sides of this relation is very similar to the analogue of $\mathscr{D}\mathscr{D}$-system for general $G$ (see \cite{Frenkel:2020}).
Thus it is reasonable to assume the existence of the analogue of the Wronskian matrix as an element in $n_-(z)h(z)n_+(z)\in N_-(z)H(z)N_+(z)\subset G(z)$. We will discuss this in future work.

Note that one important feature of generalized minors is that relations between them give cluster algebra structure for double Bruhat cells, so that our B\"acklund transformations descend from mutations for the cluster algebra elements. 

We believe that these cluster structures stand behind known cluster structures relevant for Grothendieck rings 
of quantum affine algebras.

\section{$\overline{GL}(\infty)$ and the Fermionic Fock Space} \label{Sec:FockSpace}
This section contains the material on infinite-dimensional generalizations of $GL(N)$ and their representations which will be later needed. The reader may consult with \cite{doi:10.1142/8882} for more details.

\subsection{$(SL(r+1),q)$-Miura opers and the fermionic Fock space} 
First, we note that given a defining representation $V_{\omega_1}$ of $SL(r+1)$, one can construct all other fundamental representations $V_{\omega_i}$ by considering wedge powers $\simeq \Lambda^i (V_{\omega_1})$. If $\nu_1, \dots, \nu_{r+1}$ are the standard basis vectors in $V_{\omega_1}$, so that $\nu_1$ is the highest weight, then the highest weight vectors in $V_{\omega_i}$ are $$\nu_i\wedge\nu_{i-1}\dots\wedge \nu_1.$$
Introducing operators $\psi_i$ of exterior multiplication on $\nu_i$ and $\psi^*_i$ of interior multiplication by $\nu_i$, we find that they satisfy Clifford algebra relation 
\begin{equation}\label{cliff}
\psi^*_i\psi_j+\psi_j\psi^*_i=\delta_{ij}\,.
\end{equation} 
Using those operators we can realize the Chevalley generators as follows: 
\begin{eqnarray}\label{cliffgen}
\check{\alpha}_i=\psi_i\psi^*_i-\psi_{i+1}\psi^*_{i+1}, \quad e_i=\psi_i\psi^*_{i+1}, \quad f_i=\psi_{i+1}\psi^*_{i},
\end{eqnarray}
such that $[e_i,f_i]=\check{\alpha}_i$. We arrive at the following Proposition.
\begin{Prop}
In any fundamental representation the q-connection, corresponding to $(SL(r+1),q)$-Miura oper \eqref{form of A1} reads as follows:
\begin{eqnarray}
A(z)=\prod^1_{i=r}g_i^{\check{\alpha}_i}e^{\frac{\Lambda_i(z)}{g_i(z)}e_i}=
g_{r}^{-\psi_{r+1}\psi^*_{r+1}}\prod^{1}_{i=r}e^{\Lambda_i(z)\psi_i\psi^*_{i+1}}\cdot\Big[\frac{g_i}{g_{i-1}}\Big]^{\psi_i\psi_i^*}\,,
 \end{eqnarray}
where $g_0=1$.
\end{Prop}
Our goal in the following will be to make sense of the completion of the above formula in the infinite-dimensional Fock space.

\subsection{$SL(\infty)$ and its Completions}
In the following we review some basic facts from \cite{doi:10.1142/8882} on representations of infinite dimensional Lie algebras and groups. 
We can define the group $GL(\infty)$ as a set of infinite-dimensional matrices which can be characterized as follows:
\begin{eqnarray}
&&GL(\infty)=\nonumber \\
&&\{A=(a_{ij})_{i,j\in \mathbb{Z}}|~ A \;\;{\rm is ~ invertible~ and~ all~ but ~finite~ number~ of~} a_{ij}-\delta_{ij} {\rm ~are ~}0\}
\end{eqnarray}
The $SL(\infty)$ is the subgroup of $GL(\infty)$ of unimodular matrices.  
The Lie algebra $\mathfrak{gl}(\infty)$ of $GL(\infty)$ is given by 
\begin{eqnarray}
&&\mathfrak{gl}(\infty)=\{A=(a_{ij})_{i,j\in \mathbb{Z}}| ~a_{ij}=0~ {\rm ~for ~all ~but~ finite~ number}\}
\end{eqnarray}
and $\mathfrak{sl}(\infty)$ is the subalgebra of traceless matrices.
The Lie algebra $\mathfrak{sl}(\infty)$ is the explicit realization of the simple Kac-Moody algebra $\mathfrak{a}_{\infty}$, which one associates to the infinite Dynkin diagram $A_{\infty}$.  
However, there exist a bigger algebra, known as $\bar{\mathfrak{a}}_{\infty}$, which 
consists of elements of the form:
\begin{eqnarray}\label{infelem}
x=\sum_{i \in\mathbb{Z}}c_i\check{\alpha}_i+\sum_{\alpha}\eta_{\alpha}e_{\alpha}
\end{eqnarray}
where $e_\alpha$ is an element of Cartan-Weyl basis corresponding to the root $\alpha$ with the height ${\rm ht}(\alpha)$ of $\mathfrak{sl}(\infty)$, so that the set 
$$S_x=\{k\in \mathbb{Z}\big|\exists \alpha, \eta_{\alpha}\neq 0,\,{\rm ht}(\alpha)=k\}$$
is finite. 
This algebra has two nontrivial central elements 
$c_1=\sum_ii \check{\alpha}_i$ and $c_2=\sum_i \check{\alpha}_i$. The explicit realization of this algebra is given by the central extension of the algebra $\mathfrak{gl}_\infty$:
\begin{eqnarray}\label{glinf}
&&\overline{\mathfrak{gl}}_{\infty}=\{A=(a_{ij})_{i,j\in \mathbb{Z}}| ~a_{ij}=0~ {\rm for} ~|i-j|\gg 0\}.
\end{eqnarray}
Namely, there exists a homomorphism from $\bar{\mathfrak{a}}_{\infty}$ to $\overline{\mathfrak{gl}}_\infty\oplus\mathbb{C}\mathbf{c}$ where $c_1$ is mapped to the identity matrix and $c_2=\mathbf{c}$ correspond to the central extension $\mathbf{c}$. Indeed, one can modify relations on the fundamental generators of $\mathfrak{a}_{\infty}$, namely 
\begin{eqnarray}
[e_0,f_0]_c=\check{\alpha}_0+\mathbf{c}
\end{eqnarray}  
leaving all other relations between Chevalley generators intact. This leads to a nontrivial central extension for $\overline{\mathfrak{gl}}_{\infty}$, although for any $\mathfrak{gl}(n)$ subalgebra, this central extension is trivial. 

In order to describe these algebraic structures it is convenient to use matrix notation.
Let us denote by $E_{ij}$ the matrix whose $(i,j)$ entry is $1$ and the rest are equal to zero. These matrices obey the following commutation relations 
$$
[E_{ij},E_{mn}]= \delta_{jm}E_{in}-\delta_{ni}E_{mj}\,.
$$
One can then represent 
$$
\check{\alpha}_i=E_{ii}-E_{i+1,i+1},\qquad e_i=E_{i,i+1}, \qquad f_i=E_{i,i-1}\,.
$$ 
Let us define
$$
a_{i}=\sum_{k\in \mathbb{Z}}E_{k,k+i}\,,\qquad i\neq 0
$$ 
and
$$
a_0 = \sum_{k>0}E_{k,k}-\sum_{k\leq 0}E_{k,k}\,.
$$
Then we have the following Heisenberg subalgebra: 
$$
[a_{n}, a_m]=n\mathbf{c}\delta_{n,-m},
$$

However, we will be interested in a smaller subalgebra $\bar{\mathfrak{a}}'_{\infty}\subset \bar{\mathfrak{a}}_{\infty}$, 
so that for every $x\in \bar{a}'_{\infty}$ in the form (\ref{infelem}) only finite number of coefficients $\lambda_{\alpha}\neq 0$ for negative $\alpha$. The corresponding subalgebra 
$\overline{\mathfrak{gl}}'_{\infty}\subset \overline{\mathfrak{gl}}_{\infty}$ is formed by matrices $(\ref{glinf})$ with only finite number of elements below the main diagonal.
The corresponding Lie group is denoted $\overline{GL}(\infty)$: 
\begin{eqnarray}\label{GLinf}
&&\overline{GL}(\infty)=\nonumber \\
&&\{A=(a_{ij})_{i,j\in \mathbb{Z}}| ~a_{ij}=0; ~k ~
{\rm for} ~i>j ~{\rm for} ~{\rm all }~{\rm but} ~{\rm finite} 
~{\rm number}; ~a_{ii}\neq 0 
\,\, \forall i\in \mathbb{Z}\}.
\end{eqnarray}

Given the upper Borel part $\mathfrak{b}_{+}$ of $\overline{\mathfrak{gl}}_{\infty}$, generated by $\check{\alpha}_i$, $e_i$, one can construct an upper Borel  subgroup $\overline{B}_+$ by exponentiating elements of $b_{+}$, which we denote as $B_{+}$, namely
\begin{eqnarray}
\overline{B}_{+}=\{A=(a_{ij})_{i,j\in \mathbb{Z}}|~a_{ij}=0~ {\rm for} ~i>j, ~~a_{ii}\neq 0 
\,\forall i \in \mathbb{Z}\}
\end{eqnarray}
Combining it with $B_-$, the Borel subgroup of $SL(\infty)$:
\begin{eqnarray}
&&\overline{B}_{-}=\\\nonumber
&&\{A=(a_{ij})_{i,j\in \mathbb{Z}}|~a_{ij}=0~ {\rm for} ~i<j,\nonumber\\
&&a_{ii}\neq 0~{\rm for ~all}~ i, ~~a_{ii}\neq 1 
~{\rm for} ~{\rm all }~{\rm but} ~{\rm finite} 
~{\rm number}, \\
&&~a_{ij}\neq 0~ {\rm for} ~i>j ~{\rm for} ~{\rm all }~{\rm but} ~{\rm finite} \nonumber
~{\rm number}, ~\det(A)=1\}
\end{eqnarray}

Then one can write Bruhat decomposition $\overline{GL}(\infty)=\sqcup_{\bar{w}} \overline{B}_-\bar{w}\overline{B}_+$, where $\bar{w}$ is a Weyl group element inherited from a Weyl group element of  $SL(k+1)$ subgroup for some finite $k$. 

Now we can construct the appropriate generalization of the q-connection (\ref{form of A}) :
\begin{equation}    \label{testA}
A(z)=\prod^{-\infty}_{i={+\infty}}
g_i(z)^{\check{\alpha}_i} \; e^{\frac{\Lambda_i(z)}{g_i(z)}e_i}, \qquad
g_i(z) \in \C(z)^\times, \quad \Lambda_i(z)\in \C[z].
\end{equation}
which is a well-defined element of $\overline{B}_+(z)\subset \overline{GL}(\infty)(z)$. Indeed, while it is an infinite product the multiplication is well-defined giving the element of $\overline{B}_+(z)\subset \overline{GL}(\infty)(z)$ with nonzero elements on diagonal and superdiagonal only. In the next subsection we give a simpler expressions for \eqref{testA} in the fundamental representations of  $\overline{GL}(\infty)$.

\begin{Rem}
In principle, it could be possible to consider further completions of $\overline{GL}(\infty)$ and make full use of the  central extension, however, since we are interested in Miura q-opers, we only need to complete one of the Borels to arrive to the formula above.
\end{Rem}

\subsection{Infinite wedge space representations for $\overline{GL}(\infty)$}\label{fock}

Here we will explain the construction of the fundamental representations of $\bar{\mathfrak{a}'}_{\infty}$ with central charge 1, which will serve as fundamental representations for $\overline{GL}(\infty)$ as well. 

Let $V=\oplus_{j\in \mathbb{Z}} \mathbb{C} \nu_j$ be the infinite-dimensional space where $\nu_j$ are basis elements. There is a natural action of $\mathfrak{sl}(\infty)$ on $V$ as infinite-dimensional matrices. 
Consider the following expression:
\begin{eqnarray}
\Psi_m=\nu_m \wedge \nu_{m-1}\wedge \nu_{m-2}\wedge \dots
\end{eqnarray}
We will call it the highest weight vector in the vector space $F_m$. The other basis vectors in $F_m$ have the form 
\begin{eqnarray}
\Psi=\nu_{i_m} \wedge \nu_{i_{m-1}}\wedge \nu_{i_{m-2}}\wedge \dots, 
\end{eqnarray}
where 
$i_m>i_{m-1}>i_{m-2}>\dots$ and $i_k=k$ for $k\ll 0$.
Action of $\mathfrak{sl}(\infty)$ algebra on $F_m$ is defined in the following way. We identify $e_i$, $f_i, \check{\alpha}_i$ with matrix generators  $E_{i,i+1}$, $E_{i+1,i}$, $E_{ii}-E_{i+1,i+1}$ respectively. Then we define the action of any element $X$ of $\mathfrak{sl}(\infty)$ on $F_m$ is given by the following formula:
\begin{eqnarray}
&&X \Psi=\nonumber \\
&&X\nu_{i_m} \wedge \nu_{i_{m-1}}\wedge \nu_{i_{m-2}}\wedge \dots +\nu_{i_m} \wedge X\nu_{i_{m-1}}\wedge \nu_{i_{m-2}}\wedge \dots + \nu_{i_m} \wedge \nu_{i_{m-1}}\wedge X \nu_{i_{m-2}}\wedge \dots\nonumber
\end{eqnarray}
\begin {Rem}A famous representation of $\bar{\mathfrak{a}}_{\infty}$ with central charge $\mathbf{c}=1$ is achieved in the folowing way. 
One has to modify the action of $\check{\alpha}_0$, via a shift $\check{\alpha}_0\to \check{\alpha}_0-1$, namely:
\begin{eqnarray}
&&\check{\alpha}_0 \Psi=\nonumber \\
&&\check{\alpha}_0 \nu_{i_m} \wedge \nu_{i_{m-1}}\wedge \nu_{i_{m-2}}\wedge \dots +\nu_{i_m} \wedge \check{\alpha}_0 \nu_{i_{m-1}}\wedge \nu_{i_{m-2}}\wedge \dots + \nu_{i_m} \wedge \nu_{i_{m-1}}\wedge \check{\alpha}_0 \nu_{i_{m-2}}\wedge \dots\nonumber-\nonumber \\
&&v_{i_m} \wedge \nu_{i_{m-1}}\wedge \nu_{i_{m-2}}\wedge \dots\nonumber 
\end{eqnarray}
\end{Rem}
Notice that $\bar{\mathfrak{n}}_+\Psi_m=0$ , where $\bar{\mathfrak{n}}_+=[\bar{\mathfrak{b}}_+, \bar{\mathfrak{b}}_+]$ and $\check{\alpha}_k\Psi_m=\delta_{k,m}\Psi_m$. Thus $\{F_m\}$ can be interpreted as fundamental representations of $\bar{\mathfrak{a}'}_{\infty}$ and  fundamental representations of  $\overline{GL}(\infty)$ as well. The group action is given by the formula:
\begin{eqnarray}
g\cdot \Psi=g\nu_{i_m} \wedge g\nu_{i_{m-1}}\wedge g\nu_{i_{m-2}}\wedge \dots 
\end{eqnarray}

Using the formalism of the Clifford algebra (\ref{cliff}), we have again formulas (\ref{cliffgen}) for the generators $\check{\alpha}$, $e_i$, $f_i$, where now $i\in \mathbb{Z}$. This allows us to write the expression for the element of $\overline{B}_+(z)$ from (\ref{testA}) acting on $F_m$ as 
\begin{eqnarray}\label{eq:AExplicitPsi}
A(z)=\prod^{1}_{i=+\infty}e^{\Lambda_i(z)\psi_i\psi^*_{i+1}}\Big[\frac{g_i}{g_{i-1}}\Big]^{\psi_i\psi_i^*}\cdot
\prod_{i=0}^{-\infty}e^{\Lambda_i(z)\psi_i\psi^*_{i+1}}\Big[\frac{g_i}{g_{i-1}}\Big]^{-\psi_i^*\psi_i}\,.
 \end{eqnarray}

\section{$(\overline{GL}(\infty),q)$-opers}\label{Sec:GLinfOpers}

In this and the next section we generalize the definitions and theorems from Sections 2,3,4,5 to the infinite-dimensional case. Namely,
the (Miura) $(\overline{GL}(\infty),q)$-opers, $Z$-twisted and $Z$-twisted Mi\"ura-Pl\"ucker versions, as well as nondegeneracy conditions. Then we relate them to QQ-systems and describe explicitly the trivializing operator for the related $Z$-twisted Mi\"ura q-oper. In the process we have to take into account the generally infinite number of zeroes and poles in the local expression.
The explicit formulas will also change slightly as well. 

A particularly interesting part of the infinite dimensional case is the infinite flag in the associated bundle version of the definition of $q$-oper, which will not involve the ``starting" line subbundle. Thus in the study of  $Z$-twisted Miura $(\overline{GL}(\infty),q)$-opers we have to rely on the QQ-system only, without addressing the q-Wronskian approach.

\subsection{Definitions of $(\overline{GL}(\infty),q)$-opers and the canonical form of $(\overline{GL}(\infty),q)$-Miura opers}

Given a principal $\overline{GL}(\infty)$-bundle $\cF_{\overline{GL}(\infty)}$ over $\P^1$, let $\cF_{\overline{GL}(\infty)}^q$ denote its pullback under the map $M_q: \P^1
\to \P^1$ sending $z\mapsto qz$. A meromorphic $(\overline{GL}(\infty),q)$-{\em
  connection} on a principal $\overline{GL}(\infty)$-bundle $\cF_{\overline{GL}(\infty)}$ on $\P^1$ is a section
$A$ of $\Hom_{\cO_{U}}(\cF_{\overline{GL}(\infty)},\cF_{\overline{GL}(\infty)}^q)$, where $U$ is an open
dense subset of $\P^1$ in the standard topology. Notice, that now the number of zeroes and poles which we have to exclude from 
$\P^1$ could be infinite. We assume that the only two accumulations points possible are $0,\infty$. 
We can always choose $U$ so that the
restriction $\cF_{\overline{GL}(\infty)}|_U$ of $\cF_{\overline{GL}(\infty)}$ to $U$ is isomorphic to the trivial
$\overline{GL}(\infty)$-bundle. The restriction
of $A$ to the Zariski open dense subset $U \cap M_q^{-1}(U)$ can be
written as section of the trivial $\overline{GL}(\infty)$-bundle on $U \cap M_q^{-1}(U)$,
and hence as an element $A(z)$ of $\overline{GL}(\infty)(z)$.

\begin{Def}    \label{infqop}
  A meromorphic $(\overline{GL}(\infty),q)$-{\em oper}  on
  $\mathbb{P}^1$ is a triple $(\cF_{\overline{GL}(\infty)},A,\cF_{\overline{B}_-})$, where $A$ is a
  meromorphic $(\overline{GL}(\infty),q)$-connection on a $\overline{GL}(\infty)$-bundle $\cF_{\overline{GL}(\infty)}$ on
  $\mathbb{P}^1$ and $\mathcal{F}_{\overline{B}_-}$ is the reduction of $\cF_{\overline{GL}(\infty)}$
  to $\overline{B}_-$ satisfying the following condition: there exists an 
  open dense subset $U \subset \P^1$ together with a trivialization
  $\imath_{\overline{B}_-}$ of $\mathcal{F}_{\overline{B}_-}$, such that the restriction of
  the connection $A: \cF_{\overline{GL}(\infty)} \to \cF_{\overline{GL}(\infty)}^q$ to $U \cap M_q^{-1}(U)$,
  written as an element of $\overline{GL}(\infty)(z)$ using the trivializations of
  ${\mathcal F}_{\overline{GL}(\infty)}$ and $\cF_{\overline{GL}(\infty)}^q$ on $U \cap M_q^{-1}(U)$ induced by
  $\imath_{B_-}$,  takes values in the infinte product of Bruhat cells $\prod^{-\infty}_{i=+\infty}\overline{B}_-(\C[U \cap
  M_q^{-1}(U)]) s_i \overline{B}_-(\C[U \cap M_q^{-1}(U)])$, where the ordering in the product follows the infinite version of the one in $SL(r+1)$. 
\end{Def}

Therefore any $q$-oper connection
$A$ can be written in the form
\begin{eqnarray}    \label{infqop1}
A(z)=\prod^{-\infty}_{i=+\infty}\Big[n'_i(z) (\phi_i(z)^{\check{\alpha}_i} \, s_i )n_i(z)\Big]
\end{eqnarray}
where $\phi_i(z) \in\C(z)$ and  $n_i(z), n_i'(z)\in \bar{N}_-=[\bar B_-,\bar B_-](z)$ are such that
their zeros and poles are outside the subset $U \cap M_q^{-1}(U)$ of
$\P^1$. As we stated before, we require that the only accumulation points of zeroes and poles of $\phi_i(z)$, $n_i(z), n_i'(z)$ are $0, \infty$.

We can give an alternative definition of $(\overline{GL}(\infty),q)$-oper connection using associated bundles as well.

\begin{Def}    \label{infqopflag}
  A meromorphic $(\overline{GL}(\infty),q)$-{\em oper}  on
  $\mathbb{P}^1$ is a triple $(A,E, \mathcal{L}_{\bullet})$, where $E$ is an ambient vector bundle with the fiber being infinite-dimensional vector space with countable basis and $\mathcal{L}_{\bullet}$ is the corresponding complete flag of the vector bundles, 
  $$...\subset \mathcal{L}_{i+1}\subset\mathcal{L}_i\subset\mathcal{L}_{i-1}\subset...\subset E,$$ i.e.  with the fibers for $\mathcal{L}_i$ being semi-infinite spaces  
so that 
  $A\in \Hom_{\cO_{U}}(E,E^q)$ 
  satisfies the following conditions:
\begin{enumerate} 
\item[i)] $\mathcal{A}\cdot \mathcal{L}_i\subset \mathcal{L}_{i-1} $
\item[ii)]  There exists an open dense subset $U \subset \P^1$, such that the restriction of
   $\mathcal{A}\in Hom(\mathcal{L}_{\bullet}, \mathcal{L}^q_{\bullet})$ to $U \cap M_q^{-1}(U)$ belongs to $\overline{GL}(\infty)(z)$ and satisfies the condition that the induced maps
  $\bar{\mathcal{A}}_i:\mathcal{L}_{i}/\mathcal{L}_{i+1}\to \mathcal{L}_{i-1}/\mathcal{L}_{i}$ are isomorphisms on $U \cap M_q^{-1}(U)$.\\
\end{enumerate}
\end{Def}

The equivalence between two definitions can be established as in the finite-dimensional case, using the associated bundle for the defining representation and its faithfulness. We also will use the notation $A$ for the associated version of the q-connection $\mathcal{A}$ as well: it will be  clear which one is used based on the context. 

Let us give two equivalent definitions of  the $(\overline{GL}(\infty),q)$-Miura oper, which is the same as in finite-dimensional case.

\begin{Def}    \label{Miurainf}
\begin{enumerate}
  \item[i)] A {\em Miura $(\overline{GL}(\infty),q)$-oper} on $\mathbb{P}^1$ is a quadruple
  $(\cF_{\overline{GL}(\infty)},A,\cF_{\overline{B}_-},\cF_{\overline{B}_+})$, where $(\cF_{\overline{GL}(\infty)},A,\cF_{\overline{B}_-})$ is a
  meromorphic $(\overline{GL}(\infty),q)$-oper on $\P^1$ and $\cF_{\overline{B}_+}$ is a reduction of
  the $\overline{GL}(\infty)$-bundle $\cF_{\overline{GL}(\infty)}$ to $\overline{B}_+$ that is preserved by the
  $q$-connection $A$.

\item[ii)]   A {\em Miura $(\overline{GL}(\infty),q)$-oper} on $\mathbb{P}^1$ is a quadruple
  $(E, A, \mathcal{L}_{\bullet}, \hat{\mathcal{L}}_{\bullet})$, where $(E, A, \mathcal{L}_{\bullet})$ is a
  meromorphic $\overline{GL}(\infty)$-oper on $\P^1$ and $\hat{\mathcal{L}}_{\bullet}=\{\mathcal{L}_i\}$,  is another full flag of subbundles in $E$ that is preserved by the
  $q$-connection $A$.
\end{enumerate}
\end{Def}

As in $SL(r+1)$ case, we can define relative position (see Section \ref{Miurafinsec}) between $\cF_{\overline{B}_+}, \cF_{\overline{B}_-}$ because of the Bruhat decomposition of $G$. We will say that $\cF_{\overline{B}_-}$ and $\cF_{\overline{B}_+}$ have a {\em generic
  relative position} at $x \in X$ if the element of $W_G$ assigned to
them at $x$ is equal to $1$ (this means that the corresponding element
$a^{-1}b$ belongs to the open dense Bruhat cell $\overline{B}_- \cdot \overline{B}_+ \subset
\overline{GL}(\infty)$).

We immediately have the following result, which is a generalization of a finite-dimensional case 

\begin{Thm}    \label{gen rel pos inf}
  For any Miura $(\overline{GL}(\infty),q)$-oper on $\mathbb{P}^1$, there exists an open
  dense subset $V \subset \P^1$ such that the reductions $\cF_{\overline{B}_-}$
  and $\cF_{\overline{B}_+}$ are in generic relative position for all $x \in V$.
\end{Thm}

\begin{proof}
Notice, that according to the local expression for the q-oper connection (\ref{qop1})  
and the condition that it belongs to $\overline{GL}(\infty)(z)$ means that there is a finite number of elements below the diagonal. This means that for some $k,l$ we have
\begin{eqnarray} 
A(z)=\Bigg[\prod_{i=+\infty}^kg_i^{\check{\alpha}_i}(z)e^{\frac{\phi_i(z)e_i}{g_i(z)}}\Bigg]n'(z)\prod^{l+1}_{j=k-1}(\phi_j(z)^{\check{\alpha}_j} \, s_j )n(z)\Bigg[\prod^{-\infty}_{i=l}g_i^{\check{\alpha}_i}(z)e^{\frac{\phi_i(z)e_i}{g_i(z)}}\Bigg],
\end{eqnarray}
where $n(z), n'(z)\in \overline{N}_-(z)$ belong to the $SL(k+l)$ subgroup with $H$ generated by $\{\check \alpha_j\}^{j=l+1}_{j=k-1}$. The expression in the middle, namely  $A'(z)=n'(z)\prod_j(\phi_j(z)^{\check{\alpha}_j} \, s_j )n(z)\in SL(k+l)(z)$ is the local expression for Miura $SL(k+l)$ q-oper for which generic property follows from the finite-dimensional case (see Theorem \ref{gen rel pos}) and thus we have generic relative position for Miura $(\overline{GL}(\infty),q)$-oper.
\end{proof}

That leads to the following Corollary.

\begin{Cor}    \label{gen rel pos1}
For any Miura $(\overline{GL}(\infty),q)$-oper on $\mathbb{P}^1$, there exists a
trivialization of the underlying $\overline{GL}(\infty)$-bundle $\cF_{\overline{GL}(\infty)}$ on an open
dense subset of $\P^1$ for which the oper $q$-connection has the form
\begin{equation}    \label{gicheckinf}
\prod^{-\infty}_{i=+\infty} g_i^{\check{\alpha}_i}e^{\frac{\lambda_i t_i}{g_i}e_i}, \qquad
g_i \in \mathbb{C}(z)^\times,
\end{equation}
where each $t_i \in \mathbb{C}(z)$ is determined by the lifting $s_i$ and the order in the product is canonical.
\end{Cor}

As in the finite-dimensional case,  we fix $t_i\equiv 1$ from now on. 

\subsection{Z-twisted Miura q-opers }
Now we are ready to define the notion of $(\overline{GL}(\infty),q)$-oper and  $(\overline{GL}(\infty),q)$-Miura oper, which are straightforward definitions of their $SL(r+1)$-counterparts. As in $SL(r+1)$ case, let  
$Z$ be the regular element of the maximal torus $\overline{H}=\overline{B}_+/[\overline{B}_+,\overline{B}_+]$. One can express it as follows:
\begin{equation}    \label{Zinf}
Z = \prod_{i={+\infty}}^{-\infty} \zeta_i^{\check\alpha_i}, \qquad \zeta_i \in
\C^\times.
\end{equation}

\begin{Def}    \label{Ztwoperinf}
  A {\em $Z$-twisted $(\overline{GL}(\infty),q)$-oper} on $\mathbb{P}^1$ is a $(\overline{GL}(\infty),q)$-oper
  that is equivalent to the constant element $Z \in \overline{H} \subset \overline{H}(z)$
  under the $q$-gauge action of $\overline{GL}(\infty)(z)$, i.e. if $A(z)$ is the
  meromorphic oper $q$-connection (with respect to a particular
  trivialization of the underlying bundle), there exists $g(z) \in
  \overline{GL}(\infty)(z)$ such that
\begin{eqnarray}    \label{Aginf}
A(z)=g(qz)Z g(z)^{-1}.
\end{eqnarray}
A {\em $Z$-twisted Miura $(\overline{GL}(\infty),q)$-oper} is a Miura $(\overline{GL}(\infty),q)$-oper on
$\mathbb{P}^1$ that is equivalent to the constant element $Z \in \overline{H}
\subset \overline{H}(z)$ under the $q$-gauge action of $\overline{B}_+(z)$, i.e.
\begin{eqnarray}    \label{gaugeAinf}
A(z)=v(qz)Z v(z)^{-1}, \qquad v(z) \in \overline{B}_+(z).
\end{eqnarray}
\end{Def}

Naturally, we have the proposition addressing characterization of $Z$-twisted Miura q-opers 
associated to $Z$-twisted q-opers.

\begin{Prop}    \label{Z prime inf}
  Let $Z \in \overline{H}$ be regular. For any $Z$-twisted $(\overline{GL}(\infty),q)$-oper $(\cF_{\overline{GL}(\infty)},A,\cF_{\overline{B}_-})$
  and any choice of $\overline{B}_+$-reduction $\cF_{\overline{B}_+}$ of $\cF_{\overline{GL}(\infty)}$ preserved
  by the oper $q$-connection $A$, the resulting Miura $(\overline{GL}(\infty),q)$-oper is
  $Z'$-twisted for a particular $Z' \in S_{\infty} \cdot Z$. The set of 
  $A$-invariant $\overline{B}_+$-reductions $\cF_{\overline{B}_+}$ on the 
  $(\overline{GL}(\infty),q)$-oper is in one-to-one correspondence with the elements of $W=S_{\infty}$.
\end{Prop}

Given a Miura $(\overline{GL}(\infty),q)$-oper. By Corollary
\ref{Miura form}, the underlying $(G,q)$-connection can be written in the
form \eqref{form of A}. As in $SL(r+1)$ case we obtain an $\bar{H}$-bundle
$\cF_{\overline{B}_+}/\bar{N}_+$, where $\bar{N}_+=[\overline{B}_+, \overline{B}_+]$. The corresponding $(\bar{H},q)$-connection $A^{\bar H}(z)$ according to \eqref{gicheckinf} is
given by:
\begin{equation}    \label{AHinf}
A^H(z)=\prod_ig_i(z)^{\check{\alpha}_i}.
\end{equation}
We call $A^{\bar H}(z)$ the \emph{associated Cartan $q$--connection} of the
Miura $q$-oper $A(z)$.

The same can be done in the infinite-dimensional case. If our Miura $q$-oper is $Z$-twisted (see Definition
\ref{Ztwoperinf}), then we also have $A(z)=v(qz)Z v(z)^{-1}$, where
$v(z)\in B_+(z)$.  Notice, that  $v(z)$ can be written as
\begin{equation}    \label{vzinf}
v(z)=
\prod_i y_i(z)^{\check{\alpha}_i} n(z), \qquad n(z)\in \bar{N}_+(z), \quad
y_i(z) \in \C(z)^\times,
\end{equation}
We refer to the associated to Cartan $q$-connection $A^{\bar H}(z)$ as $Z$-{\emph twisted}, so that 
the explicit realization is given by the following formula:
\begin{equation}    \label{AHZinf}
A^H(z)=\prod_i
\Bigg[\zeta_i\frac{y_i(qz)}{y_i(z)}\Bigg]^{\check{\alpha}_i} 
\end{equation}
and we note, that $A^H(z)$ determines the $y_i(z)$'s uniquely
up to a scalar.

\section{$Z$-twisted Miura $(\overline{GL}(\infty),q)$-opers and $QQ$-systems}  \label{Sec:MiuraPluckerInf}

\subsection{Definition and explicit realization.}
Let $\{ \Lambda_i(z) \}_{i\in \mathbb{Z}}$ be a collection of
non-constant polynomials with accumulation points of roots at $0$ or $\infty$ only.

\begin{Def}    \label{d:regsinginf}

i)  An $(\overline{GL}(\infty),q)$-{\em oper with regular singularities determined by $\{
    \Lambda_i(z) \}_{i\in\mathbb{Z}}$} is a $q$-oper on $\P^1$ whose
  $q$-connection may be written in the form
\begin{equation}    \label{Lambdainf}
A(z)= \Bigg[\prod_{i=+\infty}^kg_i^{\check{\alpha}_i}(z)e^{\frac{\Lambda_i(z)t_ie_i}{g_i(z)}}\Bigg]n'(z)\prod^{l+1}_{j=k-1}(\phi_j(z)^{\check{\alpha}_j} \, s_j )n(z)\Bigg[\prod^{-\infty}_{i=l}g_i^{\check{\alpha}_i}(z)e^{\frac{\Lambda_i(z)t_ie_i}{g_i(z)}}\Bigg], \qquad
\end{equation}
for some $k, l\in \mathbb{Z}$, where $n(z), n'(z)\in N_-(z)$ and belong to  $SL(k+l)$ subgroup with $H$ generated by $\{\check \alpha_j\}^{j=l+1}_{j=k-1}$.

ii)  {\em A Miura $(\overline{GL}(\infty),q)$-oper with regular singularities determined by
polynomials $\{ \Lambda_i(z) \}_{i=1,\ldots,r}$} is a Miura
  $(\overline{GL}(\infty),q)$-oper such that the underlying $q$-oper has
regular singularities determined by $\{ \Lambda_i(z)
\}_{i=1,\ldots,r}$.
\end{Def}

As in the $SL(r+1)$ case, from now on we set $t_i(z)=1$, $i\in \mathbb{Z}$. 
Then we have an analogue of the Theorem \ref{Miura form}.

\begin{Cor}    \label{Miura form inf}
For every Miura $(\overline{GL}(\infty),q)$-oper with regular singularities determined by
the polynomials $\{ \Lambda_i(z) \}_{i\in\mathbb{Z}}$, the underlying
$q$-connection can be written in the form 
\begin{equation}    \label{form of A inf}
A(z)=\prod^{-\infty}_{i={+\infty}}
g_i(z)^{\check{\alpha}_i} \; e^{\frac{\Lambda_i(z)}{g_i(z)}e_i}, \qquad
g_i(z) \in \C(z)^\times.
\end{equation}
\end{Cor}

\subsection{Fermionic realization}

Let $F_i$ be the irreducible representation of $\overline{GL}(\infty)$ with the highest weight $\omega_i$ which we discussed in Section 7.  Notice, that the one-dimensional and two-dimensional subspaces $L_i$ and $W_i$ of
$F_i$ spanned by the weight vectors $\Psi_i$, and $f_i \cdot\Psi_i$ are a $\overline{B}_+$-invariant subspaces of $V_i$.

Now let $(\cF_{\overline{GL}(\infty)},A,\cF_{\overline{B}_-},\cF_{\overline{B}_+})$ be a Miura $(\overline{GL}(\infty),q)$-oper
with regular singularities determined by polynomials $\{ \Lambda_i(z)
\}_{i\in{\mathbb{Z}}}$ (see Definition \ref{d:regsing}). Recall that
$\cF_{\overline{B}_+}$ is a $\overline{B}_+$-reduction of a $i\in\mathbb{Z}$-bundle $\cF_{(\overline{GL}(\infty)}$ on $\P^1$, 
preserved by the $(\overline{GL}(\infty),q)$-connection $A$. Therefore for each
$i\in\mathbb{Z}$, the vector bundle
$$
\mathscr{F}_i = \cF_{\overline{B}_+} \underset{\overline{B}_+}\times V_i = \cF_{G} \underset{G}\times
F_i.
$$

Thus we have the following Proposition.

\begin{Prop}
For every Miura $(\overline{GL}(\infty),q)$-oper with regular singularities determined by
the polynomials $\{ \Lambda_j(z) \}_{j\in\mathbb{Z}}$, the underlying
$q$-connection $\phi_i(A)$ in the associated bundle $\mathscr{F}_i$ for any $i\in \mathbb{Z}$ can be written in the form:
 \begin{equation}
\phi_i(A)(z)=
\prod^{1}_{j=+\infty}e^{\Lambda_j(z)\psi_{j}\psi^*_{j+1}} \Big[\frac{g_j(z)}{g_{j-1}(z)}\Big]^{\psi_j\psi^*_j}\cdot
\prod_{j=0}^{-\infty}e^{\Lambda_j(z)\psi_{j}\psi^*_{j+1}} \Big[\frac{g_j(z)}{g_{j-1}(z)}\Big]^{-\psi^*_j\psi_j}
\,.
\end{equation}
\end{Prop}

Below we will discuss the Miura-Pl\"ucker condition. As we noted in the beginning of Section 8, the key difference between finite-dimensional case and infinite-dimensional one, that we do not have a luxury of having a line bundle, such that the q-Wronskian matrix of the corresponding section will produce minors describing the complete $QQ$-system. We must rely exclusively on the properties of the infinite $QQ$-system, which will allow us to construct the trivializing group element $v(z)$, so that $A(z)=v(qz)Zv(z)^{-1}$ for any $Z$-twisted Miura-Pl\"ucker $(\overline{GL}(\infty),q)$-oper (see Subsection 9.5). 

\subsection{Miura-Pl\"ucker $(\overline{GL}(\infty),q)$-opers} 
For all $i\in \mathbb{Z}$, the infinite rank bundle $\mathscr{F}_i$ contains a rank two
subbundle
$$
\cW_i = \cF_{B_+} \underset{B_+}\times W_i
$$
associated to $W_i \subset F_i$, and $\cW_i$ in
turn contains a line subbundle
$$
\hat{\mathscr L}_i = \cF_{B_+} \underset{B_+}\times L_i
$$
associated to $L_i \subset W_i$.  

Note, that $\phi_i(A)$ preserves subbundles $\cL_i$ and $\cW_i$ of $\mathscr{F}_i$ and thus produces $(GL(2),q)$-oper on $\cW_i$. 
We denote such q-oper by $A_i$ as in subsection \ref{gl2op}.

Notice that $\cW_i$ decomposes into direct sum of two subbundles, $\hat{\mathscr{L}}_i$, preserved by $B_+$ and $\mathscr{L}_i$ with respect to which it satisfies the $(GL(2),q)$-oper condition. We can unify all that in the following Proposition.

\begin{Prop}    \label{2flagthminf}
The quadruple $(A_i, \cW_i, \mathscr{L}_i, \hat{\mathscr{L}}_i)$ for any $i\in\mathbb{Z}$ forms a $(GL(2),q)$ Miura oper, so that explicitly: 
\begin{equation}    \label{2flagformulainf}
A_i(z)=\begin{pmatrix}
  g_i(z) &  &\Lambda_i(z) g_{i-1}(z)\\
&&\\  
  0 & &g^{-1}_i(z) g_{i+1}(z)g_{i-1}(z)
 \end{pmatrix},\quad i\neq 0\,,\qquad
 A_0(z)=\begin{pmatrix}
  1 &  &\Lambda_0(z) g_{-1}(z)\\
&&\\  
  0 & &g^{-2}_0(z) g_{1}(z)g_{-1}(z)
 \end{pmatrix},
\end{equation}
where we use the ordering of the simple roots determined by the
Coxeter element $c$.
\end{Prop}

We can see that the expression for $A_0(z)$ looks slightly differently than the rest of $A_i(z)$ in \eqref{2flagformulainf}. However, if we multiply $A_0(z)$ by the diagonal matrix proportional to the identity $\text{diag}(g_0(z),\, g_0(z))$ then it will be of the same form as the rest of the matrices. This is due to the central extension in $\bar{\mathfrak{a}}_{\infty}$ algebra and shift of Chevalley generator $\check{\alpha}_0$.

Now we impose the $Z$-twisted condition on the corresponding $A^H$ connection, namely $g_i=\zeta_i\frac{y_i(qz)}{y_i(z)}.$

Let $G_i\cong \SL(2)$ be the subgroup of $\overline{GL}(\infty)$ corresponding to the
 $\mathfrak{sl}(2)$-triple spanned by $\{e_i, f_i,
 \check{\alpha}_i\}$, which preserves $W_i$, using a diagonal gauge transformation as in subsection \ref{gl2op}, we associate to $A_i$ connection a $(G_i,q)$-oper with the explicit form:
\begin{eqnarray}
{\mc A}_i(z)=g_i^{\check \alpha_i}(z) 
e^{\frac{\beta_i(z)}{g_i(z)}e_i}, 
\quad {\rm where} \quad  \beta_i(z)=\Lambda_i(z) \zeta_{i-1} y_{i+1}(z)y_{i-1}(qz).
\end{eqnarray}
Note that the diagonal transformation for ${\mc A}_0(z)$ looks a bit different than for other ${\mc A}_i(z)$ because of the aforementioned shift.

Now we are ready to define Miura-Pl\"ucker  $(\overline{GL}(\infty),q)$-opers.
\begin{Def}    \label{ZtwMPinf}
  A $Z$-{\em twisted Miura-Pl\"ucker $(\overline{GL}(\infty),q)$-oper} is a meromorphic
  Miura $(\overline{GL}(\infty),q)$-oper on $\P^1$ with the underlying $q$-connection
  $A(z)$, such that there exists $v(z) \in \overline{B}_+(z)$ such that for all
  $i\in\mathbb{Z}$, the Miura $(\GL(2),q)$-opers $A_i(z)$ associated to
  $A(z)$ by formula \eqref{2flagformulainf} can be written in the form:
\begin{equation}    \label{gaugeA3inf}
A_i(z) = v(zq) Z v(z)^{-1}|_{W_i} = v_i(zq)Z_iv_i(z)^{-1}
\end{equation}
where $v_i(z) = v(z)|_{W_i}$ and $Z_i = Z|_{W_i}$.
\end{Def}

\subsection{Nondegeneracy conditions}
In this subsection  we will generalize two nondegeneracy conditions we had in subsection \ref{nondegsl}  for $Z$-twisted
Miura-Pl\"ucker $(\overline{GL}(\infty),q)$-opers. 

The first nondegeneracy condition deals with the associated $H$ connection.

\begin{Def}    \label{nondeg Cartaninf}
  A Miura $(\overline{GL}(\infty),q)$-oper $A(z)$ of the form \eqref{form of A} is called
  $H$-\emph{nondegene\-rate} if the corresponding $(\bar{H},q)$-connection
  $A^{\bar H}(z)$ can be written in the form \eqref{AHZinf}, where zeroes and poles $y_i(z)$ and $y_{i\pm 1}(z)$ are $q$-distinct from each other and from the zeros of $\Lambda_k(z)$.
\end{Def}

The second nondegeneracy condition addresses the associated $(G_i,q)$-opers.

\begin{Def}    \label{nondeg Miura inf}
 A $Z$-twisted
  Miura-Pl\"ucker $(\overline{GL}(\infty),q)$-oper $A(z)$ is called \emph{nondegenerate} if
  its associated Cartan $q$-connection $A^H(z)$ is nondegenerate and
  each associated $Z_i$-twisted Miura $(\SL(2),q)$-oper $\mathcal{A}_i(z)$ is
  nondegenerate.
\end{Def}

This we arrive to the analogue of the Proposition \ref{nondeg1}, which is proven in exactly the same way.

\begin{Prop}    \label{nondeg1inf}
  Let $A(z)$ be a 
  $Z$-twisted Miura-Pl\"ucker $(G,q)$-oper.
  The following statements are equivalent:
\begin{enumerate}
\item\label{nondegen1inf} $A(z)$ is nondegenerate.
  \item\label{nondegen2inf} The Cartan $q$-connection $A^H(z)$ is
    nondegenerate, and each
    $A_i(z)$ has regular singularities, i.e. $\rho_i(z)$ given
    by formula \eqref{ri} is in $\C[z]$.
    \item\label{nondegen3inf} Each $y_i(z)$ from formula \eqref{AH1} is a
      polynomial, and for all $i\in \mathbb{Z}$  the zeros of $y_i(z)$ and $y_{i\pm 1}(z)$ are
  $q$-distinct from each other and from the zeros of 
  $\Lambda_k(z)$.
\end{enumerate}
\end{Prop}

\subsection{Z-twisted Miura-Pl\"ucker $(\overline{GL}(\infty),q)$-oper is Z-twisted}

From the previous Section we see that the q-connection of the nondegenerate Miura-Pl\"ucker $(\overline{GL}(\infty),q)$-oper with regular singularities defined by polynomials $\{\Lambda_i(z)\}_{i=1,\dots, r} $ reads as follows:
\begin{equation}\label{form of Ainf}
A(z)=\prod^{-\infty}_{i=+\infty}
g_i(z)^{\check{\alpha}_i} \; e^{\frac{\Lambda_i(z)}{g_i(z)}e_i}, \qquad
g_i(z)=\zeta_i\frac{Q_i^+(qz)}{Q_i^+(z)}\,. 
\end{equation}
Let us assume as in the $SL(r+1)$ case (see \eqref{assume}) that  $\xi_i$ is q-distinct from $\xi_{i+1}$. In particular this means $Z$ is regular semisimple.

First we define the \textit{$QQ$-system} for $\overline{GL}(\infty)$ as infinite generalization of \eqref{eq:QQAtype}:
\begin{equation}\label{eq:QQAtypeInf}
\xi_{i+1} \,Q^+_i(qz)Q^-_i(z)-\xi_i\,  Q^+_i(z)Q^-_i(qz)= \Lambda_i(z)Q^+_{i-1}(qz)Q^+_{i+1}(z)\,, \quad i\in\mathbb{Z}.
\end{equation}
 We say that a polynomial solution $\{ Q_i^+(z),Q_i^-(z)\}$ of \eqref{eq:QQAtypeInf} is nondegenerate if for $i\neq j$ the zeros of $Q^+_i(z)$ and $Q^-_{j}(z)$ are $q$-distinct from each other and from the zeros of $\Lambda_{k}(z)$ for $|i-k|=1,\,|j-k|=1$.

The following Theorem is the the direct analogue of Theorem \eqref{inj}:

\begin{Thm}   
  There is a one-to-one correspondence between the set of
  nondegenerate $Z$-twisted Miura-Pl\"ucker $(\overline{GL}(\infty),q)$-opers and the set
  of nondegenerate polynomial solutions of the $\overline{GL}(\infty)$ $QQ$-system \eqref{eq:QQAtypeInf}. 
\end{Thm}

It can proved the same way as in \cite{Frenkel:2020} with the use of the Proposition \ref{2flagthminf}.  The following theorem serves as an infinite-dimensional generalization of Theorem \ref{th:SLNqMiura}.
\begin{Thm}\label{th:GLinfqMiuraInf}
Let $A(z)$ be as in \eqref{form of Ainf} and $Z=\prod_i\zeta_i^{\check{\alpha}_i}$. Then the $q$-gauge transformation $v(z)$ which diagonalizes q-connection
$$
A(z)=v(qz) Z v(z)^{-1}
$$
reads
\begin{equation}\label{eq:qGaugeGenGLing}
v(z)= \prod\limits_{i=-\infty}^{+\infty} Q^+_i(z)^{\check{\alpha}_i} \cdot \prod\limits_{i=-{\infty}}^{+\infty} V_i(z)\,,
\end{equation}
where 
\begin{equation}
V_i(z)= \exp\left(-\sum_{j> i}\phi_{i,\dots,j}(z)\, e_{i,\dots,j}\right)\,,
\end{equation}
in which $e_{i,\dots,j}=[\dots[[e_i,e_{i+1}],e_{i+2}]\dots e_j]$ 
and functions $\phi_{i,\dots,j}(z)$ satisfy the following relations 
\begin{align}\label{eq:QQinf}
\xi_{i+1} \,\phi_i(z)-\xi_i\,  \phi_i(qz)&= \rho_i(z)\,,\notag\\
\xi_{i+2} \,\phi_{i,i+1}(z)-\xi_{i}\,  \phi_{i, i+1}(qz)&=\rho_{i+1}(z)\phi_i(z)\,,\notag\\
\dots&\dots\\
\xi_{i+j+1}\,\phi_{i,\dots, i+j}(z)-\xi_i\,  \phi_{i,\dots, i+j}(qz)&=\rho_{i+j}(z)\phi_{i,\dots, i+j-1}(z)\,,\notag\\
\dots&\dots\nonumber\\\notag
\end{align}
where $i\in \mathbb{Z},\,j\in\mathbb{Z}_+$ and we use the same notations as in Section \ref{Sec:MiuraSL}.
\end{Thm}

The set of equations \eqref{eq:QQinf} is called \textit{extended $QQ$-system} for $\overline{GL}(\infty)$ which can also be presented as
\begin{align}\label{eq:QQinfQform}
\xi_{i+1} \,Q^+_i(qz)Q^-_i(z)-\xi_i\,  Q^+_i(z)Q^-_i(qz)&= \Lambda_i(z)Q^+_{i-1}(qz)Q^+_{i+1}(z)\,,\notag\\
\xi_{i+2} \,Q^+_{i+1}(qz)Q^-_{i,i+1}(z)-\xi_i\,  Q^+_{i+1}(z)Q^-_{i,i+1}(qz)&= \Lambda_{i+1}(z)Q^-_{i}(qz)Q^+_{i+1}(z)\,,\notag\\
\dots&\dots\\
\xi_{i+j+1} \,Q^+_{i+j}(qz)Q^-_{i,\dots, i+j}(z)-\xi_i\,  Q^+_{i+j}(z)Q^-_{i,\dots, i+j}(qz)&= \Lambda_{i+j}(z)Q^-_{i,\dots, i+j-1}(qz)Q^+_{i+j+1}(z)\,,\notag\\
\dots&\dots\nonumber\\\notag
\end{align}

\begin{proof}
Let us first rewrite the diagonalization condition as
\begin{equation}\label{eq:OperTrivInf}
v(qz)^{-1}A(z)=  Zv(z)^{-1}
\end{equation}
as it will be easier to compute the left and right hand sides of the above equation and then compare them. We can make a statement similar to Lemma \ref{Th:BCHAConnection} and write the $(\overline{GL}(\infty),q)$-oper as 
\begin{equation}\label{new form of A1 inf}
A(z)=\prod_{i=+\infty}^{-\infty} Q^+_i(qz)^{\check{\alpha}_i} \cdot \prod\limits_{i=+\infty}^{-\infty} e^{\frac{\zeta_{i}}{\zeta_{i+1}}\rho_i(z)e_i} \cdot \prod_{i=+\infty}^{-\infty} \zeta_i^{\check{\alpha}}Q^+_i(z)^{-\check{\alpha}_i}\,.
\end{equation}

Then the left hand side of \eqref{eq:OperTrivInf} reads
\begin{equation}\label{eq:Azvzgeninf}
v(qz)^{-1}A(z)=\prod_{i=+\infty}^{-\infty} \exp\left(\sum_{j> i}\phi_{i,\dots,j}(qz)\, e_{i,\dots,j}\right)\cdot \prod^{-\infty}_{i=+\infty} e^{\frac{\zeta_{i}}{\zeta_{i+1}}\rho_i(z)e_i} \cdot  \prod^{-\infty}_{i=+\infty}\zeta_i^{\check{\alpha}_i} Q^+_i(z)^{-\check{\alpha}_i}\,.
\end{equation}
We now need to move $i$th element from the middle product above to the left until they combine with the corresponding $e^{-\phi_i(qz)e_i}$ terms. This way $e^{\frac{\zeta_{i}}{\zeta_{i+1}}\rho_i(z)e_i}$ will need to be carried over to $V_{i+1}(qz)^{-1}$ inside the first product.
Each term $e^{\frac{\zeta_{i}}{\zeta_{i+1}}\rho_i(z)e_i}$ will have nontrivial commutators with exponentials containing $e_{i+1}$ -- the others will vanish due to Serre relations. Thus \eqref{eq:Azvzgeninf} reads 
\begin{align}\label{eq:Azvzgen2}
&v(qz)^{-1}A(z)=\qquad \qquad \dots\dots \notag\\
&\cdot 
\exp\left(\frac{\zeta_{1}}{\zeta_{2}}\rho_1(z)+\phi_1(qz)\right)e_1\cdot \dots\cdot \exp\left(\frac{\zeta_{i}}{\zeta_{i+1}}\rho_r(z)\phi_{1,\dots, i-1}(qz)+\phi_{1,\dots,i}(qz)\right)e_{1,\dots,i}\cdot\dots\\
&\cdot \exp\left(\frac{\zeta_{2}}{\zeta_{3}}\rho_2(z)+\phi_2(qz)\right)e_2\cdot\dots\cdot \exp\left(\frac{\zeta_j}{\zeta_{j+1}}\rho_j(z)\phi_{2,\dots, j-1}(qz)+\phi_{2,\dots, j}(qz)\right)e_{2,\dots,j}\cdot\dots\notag\\
&\cdot \qquad \qquad \dots\dots \notag\\
&\cdot  \prod^{-\infty}_{i=+\infty}\zeta_i^{\check{\alpha}_i} Q^+_i(z)^{-\check{\alpha}_i}\notag\,.
\end{align}

This expression needs to be compared against the right hand side of \eqref{eq:OperTrivInf} which is given by
\begin{equation}
Zv(z)^{-1}=\prod_{i} \zeta^{\check{\alpha}_i}\cdot \prod_{i=+\infty}^{-\infty} \exp\left(\sum_{j> i}\phi_{i,\dots,j}(qz)\, e_{i,\dots,j}\right)\cdot\prod^{-\infty}_{i=+\infty} Q^+_i(z)^{-\check{\alpha}_i}
\end{equation}
In order to make the comparison manifest one needs to move the Cartan terms from the end to the front using the second equation from Lemma \ref{Th:LemmaMV05}
\begin{align}\label{eq:vqzgen2}
Zv(z)^{-1}&= \quad\qquad\qquad\qquad \qquad \dots\dots\notag\\
&\cdot 
\exp\left(\frac{\zeta_1^2}{\zeta_0\zeta_2}\phi_1(z)e_1\right)\cdot \dots\cdot \exp\left(\frac{\zeta_1\zeta_j}{\zeta_0}\phi_{1,\dots,j}(z)e_{1,\dots,j}\right)\dots\\
&\cdot \exp\left(\frac{\zeta_2^2}{\zeta_1\zeta_3}\phi_2(z)e_2\right)\cdot\dots\cdot \exp\left(\frac{\zeta_2\zeta_l}{\zeta_1}\phi_{2,\dots, l}(z)e_{2,\dots,l}\right)\cdot\dots\notag\\
&\cdot \quad\qquad\qquad\qquad \qquad \dots\dots \notag\\
&\cdot\prod^{-\infty}_{i=+\infty}\zeta_i^{\check{\alpha}_i} Q^+_i(z)^{-\check{\alpha}_i}\notag\,.
\end{align}

Here we used the following fact about nested commutators in Chevalley basis: 
\begin{equation}
[\check{\alpha}_a, e_{i,\dots,j}]=
\begin{cases}
e_{i,\dots, j}\,,\quad a=i,\,\, \text{or} \,\,a=j\,,\\
-e_{i,\dots, j}\,,\quad a= i-1,\,\, \text{or} \,\,a=j+1\,,\\
0\,, \quad \text{otherwise}\,.
\end{cases}
\end{equation}
Comparing \eqref{eq:Azvzgen2} with \eqref{eq:vqzgen2} leads to \eqref{eq:QQinf}.
\end{proof}

\begin{Cor}
Theorem \ref{th:SLNqMiura} follows.
\end{Cor}
\begin{proof}
In the proof of Theorem \ref{th:GLinfqMiuraInf} one needs to replace all infinite products with products ranging between $1$ and $r$ and put $\zeta_0=\zeta_{r+1}=1$.
\end{proof}

Analogously to Theorem \ref{Th:BetheQQEquiv} we can make the following statement.

\begin{Thm}\label{Th:BetheQQEquivinf}
\begin{enumerate}
\item Every nondegenerate solution of the $QQ$-system for $GL(\infty)$ \eqref{eq:QQAtypeInf} is also a nondegenerate solution of the extended $QQ$-system for $\overline{GL}(\infty)$ \eqref{eq:QQinfQform}.

\item
There is a one-to-one correspondence between the set of nondegenerate solutions of the 
the $QQ$-system for $GL(\infty)$
the set of solutions of Bethe Ansatz equations for $GL(\infty)$:
\begin{equation}    \label{eq:betheinf}
\frac{Q^+_{i}(qw_k^i)}{Q^+_{i}(q^{-1}w^k_i)} \frac{\xi_i}{\xi_{i+1}}=
- \frac{\Lambda_i(w_k^i) Q^{+}_{i+1}(qw_k^i)Q^{+}_{i-1}(w_k^i)}{\Lambda_i(q^{-1}w_k^i)Q^{+}_{i+1}(w_k^i)Q^{+}_{i-1}(q^{-1}w_k^i)},
\end{equation}
where $i\in\mathbb{Z}; k=1,\ldots,m_i$.
\end{enumerate}
\end{Thm}

\section{Toroidal q-opers}\label{Sec:QuantumToroidal}
\subsection{  Quantum toroidal algebras, Bethe equations and $QQ$-System of $\widehat{A_0}$ type}
The quantum toroidal algebra $U_{t_1,t_2}\big (\widehat{\widehat{\mathfrak{gl}}}(1)\big)$\footnote{Sometimes in the literature different notation $U_{q_1,q_2,q_3}(\widehat{\widehat{\mathfrak{gl}}}(1))$ is used, where $q_1=(t_1t_2)^{-1}$, $q_2=t_2$, $q_1q_2q_3=1$.} attracted a lot of attention in the recent years. On one hand it has explicit geometric realization: there is a natural action of this algebra on equivariant K-theory of ADHM instanton spaces \cite{Schiffmann_2013},\cite{Negut:thesis},\cite{Okounkov:2016sya} which corresponds to simplest framed quiver varieties with one loop and one vertex. 

One can describe such moduli spaces $\mathcal{M}_N$ of rank-$N$ torsion-free sheaves $F$ on $\mathbb{P}^2$ with framing at infinity (also known as the moduli space of $U(N)$ instantons on $\mathbb{R}^4$). The framing condition forces the first Chern class to vanish, however, the second Chern class can range over the non-positive integers $c_2(F) = k$. The moduli space can be represented as a disjoint sum $\mathcal{M}_N = \sqcup_k \mathcal{M}_{k,N}$. Each $ \mathcal{M}_{k,N}$ can be described as the moduli spaces of stable representations of the ADHM quiver below, where $\mathscr{W}$ is a trivial bundle of rank $N$ and $\mathscr{V}$ is a bundle of rank $k$. For $N = 1$ this quiver variety describes Hilbert scheme of $k$ points on $\mathbb{C}^2$. We refer for the details and the equivalence of various descriptions of ADHM moduli spaces to \cite{Nakajima: 1999hilb}.

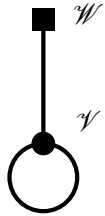
\begin{center}
\begin{tikzpicture}[xscale=1.5, yscale=1.5]
\fill [ultra thick] (1-0.1,1) rectangle (1.1,1.2);
\draw [fill] (1,0) circle [radius=0.1];
\draw [-, ultra thick] (1,1) -- (1,0.1);
\draw[-, ultra thick]
(1,0) arc [start angle=90,end angle=450,radius=.3];
\node (1) at (1.4,1.15) {$\mathscr{W}$};
\node (2) at (1.4,0.2) {$\mathscr{V}$};
\end{tikzpicture}
\end{center}

Let us denote $G=A\times (\mathbb{C}^{\times})^2$, where $A$ be the {\it framing} torus, 
i.e.  maximal torus of $GL(N)$ and the second factor $(\mathbb{C}^{\times})^2$ is the torus acting on $\mathbb{C}^2\subset \mathbb{P}^2$. We denote equivariant parameters corresponding to $A$ and $G/A$ as  $\mathrm{a}_1,\dots, \mathrm{a}_N$ and $t_1,t_2$ correspondingly.

The $G$-equivariant K-theory of $\mathcal{M}_{k,N}$ is generated by the equivariant vector 
bundle $\mathscr{V}$ of rank 
$k$ over $\mathcal{M}_{k,N}$ as in the case of the cotangent bundles to Grassmannians discussed in the introduction. 
The space of the localized $K_{G}(\mathcal{M}_{k,N})$ is a module for 
spanned by the fixed points of $U_{t_1,t_2}(\widehat{\widehat{\mathfrak{gl}}}(1))$.  

This module has the structure of the analogue of XXZ-module for the toroidal algebra. Namely, the physical space 
$$\mathcal{H}=\mathscr{F}(\mathrm{a}_1)\otimes \dots  \otimes \mathscr{F}(\mathrm{a}_N)$$ is the product of Fock space representations of the toroidal algebra $\{\mathscr{F}(\mathrm{a}_i)\}$, where the parameters $\{\mathrm{a}_i\}$,which have the meaning of evaluation parameters, correspond to the zero mode value of the infinite-dimensional Heisenberg algebra. We refer to \cite{Okounkov:2016sya} for more details. 

As we described in the Introduction, the quantum equivariant K-theory based on quasimaps is described 
by difference equations, which coincide with quantum Knizhnik-Zamolodchikov equations and the related dynamical equations. The solutions to these difference equations can be computed as certain Euler characteristics on the moduli spaces of quasimaps. It is given by certain integral formula with the asymptotics given by the Yang-Yang function $Y$, which can be described as follows.

Let $\ell(x)$ be a multi-valued function, which can be written in terms of dilogarithm (see  \cite{Gaiotto:2013bwa}), such that 
$$
\exp 2\pi \frac{\partial \ell(x)}{\partial x} = 2\sinh \pi x\,.
$$
The Yang-Yang function for ADHM space $\mathcal{M}_{k,N}$ is given by \cite{Aganagic:2017be}:
\begin{align}
Y_{ADHM} (\boldsymbol{\sigma},\boldsymbol{\alpha},\epsilon_1, \epsilon_2) &= \sum_{a=1}^k \sum_{m=1}^N \ell (\sigma_a - \alpha_m)+\ell (-\sigma_a + \alpha_m-\epsilon_1-\epsilon_2)\notag\\
&+\sum_{a\neq b}^k\ell (\sigma_a - \sigma_b+\epsilon_1)+ \ell (\sigma_a - \sigma_b+\epsilon_2)+\ell (\sigma_a - \sigma_b-\epsilon_1-\epsilon_2)\notag\\
&-\upkappa \sum_{a=1}^k \sigma_a\,.
\end{align}
where
$$
s_b = e^{2\pi \sigma_b}\,,\quad \mathrm{a}_b = e^{2\pi \alpha_b}\,,\quad t_1= e^{2\pi \epsilon_1}\,,\quad  t_2= e^{2\pi \epsilon_2}\,,\quad \kappa = e^{2\pi \upkappa}\,,\quad \kappa = (t_1 t_2)^{-\frac{N}{2}}\mathfrak{z},
$$
so that $\mathfrak{z}$ is a K\"ahler parameter of $\cM_{k,N}$.

Then Bethe equations in the case of $\mathcal{M}_{k,N}$ can be computed as critical points for $Y_{ADHM}$:
\begin{Lem}
The equations  
\begin{equation}
\exp 2\pi\frac{\partial Y_{ADHM}}{\partial \sigma_a} = 1\,,\qquad a = 1,\dots, k\,.
\end{equation}
are equivalent to the following Bethe equations
\begin{equation}
\prod_{l=1}^N\frac{s_a-\mathrm{a}_l}{t_1t_2s_a-\mathrm{a}_l}\cdot\prod_{\substack{b=1 \\ b\neq a}}^k\frac{s_a-t_1 s_b}{t_1s_a-s_b}\frac{s_a-t_2 s_b}{t_2s_a-s_b}\frac{s_a-t_1 t_2 s_b}{t_1 t_2 s_a- s_b}=\mathfrak{z}\,,\quad a=1,\dots, k\,.
\label{eq:BetheADHM}
\end{equation}
\end{Lem}

Recall that equations \eqref{eq:BetheADHM} describe relations in quantum equivariant K-theory of $\cM_{k,N}$.
Generalizing the results of \cite{Pushkar:2016qvw} with the help of \cite{2016arXiv161201048S}, one can prove the following.
\begin{Prop}\label{eq:QuantMultConj}
The eigenvalues of quantum multiplication operators by bundles $\Lambda^l\mathscr{V}\,,1\leq l\leq k$ in localized quantum $G$-equivariant K-theory of $\mathcal{M}_{N,k}$ are given by elementary symmetric polynomials $e_l(s_1,\dots s_k)$ of Bethe roots which satisfy the following Bethe equations 
\eqref{eq:BetheADHM}.
\end{Prop}

Thus quantum equivariant K-theory ring of $\mathcal{M}_{N,k}$ can be described by the symmetric 
functions of variables $s_1, \dots, s_k$ subject to Bethe equations. We refer to that as Bethe algebra of the XXZ model for quantum toroidal algebra. 

On the other hand, Feigin, Jimbo, Miwa and Mukhin \cite{Feigin_2016,Feigin_2017} studied such XXZ model explicitly and derived such Bethe equations for the corresponding transfer matrices. Another important issue, featured in \cite{Feigin_2017} is the explicit construction  of the $Q$-operator. We remind that this is the operator in the Bethe algebra  whose eigenvalues  form a generating function of the elementary symmetric functions of Bethe roots, i.e. it is a generating function of operators from Proposition \eqref{eq:QuantMultConj}.

 Recently Frenkel and Hernandez \cite{Frenkel:2016}  wrote down the $QQ$-system leading to Bethe ansatz equations for quantum toroidal $\mathfrak{gl}_1$ algebra. It reads as follows: 
\begin{equation}\label{eq:QQFH}
\xi \mathcal{Q}^+\left((t_1t_2)^{-1}z\right) \mathcal{Q}^-(z)- \mathcal{Q}^+(z) \mathcal{Q}^-\left((t_1t_2)^{-1}z\right)=\mathcal{L}(z)\mathcal{Q}^+(t_1^{-1}z) \mathcal{Q}^+(t_2^{-1}z)\,,
\end{equation}
where we altered the authors' notation slightly and introduced `framing polynomial' $\mathcal{L}(z)$.
We will refer to this functional equation as $\hat{A}_0$ $QQ$-system. 
The above $QQ$-system equations can also be rewritten as
\begin{equation}\label{eq:QQADHM}
\xi \boldsymbol{\phi}(z)- \boldsymbol{\phi}((t_1t_2)^{-1}z) = \boldsymbol{\rho}(z)\,,
\end{equation}
where
\begin{equation}\label{eq:DefQQADHM}
\boldsymbol{\phi}(z)=\frac{\mathcal{Q}^-(z)}{\mathcal{Q}^+((t_1t_2)^{-1}z)}\,,\qquad \boldsymbol{\rho}(z)=\mathcal{L}(z)\frac{\mathcal{Q}^+(t_1^{-1}z) \mathcal{Q}^+(t_1^{-2}z)}{\mathcal{Q}^+(z)\mathcal{Q}^+((t_1t_2)^{-1}z)}\,.
\end{equation}
We call the solutions for such system {\it nondegnerate} if $\mathcal{Q}^+(z) $, $\mathcal{Q}^-(z) $, $\mathcal{L}$ are $t_1, t_2$-distinct and $\xi\neq 1$. 

That leads to the following Lemma.

\begin{Lem}
There is a one-to-one correspondence between the set of nondegenerate solutions of \eqref{eq:QQFH} and \eqref{eq:BetheADHM}.
\end{Lem}

\begin{proof}
Since
$$
\mathcal{Q}(u)=\prod_{a=1}^k(z-s_a)\,,\qquad \mathcal{L}(u)=\prod_{i=1}^N(z-{\rm a}_i)
$$
we can first evaluate \eqref{eq:QQglN} at $u=s_a$, then shift variable $u$ by $t_1 t_2$ and evaluate this equation again at $z=s_a$. This leads us to the following
$$
\frac{\cL(s_a)}{\cL(t_1 t_2 s_a)}\cdot \frac{\mathcal{Q}(t_1^{-1}s_a)}{\mathcal{Q}(t_1 s_a)}\frac{\mathcal{Q}(t_2^{-1} s_a)}{\mathcal{Q}(t_2 s_a)}\frac{\mathcal{Q}(t_1 t_2 s_a)}{\mathcal{Q}((t_1 t_2)^{-1} s_a)}  = -\xi\,.
$$
This shows the implication in one direction -- from the $QQ$-system to Bethe equations. The opposite statement can be proved analogously to 
Theorem \ref{Th:BetheQQEquiv}. We leave it to the reader.
\end{proof}

There exists a generalization of this construction to higher rank quantum toroidal algebra $U_{t_1,t_2}(\widehat{\widehat{\mathfrak{gl}}}(N))$ for cyclic quiver varieties with $N$ vertices. It is easy to write the Yang-Yang function in this case as well as Bethe equations (see \cite{Aganagic:2017be} for universal treatment). It is also easy to present the analogue of the $QQ$-system (see below). 
However, the representation-theoretic approach along the lines of \cite{Frenkel:2016,Feigin_2017} has not yet been developed, i.e. construction of the $Q$-operator as a transfer matrix for special auxiliary representation of $U_{t_1,t_2}(\widehat{\widehat{\mathfrak{gl}}}(N))$.

\subsection{Miura  1-toroidal  q-opers}
We can now define toroidal opers. Let us consider the automorphism of the Dynkin diagram of $\mathfrak{a}_{\infty}$, which correspond to the shift by one vertex. This automorphism can be realized by the 
transformation corresponding to the conjugation via the infinite Coxeter element $\prod^{-\infty}_{i=+\infty} s_i$. In the matrix notation such infinite Coxeter element can be realized via $\mathcal{V}_1=\sum_{i\in\mathbb{Z}} E_{i,i-1}$.

\begin{Def}
Let $p,\xi \in\mathbb{C}^\times$.
We refer to a Z-twisted Miura $(\overline{GL}(\infty),q)$-oper \eqref{Lambda} satisfying
\begin{equation}\label{eq:V1GaugeA}
\mathcal{V}_1 A(z) \mathcal{V}_1^{-1} = \xi A(pz)\,,
\end{equation}
as the {\it $Z$-twisted 1-toroidal Miura q-oper}. We call it {\it nondegnerate} if it is nondegenerate as Z-twisted Miura $(\overline{GL}(\infty),q)$-oper.
\end{Def}

The above definition \eqref{eq:V1GaugeA} translates to the following conditions on polynomials which appear in the QQ-system 
\begin{equation}
\frac{g_{i+1}(z)}{g_{i}(z)}=\xi \frac{g_{i}(pz)}{g_{i-1}(pz)}\,, \qquad \Lambda_{i+1}(z)=\xi \Lambda_i ( p z)\,,
\end{equation}
The first equation above becomes (recall that $\xi_i=\frac{\zeta_{i}}{\zeta_{i-1}}$)
\begin{equation}
\xi_{i+1}\frac{Q^+_{i+1}(qz)Q^+_{i}(z)}{Q^+_{i+1}(z)Q^+_{i}(qz)}=\xi\, \xi_{i}\,\frac{Q^+_{i}(qz)Q^+_{i-1}(pz)}{Q^+_{i}(z)Q^+_{i-1}(pqz)}\,,
\end{equation}
which can be satisfied provided that for
\begin{equation}\label{eq:xicqQ}
\xi_i = \xi^i\,,\qquad Q^+_i(z)=\mathcal{Q}^+(p^iz)\,,\qquad \Lambda_i(z) = \xi^i \Lambda(p^i z)\,.
\end{equation}

Let us now study how the above periodic conditions affect the $QQ$-system for $GL(\infty)$.
The $QQ$-equations \eqref{eq:QQAtypeInf} can be rewritten in the following way
\begin{equation}
\xi^{i+1} \boldsymbol{\phi}(p^i z)- \xi^i \boldsymbol{\phi}(p^iqz) = \xi^i\Lambda (p^iz)\frac{\mathcal{Q}^+(p^{i-1}z)\mathcal{Q}^+(p^{i+1}qz)}{\mathcal{Q}^+(p^iz)\mathcal{Q}^+(p^iqz)} \,, 
\end{equation}
where we defined $\phi_i(z)$ is replaced by $\boldsymbol{\phi}(p^i z)$ and thusly 
\begin{equation}\label{eq:cQm}
Q^-_i(z)=\mathcal{Q}^-(p^iz)\,.
\end{equation}
By shifting the variable $z\mapsto p^{-i} z$, we get
\begin{equation}
\xi\boldsymbol{\phi}(z)-  \boldsymbol{\phi}(qz) = \Lambda (z)\frac{\mathcal{Q}^+(p^{-1}z)\mathcal{Q}^+(pqz)}{\mathcal{Q}^+(z)\mathcal{Q}^+(qz)} \,, 
\end{equation}
Equivalently we can impose $\rho_i(z)=\boldsymbol{\rho}(p^iz)$.

Notice that the above equations coincide with \eqref{eq:QQADHM} and \eqref{eq:DefQQADHM}; therefore we recover the $\widehat{A}_0$ $QQ$-system provided that
\begin{equation}
\mathcal{L}(z)= \Lambda (z)\,, \qquad p = t_1,\qquad q = (t_1 t_2)^{-1}\,.
\end{equation}

This brings us to the following
\begin{Thm} \label{tort1}
The space of nondegenerate $Z$-twisted Miura 1-toroidal q-opers with regular singularities at $a_1,\dots,a_N$ is isomorphic to the space of solutions of the nondegenerate $\widehat{A}_0$ $QQ$-system \eqref{eq:DefQQADHM} or equivalently to
\begin{equation}\label{eq:QQgl1}
\xi\mathcal{Q}^+(q z)\mathcal{Q}^-( z)- \mathcal{Q}^+(z)\mathcal{Q}^-( qz) = \Lambda(z)\mathcal{Q}^+(p^{-1}z)\mathcal{Q}^+(pq z)\,.
\end{equation}
\end{Thm}

The full set of equations for the extended $QQ$-system for $\overline{GL}(\infty)$ \eqref{eq:QQinfQform} reads
\begin{align}\label{eq:QQAllADHM}
\cQ^+(q z) \cQ^-(z) - \xi \cQ^+(z) \cQ^-(qz) &=\Lambda (z) \cQ^+(p^{-1}qz)\cQ^+(pz)\,,\notag\\
\cQ^+(q z) \cQ^-_{(1)}(z) - \xi^2 \cQ^+(z) \cQ^-_{(1)}(qz) &= \xi\Lambda (z) \cQ^-(p^{-1}qz)\cQ^+(pz)\,,\notag\\
\dots&\dots\\
\cQ^+(q z) \cQ^-_{(j)}(z) - \xi^{j+1} \cQ^+(z) \cQ^-_{(1)}(qz) &= \xi^{j}\Lambda (z) \cQ^-_{(j-1)}(p^{-1}qz)\cQ^+(p z)\,,\notag\\
\dots&\dots\notag
\end{align}
where 
\begin{equation}\label{eq:cQmm}
\cQ^-_{(j)}(p^i z)=Q^-_{i-j+1,i-j+2,\dots, i}(z).
\end{equation}

The gauge transformation which brings toroidal q-oper to the diagonal form can be directly generalized from \eqref{eq:vinverse} using \eqref{eq:xicqQ}, \eqref{eq:cQm} and \eqref{eq:cQmm}:
\begin{equation}\label{eq:vinverseInf}
v(z)^{-1}=\displaystyle\begin{pmatrix}
\ddots & \vdots & \vdots & \dots & \dots  &\vdots  & \ddots \\
0 & \displaystyle\frac{\cQ^+(p^{i-1}z)}{\cQ^+(p^iz)}  & \displaystyle\frac{\cQ^-(p^{i-1}z)}{\cQ^+(p^{i+1}z)} & \dots& \dots & \displaystyle\frac{\cQ^-_{(j)}(p^{i-1}z)}{\cQ^+(p^{i+j}z)}  & \vdots\\
0 & 0 & \displaystyle\frac{\cQ^+(p^{i}z)}{\cQ^+(p^{i+1}z)}  & \dots & \dots & \displaystyle\frac{\cQ^-_{(j-1)}(p^iz)}{\cQ^+(p^{i+j}z)} & \vdots\\
\vdots & \vdots & \vdots & \ddots & \vdots &\vdots  & \vdots\\
0 & \dots & \dots& \dots &\displaystyle\frac{\cQ^+(p^{i+j-2}z)}{\cQ^+(p^{i+j-1}z)} &  \displaystyle\frac{\cQ^-(p^{i+j-2}z)}{\cQ^+(p^{i+j}z)} & \vdots \\
0 & \dots & \dots& \dots & \dots &  \displaystyle\frac{\cQ^+(p^{i+j-1}z)}{\cQ^+(p^{i+j}z)} & \vdots \\
0 & \dots & \dots& \dots & \dots &  0 & \ddots
\end{pmatrix}\,.
\end{equation}

\vskip.1in
We would like to mention that a similar folding procedure on the level of TQ-systems for $SL(\infty)$ was performed by  Hernandez in \cite{Hernndez2010TheA}. For further developments and applications of the QQ-system \eqref{eq:QQgl1} see \cite{Koroteev:2021}.

\subsection{Miura $\mathscr{N}$-toroidal q-Opers}
Here we briefly show, how the above construction can be immediately generalized to higher rank. Namely, one has to generalize the periodicity conditions we 
\begin{Def}
Let $p,\xi\in\mathbb{C}^\times$.
The $Z$-twisted N-toroidal Miura q-oper is a $(GL(\infty),q)$-oper \eqref{Lambda} satisfying
\begin{equation}\label{eq:VNGaugeA}
\mathcal{V}_{\mathscr{N}} A(z) \mathcal{V}_{\mathscr{N}}^{-1} = \xi^{\mathscr{N}}A(p^{\mathscr{N}}z)\,,
\end{equation}
where $\mathcal{V}_{N}=\mathcal{V}^{\mathscr{N}}_1$. We call it {\it nondegnerate} if it is nondegenerate as Z-twisted Miura $(\overline{GL}(\infty),q)$-oper.
\end{Def}

If we impose \eqref{eq:VNGaugeA} on the q-connection we get the following family of equations for $i\geq j$:
\begin{equation}
\frac{g_{i+\mathscr{N}}(z)}{g_{j+\mathscr{N}}(z)}=\xi^{\mathscr{N}}\frac{g_i(p^{\mathscr{N}}z)}{g_j(p^{\mathscr{N}}z)}\,,\qquad \Lambda_{i+\mathscr{N}}(z)= \xi^{\mathscr{N}} \Lambda_i(p^{\mathscr{N}} z)\,,
\end{equation}
which imposes $\mathscr{N}$-periodicity on all functions
$$
Q^\pm_{i+\mathscr{N}}(z)=Q^\pm_{i}(p^\mathscr{N} z)\,,\qquad \xi_{i+\mathscr{N}} = \xi^{\mathscr{N}}\xi_i\,.
$$
for all $i$.

Thus we arrive to the generalization of the Theorem \ref{tort1}:
\begin{Thm}
The nondegenerate $Z$-twisted Miura N-toroidal $q$-opers with regular singularities given by $\Lambda_i(u)=\prod_{j=1}^{\mathscr{N}}(z-a_j^{(i)})$  is in one-to-one correspondence with the nondegenerate solutions of the following  $\widehat{A}_{\mathscr{N}-1}$ $QQ$-system:
\begin{align}\label{eq:QQglN}
\xi_{1} Q^+_1(q z) Q^-_1(z) - \xi_{2} Q^+_1(z) Q^-_1(qz) &=\Lambda_1 (z) Q^+_{\mathscr{N}}(qz)Q^+_{2}(z)\,,\\
\xi_{i} Q^+_i(q z) Q^-_i(z) - \xi_{i+1} Q^+_i(z) Q^-_i(qz) &=\Lambda_i (z) Q^+_{i-1}(qz)Q^+_{i+1}(z)\,,\qquad i = 2,\dots, \mathscr{N}-1\notag\\
\xi_{N} Q^+_N(q z) Q^-_N(z) - \xi_{1} Q^+_{\mathscr{N}}(z) Q^-_{\mathscr{N}}(qz) &=\Lambda_{\mathscr{N}} (z) Q^+_{\mathscr{N}-1}(qz)Q^+_{1}(z)\,.\notag
\end{align}
with the nondegeneracy conditions induced from original $\overline{GL}(\infty)$ $QQ$-system.
\end{Thm}

\bibliography{cpn1}
\end{document}